\documentclass[11pt]{article}
\usepackage{fullpage}
\usepackage{amssymb,latexsym,amsmath,amsthm}     
\usepackage{authblk}
\usepackage{bbm}
\usepackage{eqnarray}
\usepackage[utf8]{inputenc}
\usepackage{graphicx}
\usepackage{enumerate}
\usepackage[english]{babel}
\usepackage{color}
\usepackage[dvipsnames]{xcolor}
\usepackage[toc,page]{appendix}
\usepackage{scalerel,stackengine}
\usepackage{titling}
\usepackage{hyperref}
\hypersetup{
    colorlinks=true,
    linkcolor=blue,
    filecolor=blue,      
    urlcolor=blue,
    citecolor=blue,
    linktocpage=true
}

\newtheorem{theorem}{Theorem}[section] 
\newtheorem{proposition}{Proposition}[section] 
\newtheorem*{proposition*}{Proposition}
\newtheorem{definition}{Definition}[section]
\newtheorem*{definition*}{Definition}
\newtheorem{lemma}{Lemma}[section]
\newtheorem*{lemma*}{Lemma}
\newtheorem{corollary}{Corollary}[section]
\newtheorem{remark}{Remark}[section]
\newtheorem{example}{Example}[section]

\newcommand{\expect}[1]{\left\langle {#1} \right\rangle}
\newcommand{\e}{_{\varepsilon}}
\newcommand{\brac}[1]{\left({#1}\right) }
\newcommand{\cb}[1]{\left\lbrace {#1} \right\rbrace}
\newcommand\norm[1]{\left\lVert#1\right\rVert}

\makeatletter
\newcommand{\setword}[2]{%
  \phantomsection
  #1\def\@currentlabel{\unexpanded{#1}}\label{#2}%
}
\makeatother

\newcommand{\twopartdef}[4]
{
	\left\{
		\begin{array}{ll}
			#1 & \mbox{if } #2 \\
			#3 & \mbox{if } #4
		\end{array}
	\right.
}

\def\Xint#1{\mathchoice
{\XXint\displaystyle\textstyle{#1}}%
{\XXint\textstyle\scriptstyle{#1}}%
{\XXint\scriptstyle\scriptscriptstyle{#1}}%
{\XXint\scriptscriptstyle\scriptscriptstyle{#1}}%
\!\int}
\def\XXint#1#2#3{{\setbox0=\hbox{$#1{#2#3}{\int}$ }
\vcenter{\hbox{$#2#3$ }}\kern-.6\wd0}}

\def\dashint{\Xint-}

\title{Stochastic unfolding and homogenization of spring network models\thanks{This is a pre-print of an article published in Multiscale Modeling \& Simulation in 2018. The final authenticated version is available online at: https://doi.org/10.1137/17M1141230}}
\author{Stefan Neukamm\thanks{stefan.neukamm@tu-dresden.de} }
\author{Mario Varga\thanks{mario.varga@tu-dresden.de}}
\affil{Faculty of Mathematics, Technische Universit\"at Dresden}
\date{28.02.2018}

\begin{document}

\maketitle

\abstract{
The aim of our work is to provide a simple homogenization and discrete-to-continuum procedure for energy driven problems involving stochastic rapidly oscillating coefficients. Our intention is to extend the periodic unfolding method to the stochastic setting. Specifically, we recast the notion of stochastic two-scale convergence in the mean by introducing an appropriate stochastic unfolding operator. This operator admits similar properties as the periodic unfolding operator, leading to an uncomplicated method for stochastic homogenization. Second, we analyze the discrete-to-continuum (resp., stochastic homogenization) limit for a rate-independent system describing a network of linear elasto-plastic springs with random coefficients.
\medskip

}

\setlength{\parindent}{0pt}
\tableofcontents
\newpage

\section{Introduction}
The motivation for this paper is twofold: First, we introduce the method of \textit{stochastic unfolding} as an analogue to the periodic unfolding method, which in recent years has been successfully applied in the analysis and modeling of multiscale problems with periodic microstructure. Our intention is to provide an easily accessible method for stochastic homogenization and discrete-to-continuum analysis that enjoys many parallels to periodic homogenization via unfolding. Second, we analyze the macroscopic behavior (based on stochastic unfolding) of a rate-independent system describing a network of elasto-plastic springs with random coefficients. 
Our result derives (via stochastic homogenization) a continuum evolutionary rate-independent (linear) plasticity system starting from a discrete model. Discrete spring networks depict solid media as collections of material points that interact via one-dimensional elements with certain constitutive laws. They are widely used in material science and the mechanical engineering community.  On the one hand, they are used to model materials with an intrinsic discreteness (on a scale larger than the atomistic scale), such as granular media, truss-like structures, and composites. On the other hand, spring network models are used as a numerical approximation scheme for continuum models. We refer to \cite{ostoja2002lattice,jagota1994spring,hahn2010discrete} and the references therein. In this introduction we first give a brief overview of the stochastic unfolding method that we develop in this paper, and then discuss the discrete-to-continuum limits of random spring networks.
\smallskip

In order to give a brief overview of the stochastic unfolding, let us consider for a moment a prototypical example of convex homogenization: Let $O\subset \mathbb{R}^d$ be open and bounded, $p\in(1,\infty)$, $\varepsilon>0$, and consider the minimization problem
\begin{equation}\label{eq1}
\min_{u\in W^{1,p}(O)}\int_{O}V\e \brac{x,\nabla u(x)}dx \qquad\text{(subject to suitable boundary conditions)}.
\end{equation}
Above $V\e(x,F)$ denotes a family of energy densities that are convex in $F$, and which we assume to rapidly oscillate in $x$ on a scale $\varepsilon$. The objective of homogenization is to derive a simpler minimization problem, say
\begin{equation*}
  \min_{u\in W^{1,p}(O)}\int_{O}V_{0}\brac{\nabla u(x)}dx \qquad\text{(subject to suitable boundary conditions)},
\end{equation*}
with an effective (and simpler)  energy density $V_{0}$ that captures the behavior of (\ref{eq1}) for small $\varepsilon$. This is done by an asymptotic analysis for $\varepsilon\rightarrow 0$, and a classical way to approach this type of problems is based on two-scale convergence methods. The notion of (periodic) two-scale convergence was introduced and developed by Nguetseng \cite{nguetseng1989general} and Allaire \cite{allaire1992homogenization} (see also \cite{lukkassen2002two}). A sequence $u\e \in L^p(O)$ is said to two-scale converge to $u\in L^p(O \times \Box)$ if 
\begin{align*}
\int_{O}u\e(x)\varphi\brac{x,\frac{x}{\varepsilon}}dx \to \int_O\int_{\Box}u(x,y)\varphi(x,y)dy dx \quad \text{as }\varepsilon\to 0,
\end{align*}
for all $\varphi\in L^q(O,C_{\mathsf{per}}(\Box))$. Here, $\Box:=[-\frac{1}{2},\frac{1}{2})^d$ is the unit box, and $C_{\mathsf{per}}(\Box)$ is the space of continuous, $\Box$-periodic functions. The two-scale limit of a sequence refines its weak limit by capturing oscillations on a prescribed scale $\varepsilon$. It is therefore especially useful in homogenization problems involving linear (or monotone) operators and convex potentials with \textit{periodic coefficients}. With regard to problem (\ref{eq1}), two-scale convergence methods apply, e.g. if $V_{\varepsilon}(x,F)=V(x,\frac{x}{\varepsilon},F)$ with $V$ being periodic in its second component and sufficiently regular (e.g. continuous) in its first component.

In \cite{cioranescu2002periodic} the method of periodic unfolding was introduced based on the dilation technique \cite{arbogast1990derivation}. The idea of unfolding, which is closely related to two-scale convergence, is to introduce an operator (the unfolding operator), which embeds sequences of oscillating functions into a larger two-scale space, with the effect that two-scale convergence can be characterized by the usual weak convergence in the two-scale space. In some cases, this method facilitates a more straightforward, and operator theory flavored, analysis of periodic homogenization problems. In recent years periodic unfolding has been applied to a large variety of multiscale problems; e.g., see \cite{cioranescu2004homogenization,griso2004error,
mielke2007two,neukamm2010homogenization,mielke2014two,
ptashnyk2015locally,cazeaux2015homogenization,liero2015homogenization,piatnitski2017homogenization,hanke2017phase}. For a systematic investigation of two-scale calculus associated with the use of the periodic unfolding method we refer to \cite{cioranescu2008periodic,visintin2004some,visintin2006towards,mielke2007two,cioranescu2012periodic}.

Motivated by periodic two-scale convergence, in \cite{bourgeat1994stochastic} a related notion of \textit{stochastic two-scale convergence in the mean} was introduced. It is tailor-made for the study of stochastic homogenization problems. In particular, it applies to a stochastic version of problem (\ref{eq1}): Let $\Omega$ be a probability space with a corresponding measure-preserving dynamical system $\cb{T_x}_{x\in\mathbb{R}^d}$ (see Section \ref{section:297}). In the context of stochastic homogenization, we might view $\Omega$ as a configuration space, and $\expect{\cdot}$ (the associated expected value) as an ensemble average w.r.t. configurations. Then we might consider a stochastic version of \eqref{eq1}, namely
\begin{equation}\label{eq2}
\min_{u\in L^p(\Omega)\otimes W^{1,p}(O)} \expect{\int_{O}V(T_{\frac{x}{\varepsilon}}\omega,\nabla_x u(\omega,x))dx},
\end{equation}
where the potential $V(\omega,F)$ is parametrized  by $\omega\in\Omega$, and thus minimizers of \eqref{eq2} are random fields, i.e.,  they depend on $\omega\in\Omega$. The random potential $V\e(x,\cdot)=V(T_{\frac{x}{\varepsilon}}\omega,\cdot)$ in \eqref{eq2} is rapidly oscillating, and its statistics is homogeneous in space (i.e.,  for any finite number of points $x_1,...,x_m\in\mathbb{R}^d$ and all $F\in \mathbb{R}^d$, the joint distribution of $\{V(T_{x_i+z}\cdot,F)\}_{i=1,\ldots,m}$ is independent of the shift $z\in\mathbb{R}^d$). In contrast to periodic two-scale convergence, stochastic two-scale convergence requires test-functions defined not only on the physical space $O\subset\mathbb{R}^d$, but also on the probability space $\Omega$ (see Remark \ref{remark1}, where we recall the definition of stochastic two-scale convergence in the mean from \cite{bourgeat1994stochastic,andrews1998stochastic} in a discrete version).
\smallskip

\textbf{Stochastic unfolding.} In this paper we introduce a \textit{stochastic unfolding} \textit{method}, that (analogously to the periodic case) allows to characterize stochastic two-scale convergence in the mean by mere weak convergence in an extended space. Having discrete-to-continuum problems in mind, we concentrate in this paper on a discrete setting: For example, in (\ref{eq2}) $O$ is replaced by the discrete set $O\e:=O\cap \varepsilon\mathbb{Z}^{d}$ (equipped with a rescaled counting measure $m_{\varepsilon}$), and instead of the gradient we consider difference quotients (see Section \ref{Physical_space} for the specific discrete setting). As we shall demonstrate, the \textit{stochastic unfolding method} features many analogies to the periodic case; as a consequence it allows us to lift systematically and easily homogenization results and multiscale models for periodic media to the level of random media. In the following, in particular for readers familiar with periodic unfolding, we briefly summarize the main properties of stochastic unfolding and its analogies to the periodic case:
\begin{itemize}
\item We introduce an operator $\mathcal{T}_{\varepsilon}: L^p(\Omega\times \varepsilon\mathbb{Z}^{d})\rightarrow L^p(\Omega\times\mathbb{R}^d)$ which is a linear isometry, and we call it a stochastic unfolding operator (see Section \ref{section31}).
\item Two-scale convergence in the mean for $u\e\in L^p(\Omega\times\varepsilon\mathbb{Z}^{d})$ reduces to weak convergence of the unfolding $\mathcal{T}_{\varepsilon} u\e$ in $L^p(\Omega\times \mathbb{R}^d)$ (see Remark \ref{remark1}).
\item We define weak (strong) stochastic two-scale convergence as weak (strong) convergence of the unfolded sequence $\mathcal{T}_{\varepsilon} u\e$ (see Definition \ref{definitionUnf}).
\item (Compactness). Bounded sequences converge (up to a subsequence) in the weak stochastic two-scale sense (see Lemma \ref{basicfacts}).
\item (Compactness for gradients). If $u\e\in L^p(\Omega\times\varepsilon\mathbb{Z}^{d})$ is bounded and its (discrete) gradient is bounded, then (up to extracting a subsequence) $u\e$ weakly two-scale converges to $U_{0} \in L^p_{\mathsf{inv}}(\Omega)\otimes W^{1,p}(\mathbb{R}^d)$. Moreover, its gradient weakly two-scale converges, and the limit has a specific structure: $\nabla U_{0}+\chi$, where $\chi\in \mathbf{L}^p_{\mathsf{pot}}(\Omega)\otimes L^p(\mathbb{R}^d)$. Here, $\mathbf{L}^p_{\mathsf{pot}}(\Omega):=\overline{\left\lbrace D\varphi: \varphi \in L^p(\Omega) \right\rbrace}$, where $D$ denotes the \textit{horizontal derivative} for random variables and the closure is taken in $L^p(\Omega)$, and $L^p_{\mathsf{inv}}(\Omega)$ is the space of shift-invariant functions (see Section \ref{section:297} for precise definitions). In the ergodic case, the latter reduces to the space of constant functions, and thus the two-scale limit $U_0$ is deterministic, i.e.,  it does not depend on $\omega\in\Omega$.
The general structure of this compactness statement for gradients is quite similar to its analogon in the periodic case, which we briefly recall: Up to a subsequence, a bounded sequence $u\e\in W^{1,p}(\mathbb{R}^d)$ converges weakly to some $u_0\in W^{1,p}(\mathbb{R}^d)$, and the gradient $\nabla u\e$ weakly two-scale converges to $\nabla u_0(x)+v(x,y)$, where $v\in L^p(\mathbb{R}^d)\otimes L^p_{\mathsf{pot},\mathsf{per}}(\Box)$ with $\Box$ denoting the reference cell of periodicity (e.g. $\Box=[-\frac{1}{2},\frac{1}{2})^d$). In the periodic case, thanks to Poincar\'e's inequality on $\Box$, any periodic, conservative field $v\in L^p_{\mathsf{pot},\mathsf{per}}(\Box)$, can be represented as $v=\nabla_y\varphi$ for some potential $\varphi\in W^{1,p}_{\mathsf{per}}(\Box)$. Thus, the weak two-scale limit of $\nabla u\e$ takes the form $\nabla u_0(x)+\nabla_y\varphi(x,y)$ with $\varphi\in L^p(\mathbb{R}^d,W^{1,p}_{\mathsf{per}}(\Box))$. 

In the stochastic case, typically it is not possible to represent $v\in \mathbf{L}^p_{\mathsf{pot}}(\Omega)$ with the help of a potential defined on $\Omega$. This is one of the main differences between stochastic and periodic homogenization. %
\item (Recovery sequences). For $U_0$ and $\chi$ as above, we construct a sequence $u\e\in L^p(\Omega\times \varepsilon\mathbb{Z}^{d})$ which satisfies
\begin{equation*}
\mathcal{T}_{\varepsilon} u_{\varepsilon} \to U_0, \quad \mathcal{T}_{\varepsilon}\nabla^{\varepsilon} u_{\varepsilon} \to \nabla U_0 + \chi \quad \text{strongly in }L^p(\Omega\times \mathbb{R}^d)
\end{equation*}
(see Corollary \ref{rem5}).
\item The following transformation formula holds:
\begin{equation*}
\expect{\int_{\varepsilon\mathbb{Z}^{d}}V(T_{\frac{x}{\varepsilon}}\omega,v(\omega,x))dm\e(x)}=\expect{\int_{\mathbb{R}^d}V(\omega,\mathcal{T}_{\varepsilon} v(\omega,x))dx}.
\end{equation*}
Using this formula and the previous properties, the $\Gamma$-convergence analysis of the discrete version of (\ref{eq2}) becomes straightforward, and relies (as the only noteworthy additional ingredient) on the weak lower-semi\-continuity of convex integral functionals; see Proposition \ref{pro67} and Theorem \ref{gammatheorem}.
\end{itemize}
We would like to remark that some of the statements above (in particular those that involve only \textit{weak} two-scale convergence) have already been established in the continuum setting in \cite{bourgeat1994stochastic}. However, the arguments that we present have a different twist, since they are based on the unfolding operator and apply to the discrete setting.

In contrast to periodic (deterministic) two-scale convergence, in the stochastic case different meaningful notions for two-scale convergence exist, since one may ask for convergence in the $L^p(\Omega)$-sense ($\Omega$ being the probability space) or in a quenched sense, i.e.,  for a.e. $\omega\in \Omega$. The former corresponds to stochastic two-scale convergence in the mean and the notion of stochastic unfolding that we introduce in this paper. The latter corresponds to a finer notion of (quenched) stochastic two-scale convergence, as considered in \cite{zhikov2006homogenization,heida2011extension}, we also refer to \cite{mathieu2007quenched,faggionato2008random} for a discrete version of (quenched) two-scale convergence. 
\smallskip

\textbf{Discrete-to-continuum limits of random spring networks.} In the second part of this paper, we study the macroscopic, rate-independent behavior of periodic networks formed of elasto-plastic springs with random material properties. In the following, we briefly summarize our result in the simplest (nontrivial) two-dimensional setting. In Section \ref{section:4} we shall treat a general, multidimensional case.
To explain the model, we first consider a single spring that in a natural state has endpoints $x_0,x_1\in\mathbb{R}^d$, and thus is aligned with $b:=x_1-x_0$. We describe a deformation of the spring with the help of a displacement function $v$ that maps an endpoint $x_i$ to its new position $x_i+v(x_i)$. As the measure of relative elongation (resp. compression) of the spring, we consider the Cauchy strain $\frac{|(b+|b|\partial_b v)|-|b|}{|b|}$ with $\partial_b v:=\frac{v(x_0+b)-v(x_0)}{|b|}$. If the displacement is (infinitesimally) small $v=\delta u$ with $0<\delta \ll 1$ and $u: \cb{x_0,x_1} \to \mathbb{R}^d$, we arrive (by rescaling the strain $\frac{1}{\delta}\frac{|(b+|b|\partial_b v)|-|b|}{|b|}$ and passing to the limit $\delta \to 0$) at the \textit{linearized} strain $\frac{b}{|b|}\cdot \partial_b u$. As is usual in linear elasto-plasticity (see e.g. \cite[Section 3]{han2012plasticity}), we assume that the linearized strain admits an additive decomposition $\frac{b}{|b|}\cdot \partial_b u=e+z$, where $e$ and $z$ are its elastic and plastic parts, respectively. The force (its intensity) exerted by the spring is linear in the elastic strain: $\sigma=ae$, $a>0$ being the spring constant. We define a free energy (describing materials with linear kinematic hardening)
\begin{equation*}
\mathcal E_b(u,z):=\frac{1}{2}a \brac{\frac{b}{|b|}\cdot \partial_b u-z}^2+\frac{1}{2}h z^2, 
\end{equation*}
where $h>0$ denotes a hardening parameter.
The rate-independent evolution of the elasto-plastic spring under a loading $l:[0,T]\times \cb{x_0,x_1}\rightarrow \mathbb{R}^d$ is determined by
\begin{align}\label{ev11}
\begin{split}
& (-1)^{i}\brac{\frac{b}{|b|}\cdot \partial_b u(t)-z(t)}\frac{b}{|b|}+l(t,x_i)=0,\; i=1,2,\\
& \dot{z}(t)\in \partial I_{[-\sigma_y,\sigma_y]}\brac{-\frac{\partial \mathcal{E}_b}{\partial z}(u(t),z(t))}.
\end{split}
\end{align}
In (\ref{ev11}), $\sigma_y\geq 0$ is the yield stress of the spring, $\partial I_{[-\sigma_y,\sigma_y]}$ denotes the convex subdifferential of $I_{[-\sigma_y,\sigma_y]}$, which is the indicator function of the set $[-\sigma_y,\sigma_y]$. Note that the first two equations are force balance equations (inertial terms are disregarded), reasonable in regimes of small displacements, and the second expression is a flow rule for the variable $z$. 

We consider a network of springs $E=\{ e=[x,x+\varepsilon b]\,:\,x \in\varepsilon\mathbb{Z}^{2},\, b\in\{e_1,e_2,e_1+e_2\} \}$, where the nodes $x\in \varepsilon\mathbb{Z}^{2}$ represent the reference configuration of particles connected by springs. The displacement of the network is described with help of a map $u:\varepsilon\mathbb{Z}^{2} \to \mathbb{R}^2$, and the plastic strains of the springs are accounted by an internal variable $z: \varepsilon\mathbb{Z}^{2} \to \mathbb{R}^3$ (e.g. $z_1(x)$ is the plastic strain of the spring $[x,x+\varepsilon e_1]$). We assume that the particles outside of a set $O\e=O\cap \varepsilon\mathbb{Z}^{d}$ are fixed, i.e.,  $u=0$ in $\varepsilon\mathbb{Z}^{d} \setminus O\e$; furthermore, we suppose that $z$ is supported in $O\e^+:=\cb{x\in \varepsilon\mathbb{Z}^{2}:(x,x+\varepsilon b)\cap O \neq \emptyset \text{ for some }b\in \cb{e_1,e_2,e_1+e_2}}$. A small external force $\varepsilon  l\e :[0,T]\times O\e \to \mathbb{R}^2$ acts on the system. According to the evolution law (\ref{ev11}) for single springs, the evolution of the network is determined by (for $t\in [0,T]$)
\begin{align*}
& \sum_{b\in \cb{e_1,e_2,e_1+e_2}} -\partial_{-\varepsilon b}\brac{|b|a(x,b)\brac{\frac{b}{|b|}\cdot \partial_{\varepsilon b} u(t,x)-z_b(t,x)}} \frac{b}{|b|}+l\e(t,x)=0 \text{ in }O\e,\\
& \dot{z}_b(t,x)\in \partial I_{[-\sigma_y(x,b),\sigma_y(x,b)]}\brac{-\frac{\partial \mathcal{E}_b}{\partial z_b}(u(t,x),z_b(t,x))}  \text{ in }O\e^+, b\in\cb{e_1,e_2,e_1+e_2},
\end{align*}
which is a superposition of (\ref{ev11}).
We tacitly identify $b\in \cb{e_1,e_2,e_1+e_2}$ with indices $i={1,2,3}$. The coefficients $a(x,b)$, $h(x,b)$, $\sigma_y(x,b)$ describe the properties of the spring $[x,x+\varepsilon b]$.
The above equations may be equivalently recast in the global energetic formulation for rate-independent systems (see Appendix \ref{appendix2}) with the help of energy and dissipation functionals, respectively: $\mathcal{E}\e:[0,T]\times L^2_0(O\e)^2\times L^2_0(O\e^+)^3\to \mathbb{R}$, $\Psi\e:L^2_0(O\e)^2\times L^2_0(O\e^+)^3\to [0,\infty)$,
\begin{align*}
\mathcal{E}\e(t,u,z)= & \int_{O\e^+}\frac{1}{2}\mathcal{A}(x)\brac{\nabla^{\varepsilon}_{s}u(x)-z(x)}\cdot \brac{\nabla^{\varepsilon}_{s}u(x)-z(x)} \\ &+\frac{1}{2}\mathcal{H}(x)z(x)\cdot z(x) dm\e(x)-\int_{O\e}l\e(t,x)\cdot u(x)dm\e(x),\\
\Psi\e(u,z)= & \sum_{b\in \cb{e_1,e_2,e_1+e_2}} |b|\int_{O\e^+} \sigma_y(x,b)|z_b(x)|dm\e(x).
\end{align*}
Above, the coefficients are given in the form $\mathcal{A}(x)=diag\brac{a(x,e_1),a(x,e_2), \sqrt{2}a(x,e_1+e_2)}$, $\mathcal{H}(x)= diag(h(x,e_1),h(x,e_2), h(x,e_1+e_2))$, and $\nabla^{\varepsilon}_s$ stands for the \textit{symmetrized gradient} $\nabla^{\varepsilon}_{s}u=\brac{\frac{b}{|b|}\cdot \partial_{\varepsilon b} u}_{b\in \cb{e_1,e_2,e_1+e_2}}$.  
We assume that the coefficients are random fields oscillating on a scale $\varepsilon$. In particular, the deterministic coefficients $\mathcal{A}(x)$, $\mathcal{H}(x)$ and $\sigma_y(x,b)$ in the above functionals are replaced by realizations of rescaled stationary random fields $\mathcal{A}(T_{\frac{x}{\varepsilon}}\omega)$, $\mathcal{H}(T_{\frac{x}{\varepsilon}}\omega)$ and $\sigma^b_y(T_{\frac{x}{\varepsilon}}\omega)$.  As a consequence of this, the solutions of the corresponding evolutionary equation at each time instance are not deterministic functions but rather random fields on $\Omega\times \varepsilon\mathbb{Z}^{2}$. Under suitable assumptions (cf. Section \ref{sectionERIS}), there exists a unique solution $(u\e,z\e)\in C^{Lip}([0,T],(L^2(\Omega)\otimes L^2_0(O\e)^2) \times (L^2(\Omega)\otimes L^2_0(O\e^{+})^3))$ to the above described microscopic rate-independent system.

Applying the method of stochastic unfolding, we are able to capture the averaged (w.r.t. the probability measure) behavior of the solution $(u\e,z\e)$ in the limit $\varepsilon\rightarrow 0$. Particularly, we show that, upon assuming suitable strong convergence for the initial data and forces, there exists $(U,Z,\chi)\in C^{Lip}([0,T],H^1_0(O)^2 \times L^2(\Omega\times O)^3\times (\mathbf{L}^2_{\mathsf{pot}}(\Omega)\otimes L^2(O))^2)$ which solves an effective rate-independent system on a continuum physical space (see Section \ref{sectionERIS}), and for every $t\in [0,T]$
\begin{equation*}
(u\e(t), z\e(t))\overset{c2}{\rightarrow}(U(t),Z(t),\chi(t)),
\end{equation*}
where $\overset{c2}{\rightarrow}$ denotes ``cross" two-scale convergence, which is explained in Section \ref{sectionERIS}.

In the continuum case, similar results have been obtained, for deterministic periodic materials in \cite{mielke2007two,hanke2011homogenization} (via periodic unfolding), and recently for random materials in \cite{heida2017stochastic1} (using quenched stochastic two-scale convergence) and \cite{heida2016non,heida2017stochastic3}. We discuss the literature on problems involving discrete-to-continuum transition in more detail in Section \ref{section:4}.

If we consider the constraint $z\e=0$ and time-independent force $l\e(t)=l\e$, the above problem boils down to the homogenization of the functional $u\mapsto\mathcal{E}\e(0,u,0)$, which corresponds to the discrete-to-continuum limit of the static equilibrium of a spring network with (only) elastic interactions.

We remark that the methods in this paper apply as well to systems with different constitutive laws; e.g., one might consider an energy functional with an additional term depending on the gradient of the internal variable $z\e$, as is the case in gradient plasticity (see Section \ref{S:gradient}). In the ergodic case, we even obtain a deterministic elasto-plastic limiting model. Another interesting extension of our method (which we do not discuss in this paper) is the discrete-to-continuum analysis of random spring networks featuring damage or fracture. The convergence result that we establish can be seen as a justification of continuum models for microstructural spring networks that feature uncertainty in the constitutive relations on the microscopic scale. In this context, the method could also be applied to prove the consistency of computational schemes based on the lattice method as discussed in the mechanical engineering community (e.g. see \cite{hahn2010discrete}).
\smallskip

\textbf{Structure of the paper.} In Section \ref{generalFramework}, we introduce a convenient setting for problems involving homogenization and the passage from discrete to continuum systems. Section \ref{section:3} is devoted to the introduction of the stochastic unfolding operator and its most important properties. In Section \ref{section:4}, we apply the stochastic unfolding method to an example of a multidimensional network of elastic/elasto-plastic springs with random coefficients.
\section{General framework}\label{generalFramework}
In this section, we introduce the setting for functions on a discrete/continuum physical space suited for problems involving discrete-to-continuum transitions. In addition, we present the standard setting for stochastic homogenization problems.
\subsection{Functions and differential calculus on $\varepsilon\mathbb{Z}^{d}$ and $\mathbb{R}^d$}\label{Physical_space}
Throughout the paper we consider $p,q\in (1,\infty)$ exponents of integrability that are dual, i.e.,  $\frac{1}{p}+\frac{1}{q}=1$. $\cb{e_i}_{i=1,...,d}$ denotes the standard basis of $\mathbb{R}^d$. For $\varepsilon>0$, we denote the Banach space of $p$-summable functions by $L^p(\varepsilon\mathbb{Z}^{d}):=\cb{ u:\varepsilon\mathbb{Z}^{d} \to \mathbb{R}: \brac{\varepsilon^d \sum_{x\in \varepsilon\mathbb{Z}^{d}} u^p(x)}^{\frac{1}{p}}<\infty }$. For our purposes it is convenient to view $L^p(\varepsilon\mathbb{Z}^{d})$ as the $L^p$-space of $p$-integrable functions on the measure space $\brac{\varepsilon\mathbb{Z}^{d}, 2^{\varepsilon\mathbb{Z}^{d}},m\e}$ with $m\e=\varepsilon^d \sum_{x\in \varepsilon\mathbb{Z}^{d}} \delta_x$. In particular, we use the notation $\int_{\varepsilon\mathbb{Z}^{d} }u(x)dm\e(x):=\varepsilon^d \sum_{x\in \varepsilon\mathbb{Z}^{d}}u(x)$.
For $u:\varepsilon\mathbb{Z}^{d}\rightarrow \mathbb{R}$ and $g=(g_1,...,g_d):\varepsilon\mathbb{Z}^{d}\rightarrow \mathbb{R}^d$, we set
\begin{align*}
\nabla^{\varepsilon}_iu(x)&=\frac{u(x+\varepsilon e_i)-u(x)}{\varepsilon}, & \nabla_i^{\varepsilon,*}u(x)&=\frac{u(x-\varepsilon e_i)-u(x)}{\varepsilon}, \\
\nabla^{\varepsilon}u(x) &=(\nabla^{\varepsilon}_1 u(x),...,\nabla^{\varepsilon}_d u(x) ), & \nabla^{\varepsilon,*}g(x)& =\sum_{i=1}^{d} \nabla_i^{\varepsilon,*}g_i(x),
\end{align*}
and we call $\nabla^{\varepsilon}$ \textit{discrete gradient} and $\nabla^{\varepsilon,*}$ \textit{(negative) discrete divergence} (in analogy with the usual differential operators $\nabla$ and $-div$). For $u\in L^p(\varepsilon\mathbb{Z}^{d})$, $g\in L^{q}(\varepsilon\mathbb{Z}^{d})^d$, we have the discrete integration by parts formula
\begin{equation*}
  \int_{\varepsilon\mathbb{Z}^{d}}\nabla^{\varepsilon}u(x) \cdot g(x) dm_{\varepsilon}(x)=\int_{\varepsilon\mathbb{Z}^{d}} u(x) \nabla^{\varepsilon,*}g(x) dm_{\varepsilon}(x).
\end{equation*}
\begin{definition}[Weak and strong convergence] Consider $U\in L^p(\mathbb{R}^d)$ and a sequence $u\e \in L^p(\varepsilon\mathbb{Z}^{d})$. We say that
\begin{itemize}
\item
$u\e$ weakly converges to $U$ (denoted by $u\e \rightharpoonup U$ in $L^p(\mathbb{R}^d)$) if
\begin{align*}
& \limsup_{\varepsilon \to 0}\|u\e\|_{L^p(\varepsilon\mathbb{Z}^{d})}< \infty \text{ and}\\
& \lim_{\varepsilon\to 0} \int_{\varepsilon\mathbb{Z}^{d}}u\e(x) \eta(x) dm_{\varepsilon}(x) = \int_{\mathbb{R}^d} U(x) \eta(x) dx \qquad \text{for all }\eta \in C_c^{\infty}(\mathbb{R}^d).
\end{align*}
\item $u\e$ strongly converges to $U$ (denoted by $u\e \to U$ in $L^p(\mathbb{R}^d)$) if
\begin{align*}
u\e \rightharpoonup U \text{ in }L^p(\mathbb{R}^d)\text{ and } \lim_{\varepsilon\to 0}\|u\e\|_{L^p(\varepsilon\mathbb{Z}^{d})} = \|U\|_{L^p(\mathbb{R}^d)}.
\end{align*}
\end{itemize}
\end{definition}

It is convenient to consider piecewise-constant and piecewise-affine interpolations of functions in $L^p(\varepsilon\mathbb{Z}^{d})$.
\begin{definition}
\begin{enumerate}[(i)]
\item For $u: \varepsilon\mathbb{Z}^{d}\rightarrow \mathbb{R}$, its piecewise-constant interpolation $\overline{u}:\mathbb{R}^d\rightarrow \mathbb{R}$   (subordinate to $\varepsilon\mathbb{Z}^{d}$) is given by $\overline{u}(x)=\sum_{y\in\mathbb{Z}^d} \mathbf{1}_{y+\Box}\brac{\frac{x}{\varepsilon}}u(\lfloor x \rfloor\e)$, 
where $\Box=[-\frac{1}{2},\frac{1}{2})^d$ is the unit box and $\lfloor x \rfloor\e \in \varepsilon\mathbb{Z}^{d}$ is defined by $x-\lfloor x\rfloor\e\in \varepsilon\Box$.
\item Consider a triangulation of $\mathbb{R}^d$ into $d$-simplices with nodes in $\varepsilon\mathbb{Z}^{d}$ (e.g. Freudenthal's triangulation). For $u: \varepsilon\mathbb{Z}^{d}\rightarrow \mathbb{R}$, we denote its piecewise-affine interpolation (w.r.t. the triangulation) by $\widehat{u}:\mathbb{R}^d\rightarrow \mathbb{R}$.   
\item The $\varepsilon\mathbb{Z}^{d}$-discretization $\pi_{\varepsilon}: L^1_{\mathsf{loc}}(\mathbb{R}^d)\rightarrow \mathbb{R}^{\varepsilon\mathbb{Z}^{d}}$ is defined as
\begin{equation*}
  (\pi_{\varepsilon}U)(x)=\dashint_{x+\varepsilon \Box}U(y) dy.
\end{equation*}
\end{enumerate}
\end{definition}
\begin{remark}
Note that $\overline{(\cdot)}:L^{p}(\varepsilon\mathbb{Z}^{d})\rightarrow L^{p}(\mathbb{R}^d)$, $u \mapsto \overline{u}$
defines a linear isometry. Also, $\pi_{\varepsilon}:L^{p}(\mathbb{R}^d)\rightarrow L^p(\varepsilon\mathbb{Z}^{d})$ is linear and bounded with $\|\pi_{\varepsilon}\|_{L^p(\mathbb{R}^d)\rightarrow L^{p}(\varepsilon\mathbb{Z}^{d})}\leq 1$. Furthermore, $\pi\e\circ\overline{(\cdot)}=Id$ on $L^p(\varepsilon\mathbb{Z}^{d})$, and we define $\overline{\pi}\e:=\overline{(\cdot)}\circ \pi\e$, which is a contractive projection, mapping to the subspace of piecewise-constant functions (subordinate to $\varepsilon\mathbb{Z}^{d}$) in $L^p(\mathbb{R}^d)$.  
\end{remark}

The proof of the following lemma is an uncomplicated exercise, and therefore we omit it.
\begin{lemma}\label{lemma_equivalent_conv}
Let $u\e \in L^p(\varepsilon\mathbb{Z}^{d})$ and $U \in L^p(\mathbb{R}^d)$. The following claims are equivalent:
\begin{enumerate}[(i)]
\item $u\e \rightharpoonup (\to) U$  in $L^p(\mathbb{R}^d)$.
\item $\overline{u}\e \to U $ weakly (strongly) in $L^p(\mathbb{R}^d)$.
\item $\widehat{u}\e \to U $ weakly (strongly) in $L^p(\mathbb{R}^d)$.
\end{enumerate}
\end{lemma}
The applications involve problems with homogeneous Dirichlet boundary conditions, and therefore the following space is convenient: For $O\subset \mathbb{R}^d$ we set
\begin{equation*}
L^p_0(O\cap \varepsilon\mathbb{Z}^{d} )=\cb{u \in L^p(\varepsilon\mathbb{Z}^{d}): u=0 \text{ in }\varepsilon\mathbb{Z}^{d} \setminus \brac{O\cap \varepsilon\mathbb{Z}^{d}}}.
\end{equation*}
\subsection{Description of random media and a differential calculus for random variables}\label{section:297}
As is standard in stochastic homogenization, we describe \textit{random configurations} (e.g. coefficients of a PDE or energy densities that describe properties of some medium with quenched disorder) using a probability space $\brac{\Omega,\mathcal{F},P}$ together with a measure preserving dynamical system $T_x:\Omega\rightarrow \Omega$ ($x\in \mathbb{Z}^d$) such that:
\begin{enumerate}[(i)]
\item $T_x:\Omega\rightarrow\Omega$ is measurable for all $x \in \mathbb{Z}^d$,
\item $T_0=Id$ and $T_{x+y}=T_x \circ T_y$ for all $x,y\in \mathbb{Z}^d$,
\item $P(T_x A)=P(A)$ for all $A\in \mathcal{F}$ and $x\in \mathbb{Z}^d$.
\end{enumerate} 
We write $\expect{\cdot}$ for the expectation and $L^p(\Omega)$ for the usual Banach space of $p$-integrable random variables. Throughout the paper we assume that $(\Omega, \mathcal F, P,T)$ satisfies the properties above, and that $\brac{\Omega,\mathcal{F},P}$ is a separable measure space; the latter implies the separability of $L^p(\Omega)$.

The dynamical system $T$ is called \textit{ergodic} (we also say $\expect{\cdot}$ is ergodic), if for any $A\in\mathcal F$ the following implication holds:
\begin{align*}
  &\text{$A$ is shift invariant, i.e.,  }T_xA=A\text{ for all }x\in\mathbb Z^d\\
  &\Rightarrow\qquad P(A)\in\{0,1\}.
\end{align*} 

\begin{remark}
A multiparameter version of Birkhoff's ergodic theorem (see \cite[Theorem 2.4]{akcoglu1981ergodic}) states that if $\expect{\cdot}$ is ergodic and $\varphi \in L^p(\Omega)$, then
\begin{align*}
\lim_{R\to \infty}\frac{1}{|B_R|}\sum_{x\in B_R\cap \mathbb{Z}^d}\varphi(T_x \omega)\to \expect{\varphi} \quad \text{for P-a.e. }\omega\in \Omega.
\end{align*}
\end{remark}

\begin{example}
Let $(\Omega_0,\mathcal F_0,P_0)$ denote a separable probability space. We define $(\Omega,\mathcal F,P)$ as the $\mathbb Z^d$-fold product of $(\Omega_0,\mathcal F_0,P_0)$, i.e.,   $\Omega:=\Omega_0^{\mathbb Z^d}$, $\mathcal F=\otimes_{\mathbb Z^d}\mathcal F_0$, $P=\otimes_{\mathbb Z^d}P_0$. Note that a configuration $\omega\in\Omega$ can be seen as a function $\omega:\mathbb Z^d\to\Omega_0$. We define the shift $T_x$ ($x\in\mathbb Z^d$) as
  \begin{equation*}
    T_x\omega(\cdot):=\omega(\cdot+x).
  \end{equation*}
It follows that $(\Omega,\mathcal F,P,T)$ satisfies the assumptions above and defines an ergodic dynamical system. With regard to the example in the introduction, \eqref{eq2}, we might consider random potentials of the form
  \begin{equation*}
    V(\omega,F):=a_0(\omega(0))|F|^2,
  \end{equation*}
  where $a_0:\Omega_0\to(\lambda,1)$ is a random variable, with $\lambda>0$ denoting a positive constant of ellipticity. We remark that the coefficients appearing in the corresponding energy $\cb{a_0(T_{\frac{x}{\varepsilon}}\omega(0))}_{x\in \varepsilon\mathbb{Z}^{d}}$ are independent and identically distributed random variables. 
\end{example}

For  $\varphi: \Omega\rightarrow \mathbb{R}$ and $\mathbb{\psi}=(\psi_1,...,\psi_d):\Omega\rightarrow \mathbb{R}^d$ measurable, we introduce the \textit{horizontal derivative} $D$ and \textit{(negative) horizontal divergence} $D^*$:
\begin{align*}
 D_i\varphi(\omega)&=\varphi(T_{e_i}\omega)-\varphi(\omega), & D_i^*\varphi(\omega)&=\varphi(T_{-e_i}\omega)-\varphi(\omega),\\
 D\varphi(\omega)&=(D_1\varphi(\omega),...,D_d\varphi(\omega)), & D^*\mathbf{\psi}(\omega)&=\sum_{i=1}^{d}D^*_i\psi_i(\omega).
\end{align*}
\begin{remark} 
$D:L^p(\Omega)\rightarrow L^p(\Omega)^d$ and $D^*: L^p(\Omega)^d \rightarrow L^p(\Omega)$ are linear and bounded operators. Furthermore, for any $\varphi \in L^p(\Omega)$ and $\psi \in L^q(\Omega)^d$ the integration by parts formula
\begin{equation*}
  \expect{D\varphi \cdot \psi}=\expect{\phi D^*\psi}
\end{equation*}
holds. Hence, $D$ (defined on $L^p(\Omega)$) and $D^*$ (defined on $L^q(\Omega)^d$) are adjoint operators.
\end{remark}

We denote the set of \textit{shift-invariant functions} in $L^p(\Omega)$ by
\begin{equation*}
L^p_{\mathsf{inv}}(\Omega):=\left\lbrace \varphi \in L^p(\Omega): \varphi(T_x\omega)=\varphi(\omega) \text{ for all } x\in \mathbb{Z}^d \text{ and a.e. }\omega\in \Omega   \right\rbrace,
\end{equation*}
and we note that $L^p_{\mathsf{inv}}(\Omega)\simeq \mathbb{R}$ if and only if $\expect{\cdot} \text{ is ergodic}$. We denote by $P_{\mathsf{inv}}:L^p(\Omega)\rightarrow L^p_{\mathsf{inv}}(\Omega)$ the conditional expectation w.r.t. the $\sigma$-algebra generated by the family of shift invariant sets $\cb{A\in \mathcal{F}: T_xA=A \; \text{for every }x\in \mathbb{Z}^d}$. It is a contractive projection and in the ergodic case we simply have $P_{\mathsf{inv}}f=\expect{f}$. The adjoint of $P_{\mathsf{inv}}$ is denoted by $P_{\mathsf{inv}}^*:L^q(\Omega)\to L^q(\Omega)$.

It is easily verified that $L^p_{\mathsf{inv}}(\Omega)=\mathsf{ker}D$ and by standard arguments (see \cite[Section 2.6]{brezis2010functional}) we have the
 orthogonality relations
\begin{equation*}
 \mathbf{L}^p_{\mathsf{pot}}(\Omega):=\overline{\mathsf{ran}D}^{L^p(\Omega)^d}=(\mathsf{ker}D^*)^{\bot},\quad \quad \quad \quad
 \mathsf{ker}D= (\mathsf{ran}D^*)^{\bot}.
\end{equation*}
The above relations hold in the sense of $L^p$-$L^q$ duality (we identify $L^q(\Omega)'$ with $L^p(\Omega)$). Namely, $D:L^p(\Omega)\rightarrow L^p(\Omega)^d$ and $D^*:L^q(\Omega)^d\rightarrow L^q(\Omega)$ and the orthogonal of a set $A\subset L^q(\Omega)$ is given by 
\begin{equation*}
A^{\bot}=\cb{\varphi \in L^q(\Omega)': \langle \varphi,\psi \rangle_{(L^q)',L^q}=0 \text{ for all } \psi \in A}.
\end{equation*}

In this paper measurable functions defined on $\Omega\times \varepsilon\mathbb{Z}^{d} $ or on $\Omega \times \mathbb{R}^d $ are called random fields. We mainly consider the space of $p$-integrable random fields $L^p(\Omega \times \varepsilon\mathbb{Z}^{d})$, and frequently use the following notation: If $X\subset L^p(\Omega)$ and $Y\subset L^p(\varepsilon\mathbb{Z}^{d})$ (resp. $Y\subset L^p(\mathbb{R}^d)$) are linear subspaces, then we denote by $X\otimes Y$ the closure of 
\begin{equation*}
  X\stackrel{a}{\otimes} Y:=\operatorname{span}\cb{\sum_j \varphi_j \eta_j: \varphi_j\in X,\eta_j\in Y}
  \end{equation*}
in $L^p(\Omega\times\varepsilon\mathbb{Z}^{d})$ (resp. $L^p(\Omega\times\mathbb{R}^d)$). In particular, since $L^p(\Omega)$ is separable (thanks to our assumption on the underlying measure space), we have $L^p(\Omega)\otimes L^p(\varepsilon\mathbb{Z}^{d})=L^p(\Omega\times \varepsilon\mathbb{Z}^{d})$ (resp. $L^p(\Omega)\otimes L^p(\mathbb{R}^d)=L^p(\Omega\times \mathbb{R}^d)$). Similarly, if above we instead have $Y\subset W^{1,p}(\mathbb{R}^d)$ is a linear subspace, then $X\otimes Y$ is defined as the closure of $X\overset{a}{\otimes}Y$ in $L^p(\Omega, W^{1,p}(\mathbb{R}^d))$. In this respect, we tacitly identify linear and bounded operators on $X$ (or $Y$) by their obvious extension to $X\otimes Y$. 
\section{Stochastic unfolding}\label{section:3}
\subsection{Definition and properties}\label{section31}
For $u: \Omega\times \varepsilon\mathbb{Z}^{d}\rightarrow \mathbb{R}$ we define the \textit{unfolding} of $u$ via
\begin{equation*}
(\widetilde{\mathcal{T}}_{\varepsilon}u)(\omega,x)=u(T_{ - \frac{x}{\varepsilon}}\omega,x).
\end{equation*}
The above expression defines an isometric isomorphism $\widetilde{\mathcal{T}_{\varepsilon}}: L^p(\Omega\times \varepsilon\mathbb{Z}^{d})\rightarrow  L^p(\Omega\times\varepsilon\mathbb{Z}^{d})$. For our purposes, it is convenient to consider $\mathcal{T}_{\varepsilon}:=\overline{(\cdot)}\circ \widetilde{\mathcal{T}}_{\varepsilon}:L^p\brac{\Omega\times \varepsilon\mathbb{Z}^{d}}\rightarrow L^p\brac{\Omega\times \mathbb{R}^d}$, which is a linear (nonsurjective) isometry. We call both operators $\widetilde{\mathcal{T}_{\varepsilon}}$ and $\mathcal{T}_{\varepsilon}$ \textit{stochastic unfolding operators}. Note that $\widetilde{\mathcal{T}}\e$ (defined on $L^p$) and $\widetilde{\mathcal{T}}\e^{-1}$ (given by $\widetilde{\mathcal{T}}\e^{-1}v(\omega,x)=v(T_{\frac{x}{\varepsilon}}\omega,x)$ for $v\in L^q(\Omega\times \varepsilon\mathbb{Z}^{d})$) are adjoint operators.
\begin{definition}[Two-scale convergence in the mean]\label{definitionUnf} We say that a sequence $u_{\varepsilon}\in L^{p}(\Omega \times \varepsilon \mathbb{Z}^d)$ strongly (weakly) stochastically two-scale converges in the mean to $U\in L^{p}(\Omega\times \mathbb{R}^d)$ if 
\begin{equation*}
\mathcal{T}_{\varepsilon} u_{\varepsilon}\rightarrow U \quad \text{strongly (weakly) in } L^{p}(\Omega\times \mathbb{R}^d),
\end{equation*}
and we use the notation  $u_{\varepsilon}\overset{2}{\rightarrow} u$ ($u_{\varepsilon}\overset{2}{\rightharpoonup} u$) in $L^{p}(\Omega\times \mathbb{R}^d)$. For vector-valued functions, the convergence is defined componentwise.
\end{definition}
\begin{remark}\label{remark1} Note that the adaptation of the two-scale convergence in the mean from \cite{bourgeat1994stochastic,andrews1998stochastic} to the discrete setting reads as $u_{\varepsilon}\in L^{p}(\Omega \times \varepsilon \mathbb{Z}^d)$ stochastically two-scale converges in the mean to $U\in L^{p}(\Omega\times \mathbb{R}^d)$ if $\limsup_{\varepsilon\to 0} \expect{\int_{\varepsilon\mathbb{Z}^{d}}(u_{\varepsilon}(\omega,x))^pdm_{\varepsilon}(x)}<\infty$ and 
\begin{equation*}
\lim_{\varepsilon\rightarrow 0}\expect{\int_{\varepsilon\mathbb{Z}^{d}}u_{\varepsilon}(\omega,x)\varphi(T_{\frac{x}{\varepsilon}}\omega)\eta(x)dm_{\varepsilon}(x)}=\expect{\int_{\mathbb{R}^d}U(\omega,x)\varphi(\omega)\eta(x)dx}
\end{equation*}
for all $\varphi\in L^q(\Omega)$ and all $\eta \in C_c^{\infty}(\mathbb{R}^d)$. This notion is equivalent to our notion of weak stochastic two-scale convergence in the mean.
\end{remark}

The following lemma is obtained easily by exploiting the fact that the unfolding is a linear isometry and by the usual properties of weak convergence, and therefore we omit its proof. 
\begin{lemma}[Basic facts]\label{basicfacts}
We consider sequences $u_{\varepsilon}\in L^{p}(\Omega \times \varepsilon \mathbb{Z}^d)$ and $v\e\in L^q(\Omega\times \varepsilon\mathbb{Z}^{d})$.
\begin{enumerate}[(i)]
\item If $u_{\varepsilon} \overset{2}{\rightharpoonup} U$ in $L^p(\Omega\times \mathbb{R}^d)$, then 
\begin{equation*}
\sup_{\varepsilon\in (0,1)}\|u_{\varepsilon}\|_{L^{p}(\Omega \times \varepsilon \mathbb{Z}^d)}<\infty \quad \text{and} \quad \|U\|_{L^{p}(\Omega\times \mathbb{R}^d)}\leq \liminf_{\varepsilon\rightarrow 0}\|u\e\|_{L^{p}(\Omega \times \varepsilon \mathbb{Z}^d)}.
\end{equation*}
\item If $\limsup_{\varepsilon\rightarrow 0}\|u_{\varepsilon}\|_{L^{p}(\Omega \times \varepsilon \mathbb{Z}^d)}<\infty$, then there exist $U\in L^{p}(\Omega\times \mathbb{R}^d)$ and a subsequence $\varepsilon'$ such that $u_{\varepsilon'}\overset{2}{\rightharpoonup} U$ in $L^{p}(\Omega\times \mathbb{R}^d)$.
\item $u_{\varepsilon}\overset{2}{\rightarrow} U$ in $L^{p}(\Omega\times \mathbb{R}^d)$ if and only if $\lim_{\varepsilon\to 0}\|u_{\varepsilon}\|_{L^{p}(\Omega \times \varepsilon \mathbb{Z}^d)} = \|u \|_{L^{p}(\Omega\times \mathbb{R}^d)}$ and $u_{\varepsilon} \overset{2}{\rightharpoonup} U$ in $L^{p}(\Omega\times \mathbb{R}^d)$.
\item If $u_{\varepsilon}\overset{2}{\rightarrow} U$ in $L^{p}(\Omega\times \mathbb{R}^d)$ and $v_{\varepsilon}\overset{2}{\rightharpoonup} V$ in $L^{q}(\Omega\times \mathbb{R}^d)$, then
\begin{equation*}
\lim_{\varepsilon\rightarrow 0}\expect{\int_{\varepsilon\mathbb{Z}^{d}} u_{\varepsilon}(\omega,x)v_{\varepsilon}(\omega,x)dm_{\varepsilon}(x)}=\expect{\int_{\mathbb{R}^d}U(\omega,x)V(\omega,x)dx}.
\end{equation*}
\end{enumerate}
\end{lemma}

As in the periodic setting, a suitable ``inverse'' of the unfolding operator $\mathcal T\e$ is given by the linear operator
\begin{equation*}
  \mathcal{F}_{\varepsilon}:L^{p}(\Omega\times \mathbb{R}^d) \rightarrow L^{p}(\Omega \times \varepsilon \mathbb{Z}^d), \quad \mathcal{F}_{\varepsilon}=\widetilde{\mathcal{T}}^{-1}_{\varepsilon}\circ\pi_{\varepsilon}.
\end{equation*}
In analogy to the periodic case, we refer to $\mathcal F\e$ as the \textit{stochastic folding operator}. Note that $\mathcal F\e: L^{p}(\Omega\times \mathbb{R}^d) \rightarrow L^{p}(\Omega \times \varepsilon \mathbb{Z}^d)$ is precisely the adjoint of $\mathcal T\e:L^{q}(\Omega\times \varepsilon\mathbb{Z}^d) \rightarrow L^{q}(\Omega \times \mathbb{R}^d)$ (where $\frac1p+\frac1q=1$).
\begin{lemma}\label{Propfold} 
$\mathcal{F}_{\varepsilon}$ is linear and it satisfies:
\begin{enumerate}[(i)]
\item $\| \mathcal{F}_{\varepsilon}U \|_{L^{p}(\Omega \times \varepsilon \mathbb{Z}^d)}\leq \|U \|_{L^{p}(\Omega\times \mathbb{R}^d)}$ for every $U\in L^{p}(\Omega\times \mathbb{R}^d)$.
\item $\mathcal{F}_{\varepsilon} \circ \mathcal{T}_{\varepsilon}=Id$ on $L^{p}(\Omega \times \varepsilon \mathbb{Z}^d)$ and $\mathcal{T}_{\varepsilon} \circ \mathcal{F}_{\varepsilon}=\overline{\pi}_{\varepsilon}$ on $L^{p}(\Omega\times \mathbb{R}^d)$.
\item $\mathcal{F}_{\varepsilon}U\overset{2}{\rightarrow} U$ in $L^{p}(\Omega\times \mathbb{R}^d)$ for every $U \in L^{p}(\Omega\times \mathbb{R}^d)$. 
\end{enumerate}
\end{lemma}
The proof of this lemma is omitted since it mostly relies on the definition of the folding operator.
\subsection{Two-scale limits of gradients}\label{gradients}
In this section, we treat two-scale limits of gradients. First we present some compactness results and later we show that weak two-scale limits can be recovered in the strong two-scale sense by convenient linear constructions.
\begin{proposition}[Compactness]\label{comp3}
Let $\gamma\geq 0$. Let $u_{\varepsilon}\in L^{p}(\Omega \times \varepsilon \mathbb{Z}^d)$ satisfy
\begin{equation}\label{comp3:bound}
  \limsup_{\varepsilon\rightarrow 0}\left\langle \int_{\varepsilon\mathbb{Z}^{d}} |u_{\varepsilon}(\omega,x)|^p+\varepsilon^{\gamma p}|\nabla^{\varepsilon}u_{\varepsilon}(\omega,x)|^p dm_{\varepsilon}(x)\right\rangle<\infty.
\end{equation}
\begin{enumerate}[(i)]
\item If $\gamma=0$, there exist $U \in L^p_{\mathsf{inv}}(\Omega)\otimes W^{1,p}(\mathbb{R}^d)$ and $\chi \in \mathbf{L}^p_{\mathsf{pot}}(\Omega)\otimes L^p(\mathbb{R}^d)$ such that, up to a subsequence,
\begin{equation*}\label{er12}
u_{\varepsilon}\overset{2}{\rightharpoonup} U  \text{ in }L^{p}(\Omega\times \mathbb{R}^d), \quad \nabla^{\varepsilon}u_{\varepsilon} \overset{2}{\rightharpoonup} \nabla U+\chi \text{ in } L^{p}(\Omega\times \mathbb{R}^d)^d.
\end{equation*} 
\item If $\gamma\in (0,1)$, there exist $U \in L^p_{\mathsf{inv}}(\Omega)\otimes L^p(\mathbb{R}^d)$ and $\chi \in \mathbf{L}^p_{\mathsf{pot}}(\Omega)\otimes L^p(\mathbb{R}^d)$ such that, up to a subsequence,
\begin{equation*}\label{er121}
u_{\varepsilon}\overset{2}{\rightharpoonup} U  \text{ in }L^{p}(\Omega\times \mathbb{R}^d), \quad \varepsilon^{\gamma} \nabla^{\varepsilon}u_{\varepsilon} \overset{2}{\rightharpoonup} \chi \text{ in } L^{p}(\Omega\times \mathbb{R}^d)^d.
\end{equation*}
\item If $\gamma=1$, there exists $U \in L^{p}(\Omega\times \mathbb{R}^d)$ such that, up to a subsequence,
\begin{equation*}\label{er122}
u_{\varepsilon}\overset{2}{\rightharpoonup} U  \text{ in }L^{p}(\Omega\times \mathbb{R}^d), \quad \varepsilon \nabla^{\varepsilon}u_{\varepsilon} \overset{2}{\rightharpoonup}  DU \text{ in } L^{p}(\Omega\times \mathbb{R}^d)^d.
\end{equation*}
\item If $\gamma>1$, there exists $U \in L^{p}(\Omega\times \mathbb{R}^d)$ such that, up to a subsequence,
\begin{equation*}\label{er123}
u_{\varepsilon}\overset{2}{\rightharpoonup} U  \text{ in }L^{p}(\Omega\times \mathbb{R}^d), \quad \varepsilon^{\gamma} \nabla^{\varepsilon}u_{\varepsilon} \overset{2}{\rightharpoonup}  0 \text{ in } L^{p}(\Omega\times \mathbb{R}^d)^d.
\end{equation*}
\end{enumerate}
\end{proposition}
See Section \ref{proofs3.4} for the proof.

The above statement can be adapted to sequences supported in a domain: Let $O\subset \mathbb{R}^d$ be open. We denote by $W^{1,p}_0(O)$ the closure of $C^\infty_c(O)$ in $W^{1,p}(O)$. Since the unfolding operator is naturally defined for functions on $\mathbb{R}^d$, we tacitly identify functions in $L^p(O)$ and $W^{1,p}_0(O)$ with their trivial extension by $0$ to $\mathbb{R}^d$. As a corollary of Proposition~\ref{comp3} we obtain the following.
\begin{corollary}\label{cor111}
Let $O\subset{\mathbb{R}^d}$ be a bounded domain with Lipschitz boundary, and set $O^{+\varepsilon}:=\cb{x \in \mathbb{R}^d: dist(x,O)\leq C \varepsilon}$ where $C>0$ denotes an arbitrary constant independent of $\varepsilon>0$. Consider a sequence $u\e \in L^p(\Omega)\otimes L^p_0(O^{+\varepsilon}\cap \varepsilon\mathbb{Z}^{d})$  satisfying \eqref{comp3:bound}. Then in addition to the convergence statements in Proposition \ref{comp3}, the two-scale limits (from Proposition \ref{comp3}) satisfy
\begin{itemize}
\item if $\gamma=0$,  $U \in L^p_{\mathsf{inv}}(\Omega)\otimes W^{1,p}_0(O)$ and $\chi \in \mathbf{L}^p_{\mathsf{pot}}(\Omega)\otimes L^p(O)$;
\item if $\gamma \in (0,1)$, $U \in L^p_{\mathsf{inv}}(\Omega)\otimes L^{p}(O)$ and $\chi \in \mathbf{L}^p_{\mathsf{pot}}(\Omega)\otimes L^p(O)$;
\item if $\gamma\geq 1$, $U\in L^p(\Omega)\otimes L^p(O)$.
\end{itemize}
\end{corollary}
The proof of the above corollary is left to the reader. We remark that in Proposition \ref{comp3} (i) and (ii) the two-scale limit $U$ is shift-invariant and therefore in the ergodic setting it is deterministic, i.e.,  $U=P_{\mathsf{inv}}U=\expect{U}$.
\begin{corollary}\label{ergodic}
Let $\gamma \in [0,1)$ and $\expect{\cdot}$ be ergodic. Let $u\e$ satisfy the assumptions in Proposition \ref{comp3}. Then the claims in Proposition \ref{comp3} (i) and (ii) hold and we have the following:
\begin{enumerate}[(i)]
\item If $\gamma=0$, then $\expect{u\e}\rightharpoonup U$  in $L^p(\mathbb{R}^d)$, $\expect{\nabla^{\varepsilon}u\e}\rightharpoonup\nabla U$ in $L^p(\mathbb{R}^d)^d$ and $u_{\varepsilon}-\expect{u_{\varepsilon}}\overset{2}{\rightharpoonup} 0$ in $L^p(\Omega\times \mathbb{R}^d)$.
\item If $\gamma \in (0,1)$, then $\expect{u\e}\rightharpoonup U$  in $L^p(\mathbb{R}^d)$, $\expect{\varepsilon^{\gamma} \nabla^{\varepsilon}u\e}\rightharpoonup 0$ in $L^p(\mathbb{R}^d)^d$ and $u_{\varepsilon}-\expect{u_{\varepsilon}}\overset{2}{\rightharpoonup} 0$  in $L^p(\Omega\times \mathbb{R}^d)$.
\item If $\gamma=0$ and $Q\subset{\mathbb{R}^d}$ is open and bounded with Lipschitz boundary, then $\overline{\expect{u_{\varepsilon}}}\rightarrow U$ strongly in $L^p(Q)$.
\item If $\gamma \in [0,1)$ and if, additionally, $u\e \overset{2}{\to}U$  in $L^p(\Omega\times \mathbb{R}^d)$, then for any $\varphi\in L^{\infty}(\Omega)$ we have $\expect{u\e \varphi}\to \expect{\varphi} U $  in $L^p(\mathbb{R}^d)$.
\end{enumerate}
\end{corollary}
For the proof of this corollary, see Section \ref{proofs3.4}.

In the following, we show that weak two-scale accumulation points can be recovered in the strong two-scale sense. 
\begin{proposition}\label{prop1}
\begin{enumerate}[(i)]
\item Let $\gamma \in [0,1)$. For $\varepsilon>0$ there exists a linear and bounded operator $\mathcal{G}_{\varepsilon}^{\gamma}: \mathbf{L}^p_{\mathsf{pot}}(\Omega)\otimes L^p(\mathbb{R}^d) \rightarrow L^p(\Omega\times \varepsilon\mathbb{Z}^{d})$ such that
\begin{equation*}
\mathcal{G}^{\gamma}_{\varepsilon} \chi \overset{2}{\rightarrow} 0  \text{ in } L^p(\Omega \times \mathbb{R}^d)\quad\text{ and }\quad
\varepsilon^{\gamma} \nabla^{\varepsilon} \mathcal{G}_{\varepsilon}^{\gamma} \chi \overset{2}{\to} \chi \text{ in } L^p(\Omega \times \mathbb{R}^d)^d
\end{equation*}
for all $\chi\in \mathbf{L}^p_{\mathsf{pot}}(\Omega)\otimes L^p(\mathbb{R}^d)$.  Moreover, the operator norm of $\mathcal{G}\e^{\gamma}$ can be bounded independently of $0<\varepsilon\leq 1$.
\item Let $U\in L^p_{\mathsf{inv}}(\Omega)\otimes W^{1,p}(\mathbb{R}^d)$: We have
\begin{equation*}
\nabla^{\varepsilon}\mathcal{F}\e U\overset{2}{\rightarrow}\nabla U \text{ in }L^{p}(\Omega\times \mathbb{R}^d)^d.
\end{equation*} 
\item Let $\gamma\in (0,1)$. For $\varepsilon>0$ there exists a linear and bounded operator $\mathcal{F}\e^{\gamma}: L^p_{\mathsf{inv}}(\Omega)\otimes L^p(\mathbb{R}^d)\to L^p(\Omega\times \varepsilon\mathbb{Z}^{d})$ such that 
\begin{equation*}
\mathcal{F}\e^{\gamma} U \overset{2}{\rightarrow}U \text{ in }L^{p}(\Omega\times \mathbb{R}^d), \quad
\varepsilon^{\gamma} \nabla^{\varepsilon}\mathcal{F}\e^{\gamma} U\overset{2}{\rightarrow}0 \text{ in }L^{p}(\Omega\times \mathbb{R}^d)^d
\end{equation*}
for all $U \in L^p_{\mathsf{inv}}(\Omega)\otimes L^p(\mathbb{R}^d)$. Moreover, the operator norm of $\mathcal{F}\e^{\gamma}$ can be bounded independently of $0<\varepsilon\leq 1$. 
\item Let $\gamma\geq 1$. For any $U \in L^p(\Omega\times \mathbb{R}^d)$, it holds that
\begin{equation*}
\varepsilon^{\gamma}\nabla^{\varepsilon}\mathcal{F}\e U \overset{2}{\to}a_{\gamma}DU \text{ in } L^p(\Omega\times \mathbb{R}^d)^d,
\end{equation*}
where $a_{\gamma}=\twopartdef{1}{\gamma=1,}{0}{\gamma>1.}$ 
\end{enumerate} 
\end{proposition}
\begin{corollary}\label{rem5}
\begin{enumerate}[(i)]
\item The mapping 
  \begin{align*}
     &(L^p_{\mathsf{inv}}(\Omega)\otimes W^{1,p}(\mathbb{R}^d))\times (\mathbf{L}^p_{\mathsf{pot}}(\Omega)\otimes L^p(\mathbb{R}^d))\to L^p(\Omega\times \varepsilon\mathbb{Z}^{d})\\
     &\qquad\qquad (U,\chi)\mapsto \mathcal{F}\e U+\mathcal{G}^0\e \chi=:u\e(U,\chi)
   \end{align*}
   is linear and bounded, and it holds that
   \begin{align*}
     u\e(U,\chi) \overset{2}{\to} U \text{ in }L^{p}(\Omega\times \mathbb{R}^d), \quad  \nabla^{\varepsilon}u\e(U,\chi) \overset{2}{\to}\nabla U +\chi \text{ in }L^{p}(\Omega\times \mathbb{R}^d)^d.
   \end{align*}
   Moreover, its operator norm is bounded uniformly in $0<\varepsilon\leq 1$.
 \item Let $\gamma\in (0,1)$. The mapping 
   \begin{align*}
     &(L^p_{\mathsf{inv}}(\Omega)\otimes L^p(\mathbb{R}^d))\times
     (\mathbf{L}^p_{\mathsf{pot}}(\Omega)\otimes L^p(\mathbb{R}^d))
     \to
     L^p(\Omega\times \varepsilon\mathbb{Z}^{d})\\
     &\qquad\qquad (U,\chi)\mapsto \mathcal{F}\e^{\gamma} U+\mathcal{G}^{\gamma}\e \chi=:u\e(U,\chi)
  \end{align*}
  is linear and bounded and it holds that
\begin{align*}
 & u\e(U,\chi) \overset{2}{\to} U \text{ in }L^{p}(\Omega\times \mathbb{R}^d), \quad  \varepsilon^{\gamma}\nabla^{\varepsilon}u\e(U,\chi) \overset{2}{\to} \chi \text{ in }L^{p}(\Omega\times \mathbb{R}^d)^d.
\end{align*}
Moreover, its operator norm is bounded uniformly in $0<\varepsilon\leq 1$.
\smallskip

Let $O\subset\mathbb{R}^d$ be open and bounded with Lipschitz boundary.
\item For any $(U,\chi)\in ( L^p_{\mathsf{inv}}(\Omega)\otimes W^{1,p}_0(O))\times (\mathbf{L}^p_{\mathsf{pot}}(\Omega)\otimes L^p(O))$, we can find a sequence $u\e \in L^p(\Omega)\otimes L^p_0(O\cap \varepsilon\mathbb{Z}^{d})$ such that
\begin{equation*}
 u\e \overset{2}{\to} U \text{ in }L^{p}(\Omega\times \mathbb{R}^d), \quad \nabla^{\varepsilon}u\e \overset{2}{\to}\nabla U +\chi \text{ in }L^{p}(\Omega\times \mathbb{R}^d)^d.
\end{equation*}
\item  Let $\gamma \in (0,1)$. There exists a mapping 
  \begin{align*}
    &( L^p_{\mathsf{inv}}(\Omega)\otimes L^p(O)) \times (\mathbf{L}^p_{\mathsf{pot}}(\Omega)\otimes L^p(O))
    \to
    L^p(\Omega)\otimes L^p_0(O\cap \varepsilon\mathbb{Z}^{d})\\
    &\qquad\qquad (U,\chi)\mapsto u\e(U,\chi),
  \end{align*}
  which is linear and bounded, and it holds that
  \begin{align*}
    & u\e(U,\chi) \overset{2}{\to} U \text{ in }L^{p}(\Omega\times \mathbb{R}^d), \quad  \varepsilon^{\gamma} \nabla^{\varepsilon}u\e(U,\chi) \overset{2}{\to} \chi \text{ in }L^{p}(\Omega\times \mathbb{R}^d)^d.
\end{align*}
Moreover, its operator norm is bounded uniformly in $0<\varepsilon\leq 1$.
\end{enumerate}  
\end{corollary}
For the proof of the above two results, see Section \ref{proofs3.4}. We remark that in the case $\gamma\geq 1$, the recovery sequence for $U \in L^p(\Omega\times \mathbb{R}^d)$ is simply given by $\mathcal{F}\e U$. Moreover, in the case of prescribed boundary data for the recovery sequence, we might consider a cut-off procedure as in (iv) above.  
\begin{remark} 
Note that the construction of the recovery sequence in the whole-space cases (i) and (ii) (and if $\gamma\in (0,1)$ for a domain (iv)) is linear in the sense that the mapping $(U,\chi)\mapsto u\e$ is linear. In contrast, the construction for a domain (iii) is nonlinear, since it relies on a cut-off procedure applied to the whole-space construction. We remark that the cut-off procedure can be avoided in certain cases: For $p=2$, we can construct the recovery sequence, similarly as in the proof of Proposition \ref{prop1} (i), by defining $u\e$ as the unique solution of $\nabla^{\varepsilon,*}\nabla^{\varepsilon}u\e=\nabla^{\varepsilon,*}(\nabla^{\varepsilon}\mathcal{F}\e U+\mathcal{F}\e\chi)$ in the interior of $O\cap \varepsilon\mathbb{Z}^{d}$ and with prescribed homogeneous Dirichlet boundary data. For $p\neq 2$ the same strategy applies as long as the above discrete elliptic equation satisfies maximal $L^p$-regularity. The latter depends on the regularity of the domain $O$.
\end{remark}
\subsection{Unfolding and (lower semi-)continuity of convex energies}\label{section32}
Our $\Gamma$-convergence results for convex energies exploit the following result.
\begin{proposition}\label{pro67}Let $V:\Omega \times \mathbb{R}^k \rightarrow \mathbb{R}$ be jointly measurable (i.e.,  w.r.t. $\mathcal{F}\otimes \mathcal{B}(\mathbb{R}^k)$) and for $P$-a.e. $\omega\in \Omega$, $V(\omega,\cdot)$ be convex. Moreover, we assume that there exists $C>0$ such that
\begin{equation*}
\frac{1}{C}|F|^p-C\leq V(\omega,F) \leq C(|F|^p+1),
\end{equation*}
for $P$-a.e. $\omega\in \Omega$ and all $F\in \mathbb{R}^k$. Let $O,O^{+\varepsilon} \subset{\mathbb{R}^d}$ be bounded domains with Lipschitz boundaries satisfying $ O\subset O^{+\varepsilon}\subset \cb{x \in \mathbb{R}^d: dist(x,O)\leq C \varepsilon}$ for some $C>0$.
\begin{enumerate}[(i)]
\item If $u\e \in L^p(\Omega\times \varepsilon\mathbb{Z}^{d})^k$ and $u\e \overset{2}{\rightharpoonup} U$ in $L^{p}(\Omega\times \mathbb{R}^d)^k$, then
\begin{equation*}
\liminf_{\varepsilon\rightarrow 0}\expect{\int_{O^{+\varepsilon}\cap \varepsilon\mathbb{Z}^{d}}V(T_{\frac{x}{\varepsilon}}\omega,u\e(\omega,x))dm_{\varepsilon}(x)}\geq \expect{\int_{O} V(\omega,U(\omega,x))dx}.
\end{equation*}
\item If $u\e \in L^{p}(\Omega \times \varepsilon \mathbb{Z}^d)^k$ and $u\e \overset{2}{\to} U$ in $L^{p}(\Omega\times \mathbb{R}^d)^k$, then
\begin{equation*}
\lim_{\varepsilon\rightarrow 0}\expect{\int_{O^{+\varepsilon}\cap \varepsilon\mathbb{Z}^{d}} V(T_{\frac{x}{\varepsilon}}\omega,u\e(\omega,x))dm_{\varepsilon}(x)}= \expect{\int_{O}V(\omega,U(\omega,x))dx}.
\end{equation*}
\end{enumerate}
\end{proposition}
The proof of this result is in Section \ref{proofs3.4}. We have applications in mind, where such integral functionals are treated and the role of $u\e$ is played by a discrete (symmetrized) gradient (see Section \ref{section:4}).
\subsection{Proofs}\label{proofs3.4} Before proving Proposition \ref{comp3}, we present a couple of auxiliary lemmas. The following commutator identity for $u_{\varepsilon}: \Omega\times \varepsilon\mathbb{Z}^{d}\to \mathbb{R}$, obtained by direct computation, is practical:
\begin{equation}\label{commutator}
\widetilde{\mathcal{T}}_{\varepsilon}\nabla^{\varepsilon}u_{\varepsilon}-\nabla^{\varepsilon}\widetilde{\mathcal{T}}_{\varepsilon}u_{\varepsilon}=\frac{1}{\varepsilon}D\widetilde{\mathcal{T}}_{\varepsilon}u_{\varepsilon}+(D_1\nabla^{\varepsilon}_1,...,D_d\nabla^{\varepsilon}_d)\widetilde{\mathcal{T}}_{\varepsilon}u_{\varepsilon}.
\end{equation} 
\begin{lemma}\label{compactness} Consider a sequence $u_{\varepsilon}\in L^{p}(\Omega \times \varepsilon \mathbb{Z}^d)$. Suppose that $u_{\varepsilon}\overset{2}{\rightharpoonup} U$ in $L^{p}(\Omega\times \mathbb{R}^d)$ and $\varepsilon \nabla^{\varepsilon}u_{\varepsilon}\overset{2}{\rightharpoonup} 0$ in $L^{p}(\Omega\times \mathbb{R}^d)^d$. Then $U \in L^p_{\mathsf{inv}}(\Omega)\otimes L^p(\mathbb{R}^d)$.
\end{lemma}
\begin{proof}
Since $L^{p}_{\mathsf{inv}}(\Omega)=\brac{\mathsf{ran} D^*}^{\bot}$, it suffices to show that 
\begin{equation}\label{claim1}
\expect{\int_{\mathbb{R}^d}U(\omega,x)D^*_i\varphi(\omega)\eta(x)dx}=0
\end{equation}
for any $\varphi\in L^q(\Omega)$, $\eta\in C^{\infty}_c(\mathbb{R}^d)$ and $i\in \cb{1,...,d}$.

We consider the sequence $v_{\varepsilon}=\mathcal{F}_{\varepsilon}(\varphi \eta)\in L^{q}(\Omega \times \varepsilon \mathbb{Z}^d)$ and by Lemma \ref{Propfold} (iii), we have $v_{\varepsilon}\overset{2}{\rightarrow} \varphi \eta$ in $L^{q}(\Omega\times \mathbb{R}^d)$. Therefore, using Lemma \ref{basicfacts} (iv), we obtain
\begin{eqnarray}\label{firstconv}
& & \varepsilon \expect{\int_{\varepsilon\mathbb{Z}^{d}} u_{\varepsilon}(\omega,x)  \nabla_i^{\varepsilon,*} v_{\varepsilon}(\omega,x)dm\e(x)}\nonumber \\ &=& \expect{\int_{\varepsilon\mathbb{Z}^{d}} (\varepsilon \nabla^{\varepsilon}_i u_{\varepsilon}(\omega,x)) v_{\varepsilon}(\omega,x)dm\e(x)}\rightarrow 0 \quad \text{as } \varepsilon\to 0.
\end{eqnarray}
Moreover, using the definition of $\mathcal{F}\e$,
\begin{align}\label{secondpart}
\begin{split}
\varepsilon \nabla_i^{\varepsilon,*} v_{\varepsilon}(\omega,x) & =\varphi(T_{\frac{x}{\varepsilon}-e_i}\omega)\pi_{\varepsilon}\eta(x-\varepsilon e_i)-\varphi(T_{\frac{x}{\varepsilon}}\omega)\pi_{\varepsilon}\eta(x) \\
& =\varepsilon \varphi(T_{\frac{x}{\varepsilon}-e_i}\omega) \nabla_i^{\varepsilon,*} \pi_{\varepsilon}\eta(x)+D^*_i\varphi(T_{\frac{x}{\varepsilon}}\omega) \pi_{\varepsilon}\eta(x) ,
\end{split}
\end{align}
which implies $\varepsilon \nabla_i^{\varepsilon,*} v_{\varepsilon} \overset{2}{\rightarrow}  D^*_i\varphi \eta$ in $L^{q}(\Omega\times \mathbb{R}^d)$.
Indeed, the first term on the right-hand side of (\ref{secondpart}) vanishes in the strong two-scale limit since $\eta$ is compactly supported and smooth. The second term strongly two-scale converges to $ D^*_i\varphi \eta$.
This and Lemma \ref{basicfacts} (iv) imply
\begin{equation*}
\lim_{\varepsilon\rightarrow0}\varepsilon \expect{\int_{\varepsilon\mathbb{Z}^{d}} u_{\varepsilon}(\omega,x)  \nabla_i^{\varepsilon,*} v_{\varepsilon}(\omega,x) dm\e(x)}=\expect{\int_{\mathbb{R}^d} U(\omega,x) D^*_i\varphi(\omega) \eta(x) dx},
\end{equation*}
which, together with (\ref{firstconv}), yields (\ref{claim1}).
\end{proof}
\begin{lemma}\label{lemma4}
Let $u_{\varepsilon}\in L^{p}(\Omega \times \varepsilon \mathbb{Z}^d)$ satisfy
\begin{equation*}
\limsup_{\varepsilon\rightarrow 0}\expect{\int_{\varepsilon\mathbb{Z}^{d}}|u_{\varepsilon}(\omega,x)|^{p}+|\nabla^{\varepsilon}u_{\varepsilon}(\omega,x)|^pdm_{\varepsilon}(x)}<\infty.
\end{equation*}
Then there exists $U \in L^p_{\mathsf{inv}}(\Omega)\otimes W^{1,p}(\mathbb{R}^d)$ such that, up to a subsequence, 
\begin{equation*}
u_{\varepsilon}\overset{2}{\rightharpoonup} U,\quad P_{\mathsf{inv}} u_{\varepsilon}\overset{2}{\rightharpoonup} U \text{ in } L^{p}(\Omega\times \mathbb{R}^d), \quad \nabla^{\varepsilon} P_{\mathsf{inv}} u_{\varepsilon} \overset{2}{\rightharpoonup} \nabla U \text{ in } L^{p}(\Omega\times \mathbb{R}^d)^d.
\end{equation*}
\begin{proof}
Step 1. We claim that $\widetilde{\mathcal{T}}_{\varepsilon}\circ P_{\mathsf{inv}}=P_{\mathsf{inv}}\circ \widetilde{\mathcal{T}}_{\varepsilon}=P_{\mathsf{inv}}$. By shift invariance, we have $\widetilde{\mathcal{T}}_{\varepsilon}\circ P_{\mathsf{inv}}=P_{\mathsf{inv}}$. Hence, it suffices to prove $P_{\mathsf{inv}}\circ \widetilde{\mathcal{T}}_{\varepsilon}=P_{\mathsf{inv}}$. Let $\eta \in L^q(\varepsilon\mathbb{Z}^{d})$, $\varphi \in L^q(\Omega)$ and $v_{\varepsilon}\in L^{p}(\Omega \times \varepsilon \mathbb{Z}^d)$. We have
\begin{align*}
\expect{\int_{\varepsilon\mathbb{Z}^{d}} P_{\mathsf{inv}}\widetilde{\mathcal{T}}_{\varepsilon}v_{\varepsilon}(\omega,x)  \varphi(\omega) \eta(x) dm_{\varepsilon}(x)} & =\expect{\int_{\varepsilon\mathbb{Z}^{d}} v_{\varepsilon}(T_{-\frac{x}{\varepsilon}}\omega,x)  P^*_{\mathsf{inv}}\varphi(\omega) \eta(x) dm_{\varepsilon}(x) } \\ &=\expect{\int_{\varepsilon\mathbb{Z}^{d}} v_{\varepsilon}(\omega,x) P^*_{\mathsf{inv}}\varphi(\omega) \eta(x) dm_{\varepsilon}(x) }\\ &=\expect{\int_{\varepsilon\mathbb{Z}^{d}}  P_{\mathsf{inv}}v_{\varepsilon}(\omega,x)  \varphi(\omega) \eta(x) dm_{\varepsilon}(x)}.
\end{align*}
Above, in the second equality we use the fact that $P_{\mathsf{inv}}^* \simeq P_{\mathsf{inv}}$ on  $L^q(\Omega)$ and therefore $\widetilde{\mathcal{T}}\e^{-1} P_{\mathsf{inv}}^*\varphi=P_{\mathsf{inv}}^*\varphi$. Consequently, by a density argument it follows that $P_{\mathsf{inv}}\circ  \widetilde{\mathcal{T}}_{\varepsilon}=P_{\mathsf{inv}}$.
\smallskip

Step 2. Convergence of $P_{\mathsf{inv}} u\e$. Using boundedness of $P_{\mathsf{inv}}$ and the fact that $\nabla^{\varepsilon}$ and $P_{\mathsf{inv}}$ commute, we obtain
\begin{equation*}
\limsup_{\varepsilon\rightarrow 0}\expect{\int_{\varepsilon\mathbb{Z}^{d}}|P_{\mathsf{inv}}u_{\varepsilon}(\omega,x)|^p+|\nabla^{\varepsilon}P_{\mathsf{inv}}u_{\varepsilon}(\omega,x)|^p dm\e(x)}<\infty.
\end{equation*}
Applying Lemma \ref{basicfacts} (ii) and Lemma \ref{compactness}, there exist $V\in L^p_{\mathsf{inv}}(\Omega)\otimes L^p(\mathbb{R}^d)$ and $\widetilde{V}\in L^{p}(\Omega\times \mathbb{R}^d)^d$ such that
\begin{equation}\label{equation:652}
 P_{\mathsf{inv}} u_{\varepsilon}\overset{2}{\rightharpoonup} V \quad \text{in } L^{p}(\Omega\times \mathbb{R}^d), \quad 
 \nabla^{\varepsilon}P_{\mathsf{inv}} u_{\varepsilon}\overset{2}{\rightharpoonup} \widetilde{V} \quad \text{in } L^{p}(\Omega\times \mathbb{R}^d)^d,
\end{equation}
for a (not relabeled) subsequence. Note that, additionally, we have $\widetilde{V}\in L^p_{\mathsf{inv}}(\Omega)\otimes L^p(\mathbb{R}^d)^d$.

Let $\varphi\in L^q(\Omega)$ and $\eta \in C_c^{\infty}(\mathbb{R}^d)$ and denote $v_{\varepsilon}=\mathcal{F}_{\varepsilon}(\varphi \eta)$. Since $v\e \overset{2}{\to} \eta \varphi$ (Lemma \ref{Propfold} (iii)), for $i=1,...,d$, we have
\begin{equation}\label{equation:6521}
\expect{\int_{\varepsilon\mathbb{Z}^{d}} \nabla^{\varepsilon}_i P_{\mathsf{inv}} u_{\varepsilon }(\omega,x) v_{\varepsilon}(\omega,x)dm\e(x)}\rightarrow \expect{\int_{\mathbb{R}^d}\widetilde{V}_i(\omega,x) \varphi(\omega) \eta(x) dx} \quad \text{as }\varepsilon\to 0.
\end{equation}
On the other hand, it holds that
\begin{align}\label{equation:6522}
\begin{split}
\expect{\int_{\varepsilon\mathbb{Z}^{d}} \nabla^{\varepsilon}_i P_{\mathsf{inv}} u_{\varepsilon }(\omega,x) v_{\varepsilon}(\omega,x)dm\e(x)}=&\expect{\int_{\varepsilon\mathbb{Z}^{d}} P_{\mathsf{inv}} u_{\varepsilon }(\omega,x)\varphi(\omega) \nabla_i^{\varepsilon,*}\pi\e\eta(x) dm\e(x)}\\ & \overset{(\varepsilon\rightarrow 0)}{\rightarrow} -\expect{\int_{\mathbb{R}^d}V(\omega,x) \varphi(\omega) \partial_i \eta(x) dx}.
\end{split}
\end{align}
The above convergence is obtained using that $\overline{P_{\mathsf{inv}}u\e} \rightharpoonup V$ weakly in $L^p(\Omega\times \mathbb{R}^d)$ (follows from (\ref{equation:652})) and $\overline{\nabla^{\varepsilon,*}_i \pi_{\varepsilon} \eta} \to -\partial_i \eta$ strongly in $L^q(\mathbb{R}^d)$. The latter may be shown as follows. We have
\begin{eqnarray*}
&& \| \overline{\nabla^{\varepsilon,*}_i \pi_{\varepsilon} \eta} + \partial_i \eta \|_{L^q(\mathbb{R}^d)} \\ & \leq &\norm{ \overline{\pi}_{\varepsilon}\brac{\frac{\eta(\cdot -\varepsilon e_i)-\eta(\cdot)}{\varepsilon}}+ \overline{\pi}_{\varepsilon} \partial_i \eta }_{L^q(\mathbb{R}^d)} + \norm{\overline{\pi}_{\varepsilon}\partial_i \eta-\partial_i \eta}_{L^q(\mathbb{R}^d)} \\
& \leq & \norm{ \frac{\eta(\cdot -\varepsilon e_i)-\eta(\cdot)}{\varepsilon}+ \partial_i \eta }_{L^q(\mathbb{R}^d)} + \norm{\overline{\pi}_{\varepsilon}\partial_i \eta-\partial_i \eta}_{L^q(\mathbb{R}^d)},
\end{eqnarray*}
where we used that $\overline{\pi}_{\varepsilon}$ is a contraction. Since $\eta \in C^{\infty}_{c}(\mathbb{R}^d)$, it follows (by a Taylor expansion argument) that both terms on the right-hand side of the above inequality vanish in the limit $\varepsilon\to 0$.

Combining (\ref{equation:6521}) and (\ref{equation:6522}), we conclude that $V\in L^p_{\mathsf{inv}}(\Omega)\otimes W^{1,p}(\mathbb{R}^d)$ and $\nabla V= \widetilde{V}$. 
\smallskip

Step 3. We show that 
$u_{\varepsilon}\overset{2}{\rightharpoonup} V \text{ in } L^{p}(\Omega\times \mathbb{R}^d)$ (up to another subsequence).
Using Lemmas \ref{basicfacts} (ii) and \ref{compactness}, we conclude that there exist another subsequence (not relabeled) and $U \in L^p_{\mathsf{inv}}(\Omega)\otimes L^p(\mathbb{R}^d)$ such that $u_{\varepsilon}\overset{2}{\rightharpoonup} U \text{ in } L^{p}(\Omega\times \mathbb{R}^d)$. Since $P_{\mathsf{inv}} $ is a linear and bounded operator, it follows that $P_{\mathsf{inv}} (\mathcal{T}_{\varepsilon}u_{\varepsilon})\rightharpoonup P_{\mathsf{inv}} U  \text{ in } L^{p}(\Omega\times \mathbb{R}^d)$, and $P_{\mathsf{inv}} U=U$ by shift invariance of $U$. Furthermore, by Steps 1 and 2 we have that $P_{\mathsf{inv}} \mathcal{T}_{\varepsilon}u_{\varepsilon}=\mathcal{T}_{\varepsilon}P_{\mathsf{inv}} u_{\varepsilon}\rightharpoonup V$ and therefore $U=V$. This completes the proof.
\end{proof}
\end{lemma}
\begin{proof}[Proof of Proposition \ref{comp3}] (i)
By Lemma \ref{lemma4} we deduce that there exists $U\in L^p_{\mathsf{inv}}(\Omega)\otimes W^{1,p}(\mathbb{R}^d)$ such that $u_{\varepsilon}\overset{2}{\rightharpoonup} U$ and by boundedness of $\nabla^{\varepsilon}u\e$ (Lemma \ref{basicfacts} (ii)) there exists $V\in L^{p}(\Omega\times \mathbb{R}^d)^d$ such that $\nabla^{\varepsilon}u_{\varepsilon}\overset{2}{\rightharpoonup} V$ (up to a subsequence).
In order to prove the claim, it suffices to show that
\begin{equation}\label{ortho}
\expect{\int_{\mathbb{R}^d} V(\omega,x)\cdot \eta(x)\varphi(\omega)dx}=\expect{\int_{\mathbb{R}^d} \nabla U(\omega,x)\cdot \eta(x)\varphi(\omega)dx}
\end{equation}
for any $\varphi \in L^q(\Omega)^d$ with $D^*\varphi=0$ and $\eta\in C_c^{\infty}(\mathbb{R}^d)$. Indeed, this implies that $\chi:=V-\nabla U\in \mathbf{L}^p_{\mathsf{pot}}(\Omega)\otimes L^p(\mathbb{R}^d)$ and thus the claim of the proposition.

For the argument consider $v_{\varepsilon}=\mathcal{F}_{\varepsilon}(\eta \varphi)$, the folding acting componentwise. Since $v_{\varepsilon} \overset{2}{\rightarrow} \eta \varphi$ (Lemma \ref{Propfold} (iii)),  
\begin{equation*}
\expect{\int_{\varepsilon\mathbb{Z}^{d}}\nabla^{\varepsilon} u_{\varepsilon}(\omega,x) \cdot v_{\varepsilon}(\omega,x) dm\e(x)}\rightarrow \expect{\int_{\mathbb{R}^d} V(\omega,x) \cdot \eta(x) \varphi(\omega) dx} \quad \text{as $\varepsilon\to 0$}.
\end{equation*}
On the other hand, the commutator identity (\ref{commutator}) and the definition of $\mathcal{F}\e$ yield
\begin{multline*}
 \expect{\int_{\varepsilon\mathbb{Z}^{d}}\nabla^{\varepsilon}u_{\varepsilon}(\omega,x) \cdot v_{\varepsilon}(\omega,x) dm\e(x)} \\
=  \bigg\langle \int_{\varepsilon\mathbb{Z}^{d}} \brac{ \nabla^{\varepsilon}\widetilde{\mathcal{T}}_{\varepsilon}u_{\varepsilon}(\omega,x)+\frac{1}{\varepsilon}D\widetilde{\mathcal{T}}_{\varepsilon}u_{\varepsilon}(\omega,x)+(D_1\nabla^{\varepsilon}_{1},...,D_d\nabla^{\varepsilon}_{d})\widetilde{\mathcal{T_{\varepsilon}}}u_{\varepsilon}(\omega,x)}\\ \cdot \pi_{\varepsilon} \eta(x) \varphi(\omega) dm\e(x)\bigg\rangle.
\end{multline*}
Since $D^*\varphi=0$, the contribution from the second term on the right-hand side above vanishes. After a discrete integration by parts, the right-hand side reduces to 
\begin{align*}
& \sum_{i=1}^{d}\expect{\int_{\varepsilon\mathbb{Z}^{d}} \brac{\widetilde{\mathcal{T}}_{\varepsilon}u_{\varepsilon}(\omega,x)+D_i\widetilde{\mathcal{T}}_{\varepsilon}u_{\varepsilon}(\omega,x)} \nabla_i^{\varepsilon,*} \pi_{\varepsilon}\eta(x) \varphi_i(\omega) dm\e(x)}\\ & {\rightarrow} -\sum_{i=1}^{d}\expect{\int_{\mathbb{R}^d} \brac{ U(\omega,x)+D_iU(\omega,x)} \partial_i\eta(x) \varphi_i(\omega) dx} \quad \text{as }\varepsilon\to 0,
\end{align*}
which is concluded by using that $u_{\varepsilon}\overset{2}{\rightharpoonup} U$ and that $\eta$ is smooth and compactly supported. Since $U$ is shift-invariant, the second term on the right-hand side vanishes. After an integration by parts, we are able to infer (\ref{ortho}) and conclude the proof of part (i).
\smallskip

(ii) By Lemma \ref{basicfacts} (ii), there exists $U\in L^p(\Omega\times \mathbb{R}^d)$ such that $u\e \overset{2}{\rightharpoonup} U$ (up to a subsequence). Since $\gamma \in (0,1)$, $u\e$ satisfies the assumptions in Lemma \ref{compactness} and therefore $U \in L^p_{\mathsf{inv}}(\Omega)\otimes L^p(\mathbb{R}^d)$. With the help of (i), we obtain that for the sequence $v\e:=\varepsilon^{\gamma} u\e$, there exist $V \in  L^p_{\mathsf{inv}}(\Omega)\otimes W^{1,p}(\mathbb{R}^d)$ and $\chi \in \mathbf{L}^p_{\mathsf{pot}}(\Omega)\otimes L^p(\mathbb{R}^d)$ such that (up to another subsequence)
\begin{equation*}
v\e\overset{2}{\rightharpoonup} V \text{ in }L^p(\Omega\times \mathbb{R}^d), \quad \nabla^{\varepsilon} v\e \overset{2}{\rightharpoonup} \nabla V+ \chi \text{ in }L^p(\Omega\times \mathbb{R}^d)^d.
\end{equation*} 
However, using that $u\e \overset{2}{\rightharpoonup} U$, we conclude that $V=0$ and the claim is proved.
\smallskip

(iii) Lemma \ref{basicfacts} (ii) implies that there exist $U \in L^p(\Omega\times \mathbb{R}^d)$ and $V\in L^p(\Omega\times \mathbb{R}^d)^d$ such that (up to a subsequence)
\begin{equation*}
u\e \overset{2}{\rightharpoonup} U \text{ in }L^p(\Omega\times \mathbb{R}^d),\quad \varepsilon \nabla^{\varepsilon}u\e \overset{2}{\rightharpoonup} V \text{ in }L^p(\Omega\times \mathbb{R}^d)^d.
\end{equation*}
Following the same strategy as in Lemma \ref{compactness} it can be obtained that $V=DU$.
\smallskip

(iv) Lemma \ref{basicfacts} (ii) implies that there exists $U \in L^p(\Omega\times \mathbb{R}^d)$ such that $u\e\overset{2}{\rightharpoonup} U$ in $L^p(\Omega\times \mathbb{R}^d)$ (up to a subsequence). Also, using part (iii), for the sequence $v\e:=\varepsilon^{\gamma-1}u\e$, there exists $V\in  L^p(\Omega\times \mathbb{R}^d)$ such that (up to another subsequence)
\begin{equation*}
v\e \overset{2}{\rightharpoonup} V, \quad \varepsilon \nabla^{\varepsilon}v\e \overset{2}{\rightharpoonup} DV.
\end{equation*}
The fact that $u\e\overset{2}{\rightharpoonup} U$ implies that $V=0$ and the proof is complete.
\end{proof}
\smallskip

\begin{proof}[Proof of Corollary \ref{ergodic}]
(i) The claim follows directly from Lemmas \ref{lemma4} and \ref{lemma_equivalent_conv}.
\smallskip

(ii) Exploiting linearity and boundedness of $P_{\mathsf{inv}}$ and Step 1 in the proof of Lemma \ref{lemma4}, we obtain that
\begin{align*}
\overline{\expect{u\e}}=P_{\mathsf{inv}}\mathcal{T}\e u\e \rightharpoonup P_{\mathsf{inv}}U=U,\quad \overline{\expect{\varepsilon^{\gamma}\nabla^{\varepsilon}u\e}}=P_{\mathsf{inv}}\mathcal{T}\e \varepsilon^{\gamma}\nabla^{\varepsilon}u\e\rightharpoonup P_{\mathsf{inv}}\chi=\expect{\chi}.
\end{align*}
The above, Lemma \ref{lemma_equivalent_conv}, and the fact that $\expect{\chi}=0$ allow us to conclude the proof.
\smallskip

(iii) Lemma \ref{lemma4} implies that $\overline{\expect{u\e}}\rightharpoonup u$ and $\overline{\nabla^{\varepsilon}\expect{u\e}}\rightharpoonup \nabla u$ weakly in $L^p(\mathbb{R}^d)$. Lemma \ref{lemma_equivalent_conv} implies that $\widehat{\expect{u\e}}\rightharpoonup u$ weakly in $L^p(\mathbb{R}^d)$. Furthermore, for any $\eta \in L^q(\mathbb{R}^d)$ it holds that
\begin{equation*}
\int_{\mathbb{R}^d}\brac{\nabla \widehat{\expect{u\e}}(x)-\overline{\nabla^{\varepsilon}\expect{u\e}}(x)}\eta(x)dx \rightarrow 0 \text{ as }\varepsilon\to 0.
\end{equation*}
As a result of this, $\widehat{\expect{u\e}}\rightharpoonup u$ weakly in $W^{1,p}(\mathbb{R}^d)$. Rellich's embedding theorem implies that $\widehat{\expect{u\e}} \to U$ strongly in $L^p(Q)$ and using Lemma \ref{lemma_equivalent_conv} we conclude that $\overline{\expect{u\e}} \rightarrow U$ strongly in $L^p(Q)$.
\smallskip

(iv) We have by Jensen's inequality and boundedness of $\varphi$
\begin{equation*}
\int_{\mathbb{R}^d}|\expect{\overline{u}\e(\omega,x) \varphi(\omega)}-\expect{\varphi(\omega)}U(x)|^pdx\leq C \expect{\int_{\mathbb{R}^d}|\overline{u}\e(\omega,x)-U(x)|^p dx}.
\end{equation*}
The right-hand side of the above inequality equals $\expect{\int_{\mathbb{R}^d}|\mathcal{T}_{\varepsilon} u\e(\omega,x)-U(x)|^p dx}$ and therefore it vanishes as $\varepsilon\to 0$.
\end{proof}
\smallskip

Before presenting the proof of Proposition \ref{prop1}, we provide an auxiliary lemma providing a nonlinear approximation for $\chi$ in the case $\gamma=0$. 
\begin{lemma}[Nonlinear approximation]\label{lemmaaa}
 For $\chi \in \mathbf{L}^p_{\mathsf{pot}}(\Omega)\otimes
    L^p(\mathbb{R}^d)$ and $\delta>0$, there exists a sequence
    $g_{\delta,\varepsilon}\in L^{p}(\Omega \times \varepsilon \mathbb{Z}^d)$ such that
    \begin{equation*}
       \|g_{\delta,\varepsilon}\|_{L^{p}(\Omega \times \varepsilon \mathbb{Z}^d)}\leq \varepsilon C(\delta),  \quad
       \limsup_{\varepsilon\to 0}\|\mathcal{T}_{\varepsilon}
      \nabla^{\varepsilon}g_{\delta,\varepsilon}-\chi\|_{L^{p}(\Omega\times \mathbb{R}^d)^d}\leq
      \delta.
    \end{equation*}
\end{lemma}
\begin{proof}
 Let $\chi \in \mathbf{L}^p_{\mathsf{pot}}(\Omega)\otimes L^p(\mathbb{R}^d)$ and $\delta>0$ be fixed. By density, there exists $V=\sum_{j=1}^{n} \varphi_j \eta_j$ with $\varphi_j \in L^p(\Omega)$, $\eta_j \in C_c^{\infty}(\mathbb{R}^d)$ and
\begin{equation*}
\|DV-\chi\|_{L^{p}(\Omega\times \mathbb{R}^d)^d}\leq \delta.
\end{equation*}
We define $g_{\varepsilon}:=\varepsilon \mathcal{F}\e V$ and remark that $\|g_{\varepsilon}\|_{L^{p}(\Omega \times \varepsilon \mathbb{Z}^d)}\leq \varepsilon \|V\|_{L^p(\Omega\times \mathbb{R}^d)}$, which follows from the boundedness of $\mathcal{F}\e$. This proves the first part.

Note that  $\nabla^{\varepsilon}g_{\varepsilon}(\omega,x)=D\pi\e V(T_{\frac{x}{\varepsilon}}\omega,x)+\varepsilon \nabla^{\varepsilon}\pi\e V(T_{\frac{x}{\varepsilon}}\omega,x)$ and therefore we obtain $\mathcal{T}_{\varepsilon} \nabla^{\varepsilon}g_{\varepsilon}(\omega,x)=\overline{\pi}_{\varepsilon}DV(\omega,x)+\varepsilon \overline{\nabla^{\varepsilon}\pi_{\varepsilon}V}(\omega,x)$. Hence 
\begin{eqnarray*}
&& \|\mathcal{T}_{\varepsilon} \nabla^{\varepsilon}g_{\varepsilon}-\chi\|_{L^{p}(\Omega\times \mathbb{R}^d)^d}\\ & \leq & \|\overline{\pi}\e DV-DV \|_{L^{p}(\Omega\times \mathbb{R}^d)^d}+\|DV-\chi\|_{L^{p}(\Omega\times \mathbb{R}^d)^d} +\varepsilon\|\nabla^{\varepsilon}\pi\e V\|_{L^{p}(\Omega \times \varepsilon \mathbb{Z}^d)^d}. 
\end{eqnarray*}
The first and last terms on the right-hand side above vanish as $\varepsilon\to 0$ and therefore the claim follows (since we can choose $\delta$ arbitrarily small). Indeed, for the first term it is sufficient to note that $DV$ is smooth and has compact support w.r.t. its $x$-variable. Also, the last term vanishes thanks to the boundedness of $\pi\e$ and the boundedness of difference quotients by gradients; specifically (for $i=1,...,d$)
\begin{align*}
\varepsilon^{p}\|\nabla_i^{\varepsilon}\pi\e V\|^p_{L^{p}(\Omega \times \varepsilon \mathbb{Z}^d)^d} & \leq C \varepsilon^p \expect{\int_{\mathbb{R}^d}\Big| \frac{V(\omega,x+\varepsilon e_i)-V(\omega,x)}{\varepsilon}\Big|^p dx}\\ & \leq  C\varepsilon^p \|\nabla V\|^p_{L^{p}(\Omega\times \mathbb{R}^d)^d}.
\end{align*}
\end{proof}
\smallskip

\begin{proof}[Proof of Proposition \ref{prop1}] 
In the following proof we appeal to discrete maximal $L^p$-regularity for the equation
\begin{equation*}
  \lambda u+\nabla^{\varepsilon,*}\nabla^{\varepsilon} u=\nabla^*F+g \quad \text{in }\varepsilon\mathbb{Z}^{d}, \quad (\text{for some }F\in L^p(\varepsilon\mathbb{Z}^{d})^d, g\in L^p(\varepsilon\mathbb{Z}^{d}))
\end{equation*}
in the form of
\begin{equation*}
  {\lambda}^\frac{1}{2}\|u\|_{L^p(\varepsilon\mathbb{Z}^{d})}+\|\nabla^{\varepsilon} u\|_{L^p(\varepsilon\mathbb{Z}^{d})^d}\leq C(d,p)\Big(\|F\|_{L^p(\varepsilon\mathbb{Z}^{d})^d}+{\lambda}^{-\frac{1}{2}}\|g\|_{L^p(\varepsilon\mathbb{Z}^{d})}\Big)
\end{equation*}
which is uniform in $\varepsilon$. For $p=2$ this is a standard a priori estimate. For $1<p<\infty$, in the continuum setting this is a classical result (see e.g. \cite[Chapter 4, Sec. 4, Theorem 2]{krylov2008lectures}), and follows from the Calder\'on-Zygmund estimate $\|\partial_{ij}u\|_{L^p(\mathbb{R}^d)}\leq C(d,p)\|\triangle u\|_{L^p(\mathbb{R}^d)}$. The estimate above follows by the same argument from the Calder\'on-Zygmund estimate for the discrete Laplacian on $\varepsilon\mathbb{Z}^{d}$, for the latter see e.g. \cite{gloria2015quantification, ben2017moment}.
\smallskip

(i) Let $2 \gamma <\alpha < 2$. For a given $\chi \in \mathbf{L}^p_{\mathsf{pot}}(\Omega)\otimes L^p(\mathbb{R}^d)$ we define $\mathcal{G}_{\varepsilon}^{\gamma}\chi:=u\e$ as the unique solution to the following equation in $L^{p}(\Omega \times \varepsilon \mathbb{Z}^d)$ (for $P$-a.e. $\omega\in\Omega$)
\begin{equation*}
\varepsilon^{-\alpha} u\e +\nabla^{\varepsilon,*}\nabla^{\varepsilon}u\e =\nabla^{\varepsilon,*}\varepsilon^{-\gamma} \mathcal{F}\e \chi \quad \text{in }\varepsilon\mathbb{Z}^{d}.
\end{equation*}
The discrete maximal $L^p$-regularity theory implies that
\begin{equation*}
\varepsilon^{-\frac{\alpha}{2}}\| u_{\varepsilon} \|_{L^{p}(\Omega \times \varepsilon \mathbb{Z}^d)}+\|\nabla^{\varepsilon}u_{\varepsilon}\|_{L^{p}(\Omega \times \varepsilon \mathbb{Z}^d)^d} \leq \varepsilon^{-\gamma} C \|\mathcal{F}\e \chi \|_{L^p(\Omega\times \varepsilon\mathbb{Z}^{d})^d}.
\end{equation*} 
As a result of this, we have $\| u_{\varepsilon} \|_{L^{p}(\Omega \times \varepsilon \mathbb{Z}^d)}\leq \varepsilon^{\frac{\alpha}{2}-\gamma}C\|\chi\|_{L^p(\Omega\times \mathbb{R}^d)^d}$ and therefore $u\e \overset{2}{\to} 0$ as $\varepsilon\to 0$.

We consider the sequence $g_{\delta,\varepsilon}$ from Lemma \ref{lemmaaa} corresponding to $\chi$.
Note that $w_{\delta,\varepsilon}:=u\e-\varepsilon^{-\gamma}g_{\delta,\varepsilon}$ is the unique solution in $L^{p}(\Omega \times \varepsilon \mathbb{Z}^d)$ to (for $P$-a.e. $\omega\in \Omega$)
\begin{equation*}
\varepsilon^{-\alpha} w_{\delta,\varepsilon}+\nabla^{\varepsilon,*}\nabla^{\varepsilon} w_{\delta,\varepsilon}=\nabla^{\varepsilon,*}\varepsilon^{-\gamma}(\mathcal{F}\e\chi-\nabla^{\varepsilon}g_{\delta,\varepsilon})-\varepsilon^{-\alpha-\gamma}g_{\delta,\varepsilon}\text{ in } \varepsilon\mathbb{Z}^{d}.
\end{equation*}
We employ again the discrete maximal $L^p$-regularity theory to obtain 
\begin{equation*}
\|\nabla^{\varepsilon}w_{\delta,\varepsilon}\|_{L^{p}(\Omega \times \varepsilon \mathbb{Z}^d)^d}\leq C \brac{\varepsilon^{-\gamma} \|\mathcal{F}\e \chi- \nabla^{\varepsilon}g_{\delta,\varepsilon}\|_{L^{p}(\Omega \times \varepsilon \mathbb{Z}^d)^d}+\varepsilon^{-\frac{\alpha}{2}-\gamma}\|g_{\delta,\varepsilon}\|_{L^{p}(\Omega \times \varepsilon \mathbb{Z}^d)}}.
\end{equation*}
Multiplication of the above inequality by $\varepsilon^{\gamma}$ yields
\begin{equation*}
\|\varepsilon^{\gamma}\nabla^{\varepsilon}u\e- \nabla^{\varepsilon}g_{\delta,\varepsilon}\|_{L^{p}(\Omega \times \varepsilon \mathbb{Z}^d)^d}\leq C \brac{\|\mathcal{F}\e \chi- \nabla^{\varepsilon}g_{\delta,\varepsilon}\|_{L^{p}(\Omega \times \varepsilon \mathbb{Z}^d)^d}+\varepsilon^{-\frac{\alpha}{2}}\|g_{\delta,\varepsilon}\|_{L^{p}(\Omega \times \varepsilon \mathbb{Z}^d)}}.
\end{equation*}
As a result of this and with help of the isometry property of $\mathcal{T}_{\varepsilon}$, we obtain
\begin{eqnarray*}
&& \|\mathcal{T}_{\varepsilon} \varepsilon^{\gamma}\nabla^{\varepsilon}u\e-\chi\|_{L^{p}(\Omega\times \mathbb{R}^d)^d} \\ 
& \leq &  C \brac{\|\mathcal{F}\e \chi- \nabla^{\varepsilon}g_{\delta,\varepsilon}\|_{L^{p}(\Omega \times \varepsilon \mathbb{Z}^d)^d}+\varepsilon^{-\frac{\alpha}{2}}\|g_{\delta,\varepsilon}\|_{L^{p}(\Omega \times \varepsilon \mathbb{Z}^d)}+\| \mathcal{T}_{\varepsilon} \nabla^{\varepsilon}g_{\delta,\varepsilon}-\chi \|_{L^{p}(\Omega\times \mathbb{R}^d)^d}}.
\end{eqnarray*}
Letting first $\varepsilon\to 0$ and then $\delta \to 0$, the right-hand side of the above inequality vanishes using Lemma \ref{lemmaaa}. This completes the proof of (i).
\smallskip

(ii)  We consider a sequence $U_{\delta}=\sum_{i=1}^{n(\delta)} \varphi_i^{\delta}\eta_i^{\delta}$ such that $\varphi_{i}^{\delta} \in L^p_{\mathsf{inv}}(\Omega)$, $\eta_i^{\delta}\in C^{\infty}_c(\mathbb{R}^d)$, and
\begin{equation*}
\|U_{\delta}-U\|_{L^{p}(\Omega\times \mathbb{R}^d)}+\|\nabla U_{\delta}-\nabla U\|_{L^{p}(\Omega\times \mathbb{R}^d)^d}\rightarrow 0 \quad \text{as }\delta \to 0.
\end{equation*}  
Using the triangle inequality, it follows that
\begin{multline}\label{expression}
 \|\mathcal{T}_{\varepsilon} \nabla^{\varepsilon}\mathcal{F}\e U-\nabla U\|_{L^{p}(\Omega\times \mathbb{R}^d)^d} \leq  \|\mathcal{T}_{\varepsilon} \nabla^{\varepsilon} \mathcal{F}\e U-\mathcal{T}_{\varepsilon} \nabla^{\varepsilon} \mathcal{F}\e U_{\delta}\|_{L^{p}(\Omega\times \mathbb{R}^d)^d} \\  +\|\mathcal{T}_{\varepsilon} \nabla^{\varepsilon} \mathcal{F}\e U_{\delta}-\nabla U_{\delta}\|_{L^{p}(\Omega\times \mathbb{R}^d)^d} +\|\nabla U_{\delta}-\nabla U\|_{L^{p}(\Omega\times \mathbb{R}^d)^d}.
\end{multline}
First, we treat the first term on the right-hand side. For $i=1,...,d$, by the isometry property of $\mathcal{T}_{\varepsilon}$ and contraction property of $\mathcal{F}\e$,
\begin{eqnarray}\label{inequa1111}
&& \|\mathcal{T}_{\varepsilon} \nabla^{\varepsilon}_i \mathcal{F}\e U-\mathcal{T}_{\varepsilon} \nabla^{\varepsilon}_i \mathcal{F}\e U_{\delta}\|_{L^p(\Omega\times \mathbb{R}^d)}^p \nonumber \\ & \leq & \expect{\int_{\mathbb{R}^d} \Big|\frac{U(\omega,x+\varepsilon e_i)-U_{\delta}(\omega,x+\varepsilon e_i)-U(\omega,x)+U_{\delta}(\omega,x)}{\varepsilon}\Big|^p dx} \nonumber \\ & \leq & C \expect{\int_{\mathbb{R}^d}|\partial_i U(\omega,x)-\partial_i U_{\delta}(\omega,x)|^p dx}.
\end{eqnarray}
The last inequality follows using the fact that for any function $\eta \in W^{1,p}(\mathbb{R}^d)$, we have $\eta(x+\varepsilon e_i)-\eta(x)=\varepsilon \int_{0}^1 \partial_i \eta(x+\varepsilon t e_i)dt$ and therefore $\int_{\mathbb{R}^d}|\frac{\eta(x+\varepsilon e_i)-\eta(x)}{\varepsilon}|^pdx \leq \int_0^1 \int_{\mathbb{R}^d}|\partial_i \eta(x+\varepsilon t e_i)|^p dx dt = \int_{\mathbb{R}^d}|\partial_i \eta(x)|^p dx$. Second, we compute ($i=1,...,d$)
\begin{eqnarray}\label{equation:841}
&& \mathcal{T}_{\varepsilon} \nabla^{\varepsilon}_i \mathcal{F}\e U_{\delta}(\omega,x)  \\ &=&\frac{1}{\varepsilon}\brac{\overline{\pi}_{\varepsilon} U_{\delta}(T_{e_i} \omega , x+\varepsilon e_i)- \overline{\pi} \e U_{\delta}(T_{e_i}\omega,x)}+\frac{1}{\varepsilon}\brac{\overline{\pi}\e U_{\delta}(T_{e_i}\omega,x)-\overline{\pi}\e U_{\delta}(\omega,x)}.\nonumber
\end{eqnarray}
The second part of the right-hand side of the above equality vanishes (for $P$-a.e. $\omega\in \Omega$) by shift invariance of $U_{\delta}$. Further, we have
\begin{eqnarray*}
&& \|\mathcal{T}_{\varepsilon} \nabla^{\varepsilon}_i \mathcal{F}\e U_{\delta}-\partial_i U_{\delta}\|_{L^{p}(\Omega\times \mathbb{R}^d)} \\ &\leq & \expect{\int_{\mathbb{R}^d}|\frac{\overline{\pi}_{\varepsilon} U_{\delta}(\omega , x+\varepsilon e_i)- \overline{\pi} \e U_{\delta}(\omega,x)}{\varepsilon}-\partial_i U_{\delta}(\omega,x)|^p dx}^{\frac{1}{p}}. 
\end{eqnarray*}
For any $\delta>0$ the last expression converges to 0 as $\varepsilon\rightarrow 0$ since $U_{\delta}$ is smooth in its $x$-variable.
Finally, in (\ref{expression}) we first let $\varepsilon\rightarrow 0$ and then $\delta \rightarrow 0$ to conclude the proof.
\smallskip

(iii) Let $0<\alpha<2\gamma$. For a given $U \in L^p_{\mathsf{inv}}(\Omega)\otimes L^p(\mathbb{R}^d)$ we define $\mathcal{F}\e^{\gamma}U:=u\e$ as the unique solution to the following equation in $L^p(\Omega\times \varepsilon\mathbb{Z}^{d})$ (for $P$-a.e. $\omega\in \Omega$)
\begin{equation*}
\varepsilon^{-\alpha}u\e+ \nabla^{\varepsilon,*}\nabla^{\varepsilon}u\e=\varepsilon^{-\alpha}\mathcal{F}\e U \quad \text{in }\varepsilon\mathbb{Z}^{d}.
\end{equation*} 
The maximal $L^p$-regularity theory and boundedness of $\mathcal{F}\e$ imply that
\begin{equation*}
\|\nabla^{\varepsilon}u\e\|_{L^p(\Omega\times \varepsilon\mathbb{Z}^{d})^d}\leq \varepsilon^{-\frac{\alpha}{2}} C \|U\|_{L^p(\Omega\times \mathbb{R}^d)}.
\end{equation*}
As a result of this and the isometry property of $\mathcal{T}_{\varepsilon}$, we obtain that $\varepsilon^{\gamma}\nabla^{\varepsilon}u\e \overset{2}{\to} 0$.

We consider a sequence $U_{\delta}=\sum_{i=1}^{n(\delta)} \varphi_i^{\delta}\eta_i^{\delta}$ such that $\varphi_{i}^{\delta} \in L^p_{\mathsf{inv}}(\Omega)$, $\eta_i^{\delta}\in C^{\infty}_c(\mathbb{R}^d)$, and
\begin{equation*}
\|U_{\delta}-U\|_{L^{p}(\Omega\times \mathbb{R}^d)} \rightarrow 0 \text{ as }\delta \to 0.
\end{equation*}  
Note that $w_{\delta,\varepsilon}:=u\e-\mathcal{F}\e U_{\delta}$ is the unique solution in $L^p(\Omega\times \varepsilon\mathbb{Z}^{d})$ to (for $P$-a.e. $\omega\in \Omega$)
\begin{equation*}
\varepsilon^{-\alpha}w_{\delta,\varepsilon}+\nabla^{\varepsilon,*}\nabla^{\varepsilon} w_{\delta,\varepsilon}=\varepsilon^{-\alpha}\brac{\mathcal{F}\e U-\mathcal{F}\e U_{\delta}}-\nabla^{\varepsilon,*}\nabla^{\varepsilon}\mathcal{F}\e U_{\delta} \quad \text{in }\varepsilon\mathbb{Z}^{d}.
\end{equation*}
The maximal $L^p$-regularity theory implies that
\begin{equation*}
\varepsilon^{-\frac{\alpha}{2}}\|w_{\delta,\varepsilon}\|_{L^p(\Omega\times \varepsilon\mathbb{Z}^{d})}\leq C \brac{ \varepsilon^{-\frac{\alpha}{2}} \|\mathcal{F}\e U- \mathcal{F}\e U_{\delta} \|_{L^p(\Omega\times \varepsilon\mathbb{Z}^{d})}+\|\nabla^{\varepsilon}\mathcal{F}\e U_{\delta}\|_{L^p(\Omega \times \varepsilon\mathbb{Z}^{d})^d}}.
\end{equation*}
We multiply the above inequality by $\varepsilon^{\frac{\alpha}{2}}$ and use boundedness of $\mathcal{F}\e$, to obtain
\begin{equation*}
\|u\e- \mathcal{F}\e U_{\delta}\|_{L^p(\Omega\times \varepsilon\mathbb{Z}^{d})}\leq C \brac{\|U-U_{\delta}\|_{L^p(\Omega \times \mathbb{R}^d)}+\varepsilon^{\frac{\alpha}{2}}\|\nabla^{\varepsilon}\mathcal{F}\e U_{\delta}\|_{L^p(\Omega\times \varepsilon\mathbb{Z}^{d})^d}}.
\end{equation*}
Using the above inequality and the isometry property of $\mathcal{T}_{\varepsilon}$, we obtain
\begin{eqnarray*}
&& \| \mathcal{T}\e u\e- U\|_{L^p(\Omega\times \mathbb{R}^d)} \\ & \leq & C \brac{\|U-U_{\delta}\|_{L^p(\Omega \times \mathbb{R}^d)}+\varepsilon^{\frac{\alpha}{2}}\|\nabla^{\varepsilon}\mathcal{F}\e U_{\delta}\|_{L^p(\Omega\times \varepsilon\mathbb{Z}^{d})^d}+\|\mathcal{T}\e\mathcal{F}\e U_{\delta}-U_{\delta}\|_{L^p(\Omega \times \mathbb{R}^d)}}.
\end{eqnarray*}
Letting $\varepsilon\to 0$, the last two terms on the right-hand side of the above inequality vanish. Indeed, the middle term is bounded by $C \varepsilon^{\frac{\alpha}{2}}\|\nabla U_{\delta}\|_{L^p(\Omega\times \mathbb{R}^d)^d}$ (cf. part (ii) (\ref{inequa1111})) and the last term vanishes using Lemma \ref{Propfold} (iii). Finally, letting $\delta \to 0$ we conclude that $u\e \overset{2}{\to} U$.
\smallskip

(iv) We consider a sequence $U_{\delta}=\sum_{i=1}^{n(\delta)} \varphi_i^{\delta} \eta_i^{\delta}$ such that $\varphi_{i}^{\delta} \in L^p(\Omega)$, $\eta_i^{\delta}\in C^{\infty}_c(\mathbb{R}^d)$, and
\begin{equation*}
\|U_{\delta}-U\|_{L^{p}(\Omega\times \mathbb{R}^d)} \rightarrow 0 \text{ as }\delta \to 0.
\end{equation*}  
We have
\begin{multline}\label{equation:889}
 \| \mathcal{T}_{\varepsilon} \varepsilon^{\gamma} \nabla^{\varepsilon} \mathcal{F}\e U - a_{\gamma} DU \|_{L^p(\Omega\times \mathbb{R}^d)^d} \leq 
\| \mathcal{T}_{\varepsilon} \varepsilon^{\gamma} \nabla^{\varepsilon} \mathcal{F}\e (U-U_{\delta}) \|_{L^p(\Omega\times \mathbb{R}^d)^d} \\ +\|a_{\gamma} D \brac{U_{\delta} - U} \|_{L^p(\Omega\times \mathbb{R}^d)^d} +\| \mathcal{T}_{\varepsilon} \varepsilon^{\gamma} \nabla^{\varepsilon} \mathcal{F}\e U_{\delta}- a_{\gamma}DU_{\delta} \|_{L^p(\Omega\times \mathbb{R}^d)^d}.
\end{multline}
The first term on the right-hand side above is bounded by $\varepsilon^{\gamma-1} C \|U-U_{\delta}\|_{L^p(\Omega\times \mathbb{R}^d)}$ (using boundedness of all of the appearing operators). We compute, as in (\ref{equation:841}) (part (ii)), for $i=1,...,d$
\begin{equation*}
\mathcal{T}_{\varepsilon} \varepsilon^{\gamma} \nabla^{\varepsilon}_i \mathcal{F}\e U_{\delta}(\omega,x)=\varepsilon^{\gamma-1}\brac{\overline{\pi}_{\varepsilon} U_{\delta}(T_{e_i} \omega , x+\varepsilon e_i)- \overline{\pi} \e U_{\delta}(T_{e_i}\omega,x)}+ \varepsilon^{\gamma-1} \overline{\pi}\e D_i U_{\delta}(\omega,x).
\end{equation*}
As a result of this, we obtain 
\begin{eqnarray*}
& & \| \mathcal{T}_{\varepsilon} \varepsilon^{\gamma} \nabla^{\varepsilon}_i \mathcal{F}\e U_{\delta}- a_{\gamma}D_i U_{\delta} \|_{L^p(\Omega\times \mathbb{R}^d)} \\ & \leq & \varepsilon^{\gamma} \| \frac{\overline{\pi}\e U_{\delta}(\cdot, \cdot+\varepsilon e_i)-\overline{\pi}\e U_{\delta}(\cdot, \cdot)}{\varepsilon}\|_{L^p(\Omega\times \mathbb{R}^d)}+\|\varepsilon^{\gamma-1} \overline{\pi}\e D_i U_{\delta}-a_{\gamma}D_i U_{\delta}\|_{L^p(\Omega\times \mathbb{R}^d)}.
\end{eqnarray*}
The first term on the right-hand side above is bounded by $\varepsilon^{\gamma} C \|\nabla U_{\delta}\|_{L^p(\Omega\times \mathbb{R}^d)^d}$ and therefore it vanishes in the limit $\varepsilon\to 0$. The second term vanishes as well in the limit $\varepsilon\to 0$. 

Collecting the above claims and letting first $\varepsilon\to 0$ and then $\delta\to 0$ in (\ref{equation:889}), we conclude the proof.
\end{proof}
\smallskip

\begin{proof}[Proof of Corollary \ref{rem5}]
(i) and (ii) are obtained directly from Proposition \ref{prop1} and Lemma \ref{Propfold} (iii).
\smallskip

(iii) For $\delta>0$ we consider a cut-off function $\eta_{\delta} \in C_c^{\infty}(\mathbb{R}^d)$ such that $0\leq \eta_{\delta} \leq 1$, $\eta_{\delta}=0$ in $\mathbb{R}^d \setminus O$, $\eta_{\delta}=1$ in $O^{-\delta}:=\cb{x\in O: dist(x,\partial O)\geq \delta}$ and $|\nabla \eta_{\delta}|\leq \frac{C}{\delta}$.

 Also, by density we can choose a sequence $U_{\delta}(\omega,x)=\sum_{i=1}^{n(\delta)}\varphi_i^{\delta}(\omega) \xi_i^{\delta}(x)$ such that $\varphi_i^{\delta}\in L^p_{\mathsf{inv}}(\Omega)$ and $\xi_i^{\delta} \in C^{\infty}_c(\mathbb{R}^d)$, $dist(supp( U_{\delta}),\partial O)\geq \mu(\delta)$ (with $\mu(\delta)\to 0$ as $\delta\to 0$) and 
\begin{align*}
U_{\delta}\to U \quad \text{strongly in } L^p(\Omega)\otimes W^{1,p}(\mathbb{R}^d) \quad \text{as }\delta \to 0.
\end{align*}

Let $u_{\varepsilon,\delta}=\mathcal{F}\e U_{\delta} +\eta_{\delta} \mathcal{G}\e^{0} \chi$, where $\mathcal{G}\e^{0}$ denotes the operator given in Proposition \ref{prop1}. 
We have
\begin{multline}\label{estimate_cor3.3}
 \|u_{\varepsilon,\delta} -(\mathcal{F}\e U+\mathcal{G}\e^{0}\chi)\|_{L^p(\Omega\times \varepsilon\mathbb{Z}^{d})}+\|\nabla^{\varepsilon}(u_{\varepsilon,\delta}-(\mathcal{F}\e U+\mathcal{G}\e^{0}\chi))\|_{L^p(\Omega \times \varepsilon\mathbb{Z}^{d})^d}  \\
 \leq  \| \mathcal{F}\e U_{\delta}-\mathcal{F}\e U \|_{L^p(\Omega\times \varepsilon\mathbb{Z}^{d})}+ \| \brac{\eta_{\delta}-1}\mathcal{G}\e^{0}\chi\|_{L^p(\Omega\times \varepsilon\mathbb{Z}^{d})} + \|\nabla^{\varepsilon}\brac{\mathcal{F}\e U_{\delta}-\mathcal{F}\e U}\|_{L^p(\Omega\times \varepsilon\mathbb{Z}^{d})^d} \\
 + \|\nabla^{\varepsilon}\brac{\eta_{\delta}\mathcal{G}\e^{0}\chi-\mathcal{G}\e^{0} \chi}\|_{L^p(\Omega\times \varepsilon\mathbb{Z}^{d})^d}. 
\end{multline}
Above on the right-hand side, the first term can be bounded by $\| U_{\delta}- U \|_{L^p(\Omega\times \mathbb{R}^d)}$ (by boundedness of $\mathcal{F}\e$), the second term is bounded by $\|\mathcal{T}_{\varepsilon} \mathcal{G}\e^{0}\chi \|_{L^p(\Omega\times \mathbb{R}^d\setminus O^{-\delta})}$ (using the properties of $\eta_{\delta}$) and the third term is bounded by $C\|\nabla U_{\delta}- \nabla U\|_{L^p(\Omega\times \mathbb{R}^d)^d}$ (similarly as in (\ref{inequa1111})). The last term is treated as follows. We take advantage of the following product rule: For $f,g: \varepsilon\mathbb{Z}^{d} \to \mathbb{R}$ it holds that $\nabla^{\varepsilon}_i f(x)g(x)=f(x+\varepsilon e_i)\nabla^{\varepsilon}_ig(x)+g(x)\nabla^{\varepsilon}_i f(x)$. Consequently, we obtain
\begin{multline}\label{eqnarray1.1}
\|\nabla^{\varepsilon}\brac{\eta_{\delta} \mathcal{G}\e^{0}\chi - \mathcal{G}\e^{0}\chi }\|_{L^p(\Omega\times \varepsilon\mathbb{Z}^{d})^d}   \leq  \|\brac{\eta_{\delta}-1}\nabla^{\varepsilon} \mathcal{G}\e^{0}\chi \|_{L^p(\Omega\times \varepsilon\mathbb{Z}^{d})^d} \\ + C \sum_{i=1}^d \expect{\int_{\varepsilon\mathbb{Z}^{d} }|\mathcal{G}\e^{0} \chi(\omega,x+\varepsilon e_i)\nabla^{\varepsilon}_i \eta_{\delta}(x)|^pdm_{\varepsilon}(x)}^{\frac{1}{p}}.
\end{multline}
The first term on the right-hand side of (\ref{eqnarray1.1}) is bounded by $\|\mathcal{T}_{\varepsilon} \nabla^{\varepsilon} \mathcal{G}\e^{0}\chi \|_{L^p(\Omega\times \mathbb{R}^d\setminus O^{-\delta})^d}$ and for small enough $\varepsilon$, the second term is bounded by $\frac{C}{\delta}\|\mathcal{T}_{\varepsilon} \mathcal{G}\e^{0} \chi\|_{L^p(\Omega\times \mathbb{R}^d)}$. Note that 
\begin{equation*}
\limsup_{\delta\to 0} \limsup_{\varepsilon\to 0} \brac{\|\mathcal{T}_{\varepsilon} \nabla^{\varepsilon} \mathcal{G}\e^{0} \chi \|_{L^p(\Omega\times \mathbb{R}^d\setminus O^{-\delta})^d}+\frac{C}{\delta}\|\mathcal{T}_{\varepsilon} \mathcal{G}\e^{0} \chi\|_{L^p(\Omega\times \mathbb{R}^d)}}=0
\end{equation*}
since $\mathcal{T}_{\varepsilon} \mathcal{G}_{\varepsilon}^0 \to 0$ and $\mathcal{T}_{\varepsilon} \nabla^{\varepsilon}\mathcal{G}_{\varepsilon}^0 \to \chi$ as $\varepsilon\to 0$ (Proposition \ref{prop1} (i)).

Collecting all the above bounds for the inequality (\ref{estimate_cor3.3}), using the isometry property of $\mathcal{T}_{\varepsilon}$ and with the help of part (i), we obtain that 
\begin{eqnarray*}
&& \limsup_{\delta\to 0}\limsup_{\varepsilon \to 0} \brac{ \|\mathcal{T}_{\varepsilon} u_{\varepsilon,\delta} - U\|_{L^p(\Omega\times \mathbb{R}^d)}+\|\mathcal{T}_{\varepsilon} \nabla^{\varepsilon}u_{\varepsilon,\delta}-  \nabla U-\chi\|_{L^p(\Omega \times \mathbb{R}^d)^d}+g(\varepsilon, \delta)}\\&=&0,
\end{eqnarray*}
where $g(\varepsilon, \delta)=\twopartdef{0}{\varepsilon\leq \frac{\mu(\delta)}{C(d)}}{1}{\varepsilon>\frac{\mu(\delta)}{C(d)}}$ and $C(d)$ is the diameter of $\Box$. Hence, there exists a diagonal sequence $u\e:=u_{\varepsilon,\delta(\varepsilon)}$ which satisfies the claim of the corollary.
\smallskip

(iv) For a given $(U,\chi)\in ( L^p_{\mathsf{inv}}(\Omega)\otimes L^p(O))\times (\mathbf{L}^p_{\mathsf{pot}}(\Omega)\otimes L^p(O))$ we set $u\e(U,\chi)=\eta_{\delta(\varepsilon)}\brac{\mathcal{F}\e^{\gamma}U+\mathcal{G}\e^{\gamma}\chi}$, where $\eta_{\delta(\varepsilon)}$ is the cut-off function from part (iii) with $\delta(\varepsilon)=\varepsilon^{\frac{\gamma}{2}}$. For notational convenience, we write $u\e$ instead of $u\e(U,\chi)$. We have
\begin{eqnarray*}
& & \|\mathcal{T}_{\varepsilon} u\e- U\|_{L^p(\Omega\times \mathbb{R}^d)}+\|\mathcal{T}_{\varepsilon} \varepsilon^{\gamma} \nabla^{\varepsilon}u\e-\chi\|_{L^p(\Omega\times \mathbb{R}^d)^d} \nonumber \\ &\leq & \|\mathcal{T}_{\varepsilon} u\e- \mathcal{T}_{\varepsilon} \brac{\mathcal{F}\e^{\gamma}U+\mathcal{G}\e^{\gamma}\chi} \|_{L^p(\Omega\times \mathbb{R}^d)} + \|\mathcal{T}_{\varepsilon} \brac{\mathcal{F}\e^{\gamma}U+\mathcal{G}\e^{\gamma}\chi} - U\|_{L^p(\Omega\times \mathbb{R}^d)}\\ & & +\|\mathcal{T}_{\varepsilon} \varepsilon^{\gamma} \nabla^{\varepsilon}u\e-\mathcal{T}_{\varepsilon} \varepsilon^{\gamma} \nabla^{\varepsilon}\brac{\mathcal{F}\e^{\gamma}U+\mathcal{G}\e^{\gamma}\chi}\|_{L^p(\Omega\times \mathbb{R}^d)^d} \\ && +\|\mathcal{T}_{\varepsilon} \varepsilon^{\gamma}\nabla^{\varepsilon} \brac{\mathcal{F}\e^{\gamma}U +\mathcal{G}\e^{\gamma}\chi}-\chi\|_{L^p(\Omega\times \mathbb{R}^d)^d}. \nonumber
\end{eqnarray*}
The second and last terms on the right-hand side above vanish as $\varepsilon\to 0$ using the claim of part (ii). The first term is bounded by $ \|\mathcal{T}_{\varepsilon} \brac{\mathcal{F}\e^{\gamma}U+\mathcal{G}\e^{\gamma}\chi}\|_{L^p(\Omega\times \mathbb{R}^d\setminus O^{-\delta(\varepsilon)})} $ (cf. part (iii)) and therefore it vanishes as $\varepsilon\to 0$ using the fact that $\mathcal{T}_{\varepsilon} \brac{\mathcal{F}\e^{\gamma}U+\mathcal{G}\e^{\gamma}\chi}$ converges strongly (and therefore it is uniformly integrable). For small enough $\varepsilon$, the third term is bounded (up to a constant) by $\|\mathcal{T}_{\varepsilon} \varepsilon^{\gamma} \nabla^{\varepsilon}\brac{\mathcal{F}\e^{\gamma} U + \mathcal{G}\e^{\gamma}\chi} \|_{L^p(\Omega\times \mathbb{R}^d\setminus O^{-\delta(\varepsilon)})^d}+\varepsilon^{\frac{\gamma}{2}}\|\mathcal{T}_{\varepsilon} \brac{ \mathcal{F}\e^{\gamma}U+\mathcal{G}\e^{\gamma}\chi}\|_{L^p(\Omega\times \mathbb{R}^d)}$ (cf. (\ref{eqnarray1.1}) in part (iii)). The last expression vanishes in the limit $\varepsilon\to 0$ using the properties of $\mathcal{F}\e^{\gamma}U+\mathcal{G}\e^{\varepsilon}\chi$. The proof is complete.
\end{proof}
\smallskip

\begin{proof}[Proof of Proposition \ref{pro67}]
Let $\mathcal{E}\e(u\e):=\expect{\int_{O^{+\varepsilon}\cap \varepsilon\mathbb{Z}^{d} }V(T_{\frac{x}{\varepsilon}}\omega, u\e(\omega,x))dm_{\varepsilon}(x) }$.
\smallskip

(i)
Consider a sequence of open sets $A_k\subset \subset O$ which satisfy $A_k \subset A_{k+1}$ and $|O\setminus A_k|\to 0$ as $k \to \infty$. For small enough $\varepsilon$, we have
\begin{align*}
 \mathcal{E}\e(u\e)=& \expect{\int_{O^{+\varepsilon}\cap \varepsilon\mathbb{Z}^{d}}V(\omega, \widetilde{\mathcal{T}}\e u_{\varepsilon}(\omega,x))dm_{\varepsilon}(x)}\\ = & \sum_{x\in O^{+\varepsilon}\cap \varepsilon\mathbb{Z}^{d}} \expect{\int_{x+\varepsilon \Box} V(\omega,\mathcal{T}_{\varepsilon} u\e(\omega,x))dx} \\  = & \expect{\int_{A_k} V(\omega,\mathcal{T}_{\varepsilon} u_{\varepsilon}(\omega,x))dx}+\expect{\int_{L_{\varepsilon,k}} V(\omega,\mathcal{T}_{\varepsilon} u_{\varepsilon}(\omega,x))dx},
\end{align*}
for a suitable small boundary layer set $L_{\varepsilon,k} \subset \cb{x: dist(x, O^{+\varepsilon}\setminus A_k)<C_1 \varepsilon} $, where $C_1$ is a fixed constant depending only on the dimension $d$.
The growth conditions of $V$ yield 
\begin{equation*}
\mathcal{E}\e(u_{\varepsilon})\geq \expect{\int_{A_k}V(\omega,\mathcal{T}_{\varepsilon} u_{\varepsilon}(\omega,x))dx}-C|L_{\varepsilon,k}|.
\end{equation*}
Letting $\varepsilon \to 0$, we obtain
\begin{align*}
\liminf_{\varepsilon \to 0}\mathcal{E}\e(u\e) & \geq \liminf_{\varepsilon\to 0}\expect{\int_{A_k} V(\omega, \mathcal{T}_{\varepsilon} u\e(\omega,x))dx}- C |O\setminus A_k| \\ & \geq \expect{\int_{A_k} V(\omega,U(\omega,x))dx}-C |O\setminus A_k|.
\end{align*}
The last inequality is obtained using the fact that the functional $U\mapsto \expect{\int_{A_k} V(\omega, U)}$ is weakly lower-semicontinuous (since it is convex and continuous w.r.t. the strong $L^p$-topology). Finally, letting $k\to \infty$ we obtain the claimed result.
\smallskip

(ii) Similarly as in part (i), we find suitable boundary layer sets $L\e^+$ and $L\e^-$ (with $|L\e^{\pm}|\to 0$ as $\varepsilon\to 0$), such that
\begin{eqnarray*}
&& \mathcal{E}\e(u_{\varepsilon})\\ & = &\expect{\int_{O}V(\omega,\mathcal{T}\e u\e(\omega,x))dx}+ \expect{\int_{L\e^+}V(\omega,\mathcal{T}\e u\e(\omega,x))dx-\int_{L\e^-}V(\omega,\mathcal{T}\e u\e(\omega,x))dx} \\ & \leq & \expect{\int_{O}V(\omega,\mathcal{T}_{\varepsilon} u_{\varepsilon}(\omega,x))}+C(|L_{\varepsilon}^+|+|L_{\varepsilon}^-|)+ C\expect{\int_{L_{\varepsilon}^+}|\mathcal{T}_{\varepsilon}u_{\varepsilon}(\omega,x)|^p},
\end{eqnarray*}
where we use the growth conditions of the integrand $V$.
The terms in the middle on the right-hand side vanish as $\varepsilon\to 0$, as well as the last term (using strong convergence of $\mathcal{T}_{\varepsilon} u\e$). As a result of this and using strong continuity of $U \mapsto \expect{\int_{O}V(\omega,U)}$, we obtain that $\limsup_{\varepsilon\to 0}\mathcal{E}\e(u\e)\leq \expect{\int_O V(\omega,U(\omega,x))dx}$. Using part (i), we conclude the proof.
\end{proof}
\section{Stochastic homogenization of spring networks}\label{section:4}
In this section, we illustrate the capabilities of the stochastic unfolding operator in homogenization of energy driven problems that invoke convex functionals. We treat a multidimensional analogon of the problem presented in the introduction - a network of springs which exhibit either elastic or elasto-plastic response. The material coefficients are assumed to be rapidly oscillating random fields and we derive effective models in the sense of a discrete-to-continuum transition. 
\smallskip

In Section \ref{lattice_graphs}, we briefly present the setting of periodic lattice graphs and the corresponding differential calculus. If the springs display only elastic behavior and the forces acting on the system do not depend on time, the static equilibrium of the spring network is determined by a convex minimization problem. Accordingly, in Section \ref{section_870} we present homogenization results for convex functionals. On the other hand, in the case of elasto-plastic springs, the evolution of the system is embedded in the framework of evolutionary rate-independent systems (ERIS). A short description of that framework can be found in Appendix \ref{appendix2}, and for a detailed study we refer to \cite{mielke2005evolution,mielke2015rate}. In the limit, as the characteristic size of the springs vanishes, we obtain a homogenized model, which is also described by an ERIS on a continuum physical space (Section \ref{sectionERIS} and \ref{S:gradient}).
\smallskip

We remark that homogenization results concerning minimization problems in the discrete-to-continuum setting are already well established. Earlier works (e.g. \cite{alicandro2011integral}) treat more general problems than we do in this paper. Namely, the considered potentials might be nonconvex and if the media are ergodic, previous works feature quenched homogenization results. This means that for almost every configuration, the energy functional $\Gamma$-converges to a homogenized energy functional (cf. Remark \ref{remark2000}). In our results for convex minimization problems, we obtain weaker, averaged homogenization results (see Section \ref{section_870}). Despite our results being weaker in a general situation, we would like to point out that our strategy relies only on the simple idea of the unfolding operator; namely the compactness properties of the unfolding (which are closely related to compactness statements in usual $L^p$-spaces) and lower-semicontinuity of convex functionals play a central role in our analysis. On the other hand, previous works are based on more involved techniques, such as the subadditive ergodic theorem \cite{akcoglu1981ergodic} (\cite{alicandro2011integral,neukamm2016stochastic}), or the notion of quenched stochastic two-scale convergence \cite{heida2017stochastic}.
\subsection{Functions on lattice graphs}\label{lattice_graphs}
Let $\mathsf{E}_0=\cb{b_1,...,b_k}\in \mathbb{Z}^d\setminus \cb{0}$ be the \textit{edge generating set} and we assume that $\mathsf{E}_0$ includes $\cb{e_i}_{i=1,...,d}$. We consider the rescaled periodic lattice graph $(\varepsilon \mathbb{Z}^d ,\varepsilon\mathsf{E} )$, where the set of edges is given by $\mathsf{E}=\cb{[x,x+b_i]:x\in \mathbb{Z}^d, i=1,...,k}$. For $u:\varepsilon\mathbb{Z}^{d}\rightarrow \mathbb{R}^d$ the difference quotient along the edge generated by $b_i$ is
\begin{equation*}
\partial_i^{\varepsilon} u:\varepsilon\mathbb{Z}^{d} \rightarrow \mathbb{R}^d, \quad
\partial_i^{\varepsilon} u(x)=\frac{u(x+\varepsilon b_i)-u(x)}{\varepsilon|b_i|}.
\end{equation*}
Note that for each $b_i$ there exists $B_i:\mathbb{Z}^d \rightarrow \mathbb{Z}^d$ such that\footnote{$B_i$ are not uniquely determined, however we consider one such choice corresponding to a path between $0$ and $b_i$.}
\begin{equation*}
\partial_i^{\varepsilon} u(x)=\sum_{y\in\mathbb{Z}^{d}}\nabla^{\varepsilon}u(x-\varepsilon y)B_i(y),
\end{equation*}
where the discrete gradient $\nabla^{\varepsilon}u$ is defined as in Section \ref{Physical_space}.
We define the \textit{discrete symmetrized gradient} as $\nabla_{s}^{\varepsilon} u:\varepsilon\mathbb{Z}^{d}\rightarrow \mathbb{R}^k$ 
\begin{equation*}
\nabla_{s}^{\varepsilon} u(x)=\brac{\frac{b_1}{|b_1|}\cdot \partial_1^{\varepsilon} u(x),...,\frac{b_k}{|b_k|}\cdot \partial_k^{\varepsilon} u(x)}.
\end{equation*}
Furthermore, we introduce a suitable symmetrization operator for random fields as follows. For a matrix $F \in \mathbb{R}^{d\times d}$, we denote by $F_s\in \mathbb{R}^k$ the vector with entries $(F_s)_i=\frac{b_i}{|b_i|}\cdot F \frac{b_i}{|b_i|}$ ($i=1,...,k$). Analogously, for $F:\Omega\rightarrow \mathbb{R}^{d\times d}$ measurable, we set $ F_s:\Omega\rightarrow \mathbb{R}^k$,
\begin{equation*}
(F_s)_i(\omega)=\frac{b_i}{|b_i|}\cdot \sum_{y\in\mathbb{Z}^d}F(T_{-y}\omega)B_i(y), \quad (i=1,...,k).
\end{equation*}
If $F=\nabla U$ or $F=DU$, instead of $F_s$ we write $\nabla_s U$ or $D_s U$.

Throughout the paper, we assume that $(\mathbb{Z}^d ,\mathsf{E})$ satisfies a discrete version of Korn's inequality:
\begin{align}
\text{There exists } & C(d,p)>0 \text{ such that} \text{ for all compactly supported } u:\mathbb{Z}^{d}\rightarrow \mathbb{R}^d \nonumber \\ & \int_{\mathbb{Z}^{d}} |\nabla u(x)|^p dm(x) \leq C(d,p) \int_{\mathbb{Z}^{d}} |\nabla_{s}u(x)|^p dm(x). \tag{Korn}\label{Korn_assumpt} 
\end{align}
\begin{remark} An example of a lattice satisfying the above Korn's inequality is $\brac{\mathbb{Z}^d,\mathsf{E}}$ with $\mathsf{E}_0=\cb{\sum_{i=1}^d \delta_i e_i: \delta \in \cb{0,1}^d\setminus 0}$.
\end{remark}
The assumption (\ref{Korn_assumpt}) implies a continuum version of Korn's inequality. Namely, let $O\subset \mathbb{R}^d$ be open and bounded, there exists $C(p)>0$ such that
\begin{equation}\label{continuum_korn}
\int_O |\nabla U|^p dx \leq C(p) \int_O|\nabla_s U|^p dx \quad \text{for any } U \in W^{1,p}_0(O).
\end{equation}
This inequality is obtained applying (\ref{Korn_assumpt}) to $\pi\e u_{\delta}$, where $u_{\delta}$ is a smooth approximation of $u$, and passing to the limits $\varepsilon\to 0$ and $\delta\to 0$ (cf. Lemma \ref{symmet} in Section \ref{section_1111}).
Note that (\ref{Korn_assumpt}) implies another, stochastic version of Korn's inequality - see Lemma \ref{kornstoch} in Section \ref{section_1111}.
\subsection{Static problem}\label{section_870}
As a preparation for the rate-independent evolutionary problem, we first discuss a homogenization procedure for a static convex minimization problem, and then discuss different notions of convergence in the homogenization result. 
We consider a set of particles with reference positions at $\varepsilon\mathbb{Z}^{d}$. It is assumed that the edges $\varepsilon \mathsf{E}$ represent springs with elastic response (cf. the introduction with internal variable $z=0$ and loading $l(t)=l$). The equilibrium state of the system is determined by a minimization problem which (in a slightly more general setting) reads as
\begin{multline}\label{minimum}
\min_{u\in \brac{L^p(\Omega)\otimes L^p_0(O\cap \varepsilon\mathbb{Z}^{d})}^d}\bigg\langle \int_{O^{+\varepsilon}\cap \varepsilon\mathbb{Z}^{d}} V(T_{\frac{x}{\varepsilon}}\omega,\nabla_{s}^{\varepsilon} u(\omega,x))dm_{\varepsilon}(x) \\ -\int_{O\cap \varepsilon\mathbb{Z}^{d}}l\e(\omega,x)\cdot u(\omega,x)dm\e(x)\bigg\rangle.
\end{multline}
We assume the following:
\begin{itemize}
\item[(A0)] $O\subset \mathbb{R}^d$ is a bounded domain with Lipschitz boundary. 

We set $O^{+\varepsilon}=O\cup \cb{x \in \mathbb{R}^d: (x,x+\varepsilon b_i)\cap O\neq \emptyset \text{ for some }b_i \in \mathsf{E}_0}$.
\item[(A1)] $V:\Omega \times \mathbb{R}^k \rightarrow \mathbb{R}$ is jointly measurable (i.e.,  w.r.t. the product $\sigma$-algebra $\mathcal{F}\otimes \mathcal{B}(\mathbb{R}^k)$).
\item[(A2)]For $P$-a.e. $\omega\in \Omega$, $V(\omega,\cdot)$ is convex.
\item[(A3)] There exists $C>0$ such that
\begin{equation*}
\frac{1}{C}|F|^p-C\leq V(\omega,F) \leq C(|F|^p+1),
\end{equation*}
for $P$-a.e. $\omega\in \Omega$ and all $F\in \mathbb{R}^k$.
\end{itemize}
In the case that the loading $l\e$ converges in a sufficiently strong sense (see Remark \ref{remark_930}), in order to describe the asymptotic behavior of minimizers in (\ref{minimum}), it is sufficient to consider the energy functional $\mathcal{E}\e:\brac{L^p(\Omega)\otimes L^p_0(O\cap \varepsilon\mathbb{Z}^{d})}^d \rightarrow \mathbb{R},$
\begin{equation*}
\mathcal{E}\e(u)=\expect{\int_{O^{+\varepsilon}\cap \varepsilon\mathbb{Z}^{d}} V(T_{\frac{x}{\varepsilon}}\omega,\nabla_{s}^{\varepsilon} u(\omega,x))dm_{\varepsilon}(x)}.
\end{equation*}
As shown below for $\varepsilon\rightarrow 0$, we derive the effective two-scale functional 
\begin{align*}
& \mathcal{E}_0:(L^p_{\mathsf{inv}}(\Omega)\otimes W^{1,p}_0(O))^d \times(\mathbf{L}^p_{\mathsf{pot}}(\Omega)\otimes L^p(O))^d \rightarrow \mathbb{R},\\
& \mathcal{E}_0(U,\chi)=\expect{\int_{O} V(\omega,\nabla_s U(\omega,x) + \chi_s(\omega,x)) dx}.
\end{align*}
Moreover, if we assume that $\expect{\cdot}$ is ergodic, the effective energy reduces to a single-scale functional
 \begin{equation*}
\mathcal{E}_{\mathsf{hom}}:W^{1,p}_0(O)^d \rightarrow \mathbb{R}, \quad \mathcal{E}_{\mathsf{hom}}(U)=\int_{O} V_{\mathsf{hom}}(\nabla U(x)) dx,
\end{equation*}
where the homogenized energy density $V_{\mathsf{hom}}:\mathbb{R}^{d\times d}\rightarrow \mathbb{R}$ is defined by the corrector problem
\begin{equation}\label{problem}
V_{\mathsf{hom}}(F)=\inf_{\chi\in \mathbf{L}^p_{\mathsf{pot}}(\Omega)^d}\expect{ V(\omega, F_s + \chi_s(\omega))}.
\end{equation}

\begin{theorem}[Two-scale homogenization]\label{gammatheorem}
Assume $(A0)$-$(A3)$.
\begin{enumerate}[(i)]
\item (Compactness) For $u_{\varepsilon}\in \brac{L^p(\Omega)\otimes L^p_0(O\cap \varepsilon\mathbb{Z}^{d})}^d$ with $\limsup_{\varepsilon\rightarrow 0}\mathcal{E}\e(u_{\varepsilon})<\infty$, there exist a subsequence (not relabeled), $U\in (L^p_{\mathsf{inv}}(\Omega)\otimes W^{1,p}_0(O))^d$ and $\chi \in (\mathbf{L}^p_{\mathsf{pot}}(\Omega)\otimes L^p(O))^d$ such that 
\begin{equation}\label{c5}
u_{\varepsilon} \overset{2}{\rightharpoonup} U \text{ in }L^{p}(\Omega\times \mathbb{R}^d)^d, \quad \nabla^{\varepsilon}u_{\varepsilon} \overset{2}{\rightharpoonup} \nabla U+\chi \text{ in } L^{p}(\Omega\times \mathbb{R}^d)^{d\times d}.
\end{equation}
\item (Lower bound) Assume that the convergence (\ref{c5}) holds for the whole sequence $u_{\varepsilon}$. Then 
\begin{equation*}
 \liminf_{\varepsilon\rightarrow 0} \mathcal{E}\e(u_{\varepsilon})\geq \mathcal{E}_{0}(U,\chi).
\end{equation*}
\item (Upper bound) For any $U\in (L^p_{\mathsf{inv}}(\Omega)\otimes W^{1,p}_0(O))^d$ and $\chi\in(\mathbf{L}^p_{\mathsf{pot}}(\Omega)\otimes L^p(O))^d$, there exists a sequence $ u_{\varepsilon}\in \brac{L^p(\Omega)\otimes L^p_0(O\cap \varepsilon\mathbb{Z}^{d})}^d$ such that
\begin{align*}
& u_{\varepsilon} \overset{2}{\rightarrow} U \text{ in } L^p(\Omega\times \mathbb{R}^d)^d, \quad 
\nabla^{\varepsilon}u_{\varepsilon}\overset{2}{\rightarrow} \nabla U+\chi \text{ in } L^{p}(\Omega\times \mathbb{R}^d)^{d\times d}, \\
& \lim_{\varepsilon\rightarrow 0}\mathcal{E}\e(u_{\varepsilon})=\mathcal{E}_{0}(U,\chi).
\end{align*}
\end{enumerate}
\end{theorem}
\begin{remark}[Convergence of minimizers]\label{remark_930}
Under the assumptions of Theorem \ref{gammatheorem} and if the loadings $l\e \in L^{q}(\Omega \times \varepsilon \mathbb{Z}^d)^d$ satisfy $l\e\overset{2}{\to}l$, where $l\in L^q(\Omega\times O)^d$, the above theorem implies (by a standard argument from $\Gamma$-convergence) that minimizers $u\e$ in (\ref{minimum}) satisfy (up to a subsequence)
\begin{equation*}
u\e \overset{2}{\rightharpoonup} U \text{ in }L^p(\Omega\times \mathbb{R}^d)^d \quad \text{and} \quad \nabla^{\varepsilon}u\e \overset{2}{\rightharpoonup} \nabla U + \chi  \text{ in }L^p(\Omega \times \mathbb{R}^d)^{d\times d},
\end{equation*}
where $(U,\chi)\in (L^p_{\mathsf{inv}}(\Omega)\otimes W^{1,p}_0(O))^d \times(\mathbf{L}^p_{\mathsf{pot}}(\Omega)\otimes L^p(O))^d$ is a minimizer of the two-scale functional $\mathcal{I}_{0}: (L^p_{\mathsf{inv}}(\Omega)\otimes W^{1,p}_0(O))^d \times(\mathbf{L}^p_{\mathsf{pot}}(\Omega)\otimes L^p(O))^d \rightarrow \mathbb{R},$ 
\begin{equation*}
\mathcal{I}_{0}(U,\chi)=\mathcal{E}_{0}(U,\chi)-\expect{\int_O l \cdot U dx}.
\end{equation*}
\end{remark}

In the ergodic case, the limit is deterministic:
\begin{theorem}[Ergodic case]\label{gamma2}
Assume $(A0)-(A3)$ and that $\expect{\cdot}$ is ergodic.
\begin{enumerate}[(i)]
\item (Compactness and lower bound) Let $u_{\varepsilon}\in \brac{L^p(\Omega)\otimes L^p_0(O\cap \varepsilon\mathbb{Z}^{d})}^d$ satisfy
\begin{equation*}
\limsup_{\varepsilon\rightarrow 0}\mathcal{E}\e(u_{\varepsilon})<\infty.
\end{equation*}
There exists $U\in W^{1,p}_0(O)^d$ such that, up to a subsequence, 
\begin{align*}
& \expect{u_{\varepsilon}}\rightarrow U \text{ in } L^p(\mathbb{R}^d)^d,\quad
\expect{\nabla^{\varepsilon}u\e}\rightharpoonup{\nabla U} \text{ in }L^p(\mathbb{R}^d)^{d \times d},\\
 & \mathcal{E}_{\mathsf{hom}}(U)\leq \liminf_{\varepsilon\rightarrow 0} \mathcal{E}\e(u_{\varepsilon}). 
\end{align*}
\item (Upper bound) For any $U\in W^{1,p}_0(O)^d$, there exists $u\e \in L^p(\Omega)\otimes L^p_0(O\cap \varepsilon\mathbb{Z}^d)^d$ such that 
\begin{equation*}
\expect{u_{\varepsilon}}\rightarrow U \text{ in } L^p(\mathbb{R}^d)^d,\quad
\expect{\nabla^{\varepsilon} u\e}\rightharpoonup \nabla U \text{ in }L^p(\mathbb{R}^d)^{d \times d},\quad
\lim_{\varepsilon\rightarrow 0}\mathcal{E}\e(u_{\varepsilon})=\mathcal{E}_{\mathsf{hom}}(U).
\end{equation*}
\end{enumerate}
\end{theorem}
\begin{remark}\label{remark2000}
In the ergodic case, $\mathcal{E}\e$ $\Gamma$-converges to the deterministic functional $\mathcal{E}_{\mathsf{hom}}$. In fact, it is known that for $P$-a.e. $\omega\in \Omega$ the functional 
\begin{equation*}
\mathcal{E}\e(\omega, \cdot):u\in L^p_0(O\cap \varepsilon\mathbb{Z}^{d} )^d\mapsto \int_{O^{+\varepsilon}\cap \varepsilon\mathbb{Z}^{d} }V(T_{\frac{x}{\varepsilon}}\omega, \nabla^{\varepsilon}_s u\e(x))dm_{\varepsilon}(x) \in \mathbb{R}
\end{equation*}
$\Gamma$-converges to $\mathcal{E}_{\mathsf{hom}}$. This quenched convergence result can be found, e.g., in \cite{alicandro2011integral}, where even  nonconvex integrands are treated. Based on stochastic unfolding, we obtain the weaker ``averaged'' result of Theorem~\ref{gamma2} as a corollary of Theorem \ref{gammatheorem} and Corollary \ref{ergodic}. While our argument is relatively easy, the analysis of the stronger quenched convergence result is based on the subadditive ergodic theorem \cite{akcoglu1981ergodic} and is more involved. We remark that minimizers $\omega \mapsto u\e(\omega)$ of the functionals $\mathcal{E}\e(\omega,\cdot)$ present a random field (under the above assumptions), which minimizes the averaged energy $\mathcal{E}\e$ (and vice versa).
\end{remark}

If we, additionally, assume the following assumption, we obtain strong convergence for minimizers.
\begin{itemize}
\item[(A4)]  For $P$-a.e. $\omega \in \Omega$, $V(\omega,\cdot)$ is uniformly convex with modulus $(\cdot)^p$, i.e.,  there exists $C>0$ such that for all $F, G\in \mathbb{R}^{k}$ and $t\in [0,1]$
\begin{align*}
V(\omega,tF+(1-t)G)\leq t V(\omega,F)+(1-t) V(\omega,G)-(1-t)tC|F-G|^p.
\end{align*} 
\end{itemize}
\begin{proposition}[Strong convergence]\label{prop41} Let the assumptions of Theorem \ref{gamma2} and $(A4)$ hold. Let $l\e\in L^{q}(\Omega \times \varepsilon \mathbb{Z}^d)^d$ satisfy $l\e\overset{2}{\to}l$ in $L^q(\Omega\times \mathbb{R}^d)^d$, where $l\in L^q(O)^d$. The problem (\ref{minimum}) admits a unique minimizer $u\e \in \brac{L^p(\Omega)\otimes L^p_0(O\cap \varepsilon\mathbb{Z}^{d})}^d$, which satisfies
\begin{equation}\label{eq:additional}
\quad u\e \overset{2}{\to} U  \text{ in }L^p(\Omega\times\mathbb{R}^d)^d,
\end{equation}
where $U\in W^{1,p}_0(O)^d$ is the unique minimizer of 
\begin{align*}
\mathcal{I}_{\mathsf{hom}}: W^{1,p}_0(O)^d \rightarrow \mathbb{R}, \quad
\mathcal{I}_{\mathsf{hom}}(U)=\mathcal{E}_{\mathsf{hom}}(U)-\int_O l \cdot U dx.
\end{align*}
\end{proposition}
\begin{remark}[Periodic homogenization]\label{remark3000}
As mentioned earlier, one knows that $\mathcal{E}\e(\omega, \cdot) \overset{\Gamma}{\to} \mathcal{E}_{\mathsf{hom}}$ in the quenched sense, whereas we obtain convergence for minimizers in an averaged sense (as above in (\ref{eq:additional})). Yet if we consider the setting for periodic homogenization, using the above convergence in the mean, we recover a standard (pointwise) periodic homogenization result. In particular, for $N\in \mathbb{N}$ we set $\Omega=N\Box\cap \mathbb{Z}^d_{/N\mathbb{Z}^d}$ the discrete $N$-torus with a corresponding (rescaled) counting measure. The dynamical system $(T_x)$ is defined as $T_x\omega= \omega+x \mod N$. The above example of the probability space with the dynamical system $(T_x)$ satisfies the assumptions given in Section \ref{section:297} and is ergodic. We remark that in this case $\Omega$ is a finite set and therefore (\ref{eq:additional}) implies that $u_{\varepsilon}(\omega)\to U$ in $L^p(\mathbb{R}^d)^d$ for all $\omega\in \Omega$.

\end{remark}
\subsection{Rate-independent evolutionary problem}\label{sectionERIS}
Let us first describe the system we have in mind. A system of particles connected with springs is represented using $(\varepsilon\mathbb{Z}^{d}, \varepsilon \mathsf{E})$. Namely, $\varepsilon\mathbb{Z}^{d}\cap O$ serves as the reference configuration of particles. The edges $\varepsilon \mathsf{E}$ represent springs with elasto-plastic response (cf. the introduction). Upon an external loading $l$, the system evolves according to an ERIS (see Appendix \ref{appendix2}). Let $O,O^{+\varepsilon}\subset \mathbb{R}^d$ be open (see below for specific assumptions). The following model is a random and discrete counterpart of the model considered in \cite{mielke2007two}, where the periodic continuum case is treated.

\begin{itemize}
\item The state space is $Y\e=\brac{L^2(\Omega)\otimes L^2_0(O\cap \varepsilon\mathbb{Z}^{d})}^d \times \brac{L^2(\Omega)\otimes L^2_0(O^{+\varepsilon}\cap \varepsilon\mathbb{Z}^{d})}^k$, and the displacement $u\e$ and the internal variable $z\e$ are merged into a joint variable $y\e=\brac{u\e, z\e}$. We equip $Y\e$ with the scalar product
\begin{eqnarray*}
&&\expect{y_1, y_2}_{Y\e}\\ &=&  \expect{\int_{\varepsilon\mathbb{Z}^{d}}u_1(\omega,x)\cdot u_2(\omega,x) dm_{\varepsilon}(x)} +\expect{\int_{\varepsilon\mathbb{Z}^{d}}\nabla^{\varepsilon}u_1(\omega,x) : \nabla^{\varepsilon}u_2(\omega,x) dm_{\varepsilon}(x)} \\ && +\expect{\int_{\varepsilon\mathbb{Z}^{d}} z_1(\omega,x)\cdot z_2(\omega,x) dm_{\varepsilon}(x)}.
\end{eqnarray*}
\item The total energy functional is $\mathcal{E}\e:[0,T]\times Y_{\varepsilon} \rightarrow \mathbb{R}$,
\begin{align*}
& \mathcal{E}\e(t,y_{\varepsilon})=\frac{1}{2} \expect{ \mathbb{A}_{\varepsilon}y_{\varepsilon},y_{\varepsilon} }_{Y\e^*,Y\e} -\expect{\int_{O\cap \varepsilon\mathbb{Z}^{d}}\pi\e l(t)(x)\cdot u_{\varepsilon}(\omega,x) dm_{\varepsilon}(x)},\\
& \expect{ \mathbb{A}_{\varepsilon}y_1,y_2}_{Y\e^*,Y\e}=\expect{\int_{O^{+\varepsilon}\cap \varepsilon\mathbb{Z}^{d} }A(T_{\frac{x}{\varepsilon}}\omega)\binom{\nabla_{s}^{\varepsilon} u_1(\omega,x)}{z_1(\omega,x)}\cdot \binom{\nabla_{s}^{\varepsilon} u_2(\omega,x)}{z_2(\omega,x)} dm_{\varepsilon}(x)}.
\end{align*}
\item The dissipation potential is $\Psi\e:Y_{\varepsilon}\rightarrow [0,\infty)$,
\begin{equation*}
\Psi\e(y_{\varepsilon})=\expect{\int_{O^{+\varepsilon}\cap \varepsilon\mathbb{Z}^{d}}\rho(T_{\frac{x}{\varepsilon}}\omega,z_{\varepsilon}(\omega,x))dm_{\varepsilon}(x)}.
\end{equation*}
\end{itemize}

We assume the following:
\begin{itemize}
\item[(B0)] $O\subset \mathbb{R}^d$ is a bounded domain with Lipschitz boundary. 

We set $O^{+\varepsilon}=O\cup \cb{x \in \mathbb{R}^d: (x,x+\varepsilon b_i)\cap O\neq \emptyset \text{ for some }b_i \in \mathsf{E}_0}$.
\item[(B1)] $A\in L^{\infty}(\Omega,\mathbb{R}^{2k\times 2k}_{sym})$ and it satisfies: there exists $C>0$ such that $A(\omega)F\cdot F\geq C|F|^2$ for $P$-a.e. $\omega\in \Omega$ and every $F\in \mathbb{R}^{2k}$.
\item[(B2)] $\rho:\Omega\times \mathbb{R}^k\rightarrow [0,\infty)$ is jointly measurable (i.e.,  w.r.t. $\mathcal{F}\otimes \mathcal{B}(\mathbb{R}^k)$) and for $P$-a.e. $\omega$, $\rho(\omega,\cdot)$ is convex and positively homogeneous of order 1, i.e.,  $\rho(\omega,\alpha F)=\alpha \rho(\omega, F)$ for all $\alpha\geq 0$ and $F\in \mathbb{R}^k$ (we also say \textit{positively 1-homogeneous}).
\end{itemize}
We consider the ERIS (see Appendix) associated with $\brac{\mathcal{E}\e, \Psi\e}$ and we denote by $
S_{\varepsilon}(t) := \cb{y \in Y\e:\mathcal{E}\e(t,y)\leq \mathcal{E}\e(t,\widetilde{y})+\Psi\e(\widetilde{y}-y) \text{ for all }\widetilde{y}\in Y\e  } 
$ the set of stable states corresponding to $(\mathcal{E}\e,\Psi\e)$ at time $t\in [0,T]$.
\begin{remark}\label{remark10}
If we assume $(B0)$-$(B2)$, $l\in C^1([0,T],L^2(O)^d)$ and $y_{\varepsilon}^0\in S_{\varepsilon}(0)$, using Theorem \ref{abstract} we obtain that there exists a unique energetic solution $y_{\varepsilon}\in C^{Lip}([0,T],Y_{\varepsilon})$ to the ERIS associated with $\brac{\mathcal{E}\e,\Psi\e}$ with $y\e(0)=y\e^0$, i.e.,  for all $t\in [0,T]$ we have $y\e(t)\in S\e(t)$ and
\begin{equation}\label{enbal1}
\mathcal{E}\e(t,y\e(t))+\int_0^t \Psi\e(\dot{y}\e(s))ds=\mathcal{E}\e(0,y\e(0))-\int_0^t \expect{\int_{O\cap \varepsilon\mathbb{Z}^{d}}\pi\e \dot{l}(s)\cdot u\e(s)dm_{\varepsilon}}ds, 
\end{equation}
and, moreover, $\|y_{\varepsilon}(t)-y_{\varepsilon}(s)\|_{Y_{\varepsilon}}\leq C|t-s|$ for all $t,s \in [0,T]$.

\end{remark}
The passage to the limit model (as $\varepsilon\rightarrow 0$) is conducted in the setting of evolutionary $\Gamma$-convergence \cite{mielke2008gamma} and involves a discrete-to-continuum transition. The homogenized model as well is described by an ERIS:
\begin{itemize}
\item The state space is given by $Y=H^1_0(O)^d\times L^2(\Omega\times O)^k\times (\mathbf{L}^2_{\mathsf{pot}}(\Omega)\otimes L^2(O))^d$ and we denote the state variable by $y=\brac{U,Z,\chi}$. 
\item The energy functional is 
\begin{align*}
& \mathcal{E}_{0}:[0,T]\times Y\rightarrow \mathbb{R}, \quad
\mathcal{E}_{0}(t,y)=\frac{1}{2}\expect{\mathbb{A}y,y}_{Y^*,Y}-\int_{O} l(t)\cdot U dx, \\
& \expect{\mathbb{A}y,y}_{Y^*,Y}=\int_O\expect{A(\omega)\binom{\nabla_s U(x)+\chi_s(\omega,x)}{Z(\omega,x)}\cdot \binom{\nabla_s U(x)+\chi_s(\omega,x)}{Z(\omega,x)}}dx.
\end{align*}
\item The limit dissipation functional is given by
\begin{equation*}
\Psi_0:Y\rightarrow [0,\infty], \quad
\Psi_0(y)=\int_O \expect{\rho(\omega, Z(\omega,x))}dx.
\end{equation*}   
\end{itemize}
We consider the set of stable states corresponding to $\brac{\mathcal{E}_{0},\Psi_0}$ at time $t\in [0,T]$
$
S(t):= \cb{y \in Y:\mathcal{E}_0(t,y)\leq \mathcal{E}_0(t,\widetilde{y})+\Psi_0(\widetilde{y}-y) \text{ for all }\widetilde{y}\in Y  }.
$
\begin{remark}\label{remark_1030}
If we assume (B0)-(B2), $l\in C^1([0,T],L^2(O)^d)$ and $y^0\in S(0)$, then the assumptions of Theorem \ref{abstract} are satisfied (see Lemma \ref{lemm12}) and therefore there exists a unique energetic solution $y \in C^{Lip}([0,T],Y)$ to the ERIS associated with $\brac{\mathcal{E}_{0},\Psi_0}$ with $y(0)=y^0$, i.e.,  for all $t\in [0,T]$ we have $y(t)\in S(t)$ and
\begin{equation}\label{enbal2}
\mathcal{E}_{0}(t,y(t))+\int_0^t \Psi_0(\dot{y}(s))ds = \mathcal{E}_{0}(0,y(0))-\int_0^t \int_O \dot{l}(s)\cdot U(s)dx ds,
\end{equation} 
and, moreover, $\|y(t)-y(s)\|_{Y}\leq C|t-s|$ for all $t,s \in [0,T]$.

\end{remark}
For notational convenience, we introduce the following abbreviation: For $y_{\varepsilon}\in Y_{\varepsilon}$ and $y\in Y$,
\begin{equation*}
y_{\varepsilon}\overset{c2}{\rightharpoonup} y \quad \Leftrightarrow \quad u_{\varepsilon}\overset{2}{\rightharpoonup} U  , \; \nabla^{\varepsilon}u_{\varepsilon}\overset{2}{\rightharpoonup} \nabla U+\chi \;\text{and } z_{\varepsilon}\overset{2}{\rightharpoonup} Z \quad \text{(in the corresp. } L^2 \text{-spaces}\text{)}.
\end{equation*}
Also, we use $\overset{c2}{\rightarrow}$ if the above quantities strongly two-scale converge. The ``c" in this shorthand refers to ``cross" convergence as in the periodic case. The proof of the following homogenization theorem closely follows the strategy developed in \cite{mielke2007two} (see also \cite{attouch1984variational,mielke2016evolutionary} for general strategies for evolutionary $\Gamma$-convergence for abstract gradient systems). In that paper, the periodic unfolding method is applied to a similar (continuum) problem with periodic coefficients.
\begin{theorem}\label{evgamma}
Assume $(B0)$-$(B2)$, $\expect{\cdot}$ is ergodic, $l \in C^1([0,T],L^2(O)^d)$ and $y_{\varepsilon}^0\in S_{\varepsilon}(0)$ with
\begin{equation*}
y_{\varepsilon}^0\overset{c2}{\rightarrow} y_0\in Y.
\end{equation*}
Let $y\e \in C^{Lip}([0,T],Y\e)$ be the unique energetic solution associated with $\brac{\mathcal{E}\e,\Psi\e}$ and $y\e(0)=y\e^0$. Then 
\begin{equation*}
y_0\in S(0) \text{ and } 
\text{for every }t\in [0,T]: \quad y_{\varepsilon}(t)\overset{c2}{\rightarrow}y(t),
\end{equation*}
where $y\in C^{Lip}([0,T],Y)$ is the unique energetic solution associated with $\brac{\mathcal{E}_{0},\Psi_0}$ and $y(0)=y^0$.
\end{theorem}
\begin{remark}
We remark that the above result holds true in the case that $\expect{\cdot}$ is not ergodic (with minor changes in the proof) with a modified state space for the continuum model, specifically $Y=(L^2_{\mathsf{inv}}(\Omega)\otimes H^1_0(O))^d \times  L^2(\Omega\times O)^k\times (\mathbf{L}^2_{\mathsf{pot}}(\Omega)\otimes L^2(O))^d$.
\end{remark}
\subsection{Gradient plasticity}\label{S:gradient} The limit rate-independent system from the previous section cannot be equivalently recast as a rate-independent system with deterministic properties as in the case of convex minimization (Theorem \ref{gamma2}). The reason for this is that the limiting internal variable $Z$ is in general not deterministic. The microscopic problem might be regularized by adding a gradient term of the internal variable $z\e$ and in that way homogenization yields a deterministic limit problem. This strategy was demonstrated in \cite{hanke2011homogenization}, where periodic homogenization of gradient plasticity in the continuum setting is discussed. In the following, we show that the same applies in our stochastic, discrete-to-continuum setting. 
\smallskip

Let $\gamma \in (0,1)$. The new microscopic system involves the same dissipation potential $\Psi\e$ as before, as well as the same state space $Y\e$, yet now equipped with the scalar product
\begin{eqnarray*}
&& \expect{y_1, y_2}_{Y\e^{\gamma}} \\ & = & \expect{\int_{\varepsilon\mathbb{Z}^{d}}u_1(\omega,x)\cdot u_2(\omega,x) dm_{\varepsilon}(x)}+\expect{\int_{\varepsilon\mathbb{Z}^{d}}\nabla^{\varepsilon}u_1(\omega,x) : \nabla^{\varepsilon}u_2(\omega,x) dm_{\varepsilon}(x)}\\ & &+ \expect{\int_{\varepsilon\mathbb{Z}^{d}} z_1(\omega,x)\cdot z_2(\omega,x) dm_{\varepsilon}(x)}+\expect{ \int_{\varepsilon \mathbb{Z}^d} \varepsilon^{\gamma} \nabla^{\varepsilon} z_1(\omega,x) :\varepsilon^{\gamma} \nabla^{\varepsilon} z_2(\omega,x) dm_{\varepsilon}(x) }.
\end{eqnarray*} 
We consider a modified energy functional $\mathcal{E}^{\gamma}_{\varepsilon}:[0,T]\times Y\e\to \mathbb{R} $ 
\begin{equation*}
\mathcal{E}^{\gamma}_{\varepsilon}(t,y\e)=\mathcal{E}_{\varepsilon}(t,y\e)+\expect{\int_{\varepsilon\mathbb{Z}^{d}} G(T_{\frac{x}{\varepsilon}} \omega) \varepsilon^{\gamma}\nabla^{\varepsilon} z\e (\omega,x) : \varepsilon^{\gamma} \nabla^{\varepsilon} z\e(\omega,x) dm_{\varepsilon}(x) },
\end{equation*}
where $G: \Omega \to \mathbb{R}^{k\times d}$. We assume the following:
\begin{itemize}
\item[(B3)] $G \in L^{\infty}(\Omega,Lin(\mathbb{R}^{k \times d},\mathbb{R}^{k \times d}))$ and it satisfies the following: There exists $C>0$ such that $G(\omega)F_1: F_1\geq C|F_1|^2$ and $G(\omega)F_1:F_2=F_1:G(\omega)F_2$ for $P$-a.e. $\omega\in \Omega$ and all $F_1,F_2 \in \mathbb{R}^{k\times d}$.
\end{itemize}
The set of stable states at time $t\in [0,T]$ corresponding to $\brac{\mathcal{E}^{\gamma}_{\varepsilon},\Psi\e}$ is denoted by 
$
S\e^{\gamma}(t):=\cb{y \in Y\e:\mathcal{E}\e^{\gamma}(t,y)\leq \mathcal{E}\e^{\gamma}(t,\widetilde{y})+\Psi\e(\widetilde{y}-y) \text{ for all }\widetilde{y}\in Y\e}.
$
\begin{remark}\label{remark:1292}
If we assume $(B0)$-$(B3)$, $l\in C^1([0,T],L^2(O)^d)$ and $y_{\varepsilon}^0\in S_{\varepsilon}^{\gamma}(0)$, then, using Theorem \ref{abstract}, there exists a unique energetic solution $y_{\varepsilon}\in C^{Lip}([0,T],Y_{\varepsilon})$ to the ERIS associated with $\brac{\mathcal{E}\e^{\gamma},\Psi\e}$ with $y\e(0)=y\e^0$, i.e.,  for all $t\in [0,T]$ we have $y\e(t)\in S\e^{\gamma}(t)$ and
\begin{equation}\label{enbal3}
\mathcal{E}\e^{\gamma}(t,y\e(t))+\int_0^t \Psi\e(\dot{y}\e(s))ds = \mathcal{E}\e^{\gamma}(0,y\e(0))-\int_0^t \expect{\int_{O\cap \varepsilon\mathbb{Z}^{d}}\pi\e \dot{l}(s)\cdot u\e(s)dm_{\varepsilon}}ds,
\end{equation}  
and, moreover, $\|y\e(t)-y\e(s)\|_{Y\e^{\gamma}} \leq C|t-s|$ for all $t,s \in [0,T]$.
\end{remark}

In the limit $\varepsilon \to 0$, we obtain a deterministic rate-independent system described as follows:
\begin{itemize}
\item The state space is $Q=H^1(O)^d \times L^2(O)^k$ and the state variable is denoted by $q=(U,Z)$.
\item The energy functional is $\mathcal{E}_{\mathsf{hom}}:[0,T]\times Q\to \mathbb{R}$
\begin{equation*}
\mathcal{E}_{\mathsf{hom}}(t,q)=\int_{O} V_{\mathsf{hom}}(\nabla_s U,Z) dx- \int_{O}l(t)\cdot U dx,
\end{equation*}
where $V_{\mathsf{hom}}$ is given by the corrector problem: For $F_1, F_2\in \mathbb{R}^k$,
\begin{align*}
V_{\mathsf{hom}}(F_1,F_2)= & \inf_{\chi\in \mathbf{L}^2_{\mathsf{pot}}(\Omega)^d } \expect{A(\omega)\binom{F_1+\chi_s(\omega)}{F_2}\cdot \binom{ F_1+\chi_s(\omega)}{F_2}}.
\end{align*}
In fact, it can be shown that $V_{\mathsf{hom}}$ is quadratic: There exists $A_{\mathsf{hom}}\in \mathbb{R}^{2k \times 2k}_{sym}$ positive-definite such that $V_{\mathsf{hom}}(F_1, F_2)=A_{\mathsf{hom}}\binom{F_1}{F_2}\cdot \binom{F_1}{F_2}$ for all $F_1,F_2 \in \mathbb{R}^k$. 
\item The dissipation potential is given by $\Psi_{\mathsf{hom}}: Q \to [0,\infty)$
\begin{equation*}
\Psi_{\mathsf{hom}}(q)=\int_{O}\expect{\rho(\omega,Z(x))}dx.
\end{equation*}
\end{itemize}
The set of stable states at time $t\in [0,T]$ corresponding to the functionals $\brac{\mathcal{E}_{\mathsf{hom}},\Psi_{\mathsf{hom}}}$ is denoted by 
$
{S}_{\mathsf{hom}}(t)=\cb{q \in Q:\mathcal{E}_{\mathsf{hom}}(t,q)\leq \mathcal{E}_{\mathsf{hom}}(t,\widetilde{q})+\Psi_{\mathsf{hom}}(\widetilde{q}-q) \text{ for all }\widetilde{q}\in Q}.
$
\begin{remark}
If we assume (B0)-(B3), $l\in C^1([0,T],L^2(O)^d)$ and $q^0\in {S}_{\mathsf{hom}}(0)$, then, using Theorem \ref{abstract}, there exists a unique energetic solution $q \in C^{Lip}([0,T],Q)$ to the ERIS associated with $\brac{\mathcal{E}_{\mathsf{hom}},\Psi_{\mathsf{hom}}}$ with $q(0)=q^0$, i.e.,  for all $t\in [0,T]$ we have $q(t)\in S_{\mathsf{hom}}(t)$ and
\begin{equation}\label{enbal4}
\mathcal{E}_{\mathsf{hom}}(t,q(t))+\int_0^t \Psi_{\mathsf{hom}}(\dot{q}(s))ds = \mathcal{E}_{\mathsf{hom}}(0,q(0))-\int_0^t \int_O \dot{l}(s)\cdot U(s)dx ds.
\end{equation} 
\end{remark}
\begin{theorem}\label{evgamma4}
Assume $(B0)$-$(B3)$, $\expect{\cdot}$ is ergodic. Let $l \in C^1([0,T],L^2(O)^d)$, $y_{\varepsilon}^0\in S_{\varepsilon}^{\gamma}(0)$, $q_0\in Q$, $\chi \in (\mathbf{L}^2_{\mathsf{pot}}(\Omega)\otimes L^2(O))^d$ satisfy
\begin{equation*}
y_{\varepsilon}^0\overset{c2}{\rightarrow} (q_0, \chi), \quad \varepsilon^{\gamma}\nabla^{\varepsilon}z\e^0 \overset{2}{\to}0 \text{ in }L^2(\Omega\times \mathbb{R}^d)^{k\times d},\quad  \mathcal{E}_{\mathsf{hom}}(0,q_0)=\mathcal{E}_0(0,(q_0,\chi)).
\end{equation*}
Let $y\e \in C^{Lip}([0,T],Y\e)$ be the unique energetic solution associated with $\brac{\mathcal{E}\e^{\gamma},\Psi\e}$ and $y\e(0)=y\e^0$. Then $q_0\in S_{\mathsf{hom}}(0)$ and for every $t\in [0,T]$,
\begin{equation*}
 u\e(t) \overset{2}{\to}U(t) \text{ in }L^2(\Omega\times \mathbb{R}^d)^d, \quad z\e(t) \overset{2}{\to} Z(t) \text{ in }L^2(\Omega\times \mathbb{R}^{d})^k,
\end{equation*}
where $q=(U,Z)\in C^{Lip}([0,T],Q)$ is the unique energetic solution associated with $\brac{\mathcal{E}_{\mathsf{hom}},\Psi_{\mathsf{hom}}}$ and $q(0)=q_0$.
\end{theorem}
\begin{remark}
A close look at the proof of Theorem~\ref{evgamma4} shows that in addition we have for all $t\in [0,T]$
\begin{equation*}
\nabla^{\varepsilon}u\e(t) \overset{2}{\rightharpoonup} \nabla U(t)+ \chi(t) \text{ in }L^2(\Omega\times \mathbb{R}^d)^{d\times d}, \quad \varepsilon^{\gamma}\nabla^{\varepsilon}z\e(t) \overset{2}{\rightharpoonup} 0 \text{ in }L^2(\Omega\times \mathbb{R}^d)^{k\times d},
\end{equation*}
where $\chi(t)\in (\mathbf{L}^2_{\mathsf{pot}}(\Omega)\otimes L^2(O))^d$ is uniquely determined by the identity
\begin{eqnarray*}
&& V_{\mathsf{hom}}(\nabla_s U(t)(x),Z(t)(x))\\ &=& \expect{A(\omega)\binom{\nabla_s U(t)(x)+\chi_s(t)(\omega,x)}{Z(t)(x)}\cdot \binom{\nabla_s U(t)(x)+\chi_s(t)(\omega,x)}{Z(t)(x)}}
\end{eqnarray*}
for a.e. $x\in O$.
\end{remark}
\subsection{Proofs}\label{section_1111}
Before presenting the proofs, we consider a couple of auxiliary lemmas.
\begin{lemma}[Two-scale convergence of symmetrized gradients]\label{symmet} 
We consider $u_{\varepsilon}\in L^{p}(\Omega \times \varepsilon \mathbb{Z}^d)^d$ and $F\in L^{p}(\Omega\times \mathbb{R}^d)^{d\times d}$ such that $\nabla^{\varepsilon} u\e \overset{2}{\rightharpoonup} F$ in $ L^{p}(\Omega\times \mathbb{R}^d)^{d \times d} $. Then
\begin{equation*}
\nabla_{s}^{\varepsilon}u_{\varepsilon}\overset{2}{\rightharpoonup}  F_s \quad \text{in }L^{p}(\Omega\times \mathbb{R}^d)^k.
\end{equation*}
If we have strong two-scale convergence for $\nabla^{\varepsilon}u_{\varepsilon}$, strong two-scale convergence for $\nabla^{\varepsilon}_s u\e$ follows.
\end{lemma}
\begin{proof}
For any $i\in \cb{1,...,k}$ we compute
\begin{align*}
\mathcal{T}_{\varepsilon} (\nabla_{s}^{\varepsilon}u\e)_i(\omega,x)&=\frac{b_i}{|b_i|}\cdot \mathcal{T}_{\varepsilon} \sum_{y\in \mathbb{Z}^d} \nabla^{\varepsilon}u\e(\omega,x-\varepsilon y)B_i(y)\\&=\frac{b_i}{|b_i|}\cdot \sum_{y\in \mathbb{Z}^d} \mathcal{T}_{\varepsilon}\nabla^{\varepsilon}u\e(T_{-y}\omega,x-\varepsilon y)B_i(y).
\end{align*}
For any fixed $y\in \mathbb{Z}^d$, the function $(\omega,x)\mapsto\mathcal{T}_{\varepsilon}\nabla^{\varepsilon}u\e(T_{-y}\omega,x-\varepsilon y)B_i(y)$ weakly converges to $(\omega,x)\mapsto F(T_{-y}\omega,x)B_i(y)$. If we assume strong two-scale convergence for the gradient, the previous quantities converge in the strong sense. Using this and the fact that $B_i(y)=0$ for all but finitely many $y\in \mathbb{Z}^d$, the claim follows.
\end{proof}
\begin{lemma}[Stochastic Korn's inequality]\label{kornstoch}
Recall that it is assumed that $(\mathbb{Z}^d, \mathsf{E})$ satisfies (\ref{Korn_assumpt}) (see Section \ref{lattice_graphs}). There exists $C>0$ such that
\begin{equation*}
\expect{|\chi|^p}\leq C \expect{|\chi_s|^p} \quad \text{for all }\chi \in \mathbf{L}^p_{\mathsf{pot}}(\Omega)^d.
\end{equation*}
\end{lemma}
\begin{proof}
We show the inequality in the case $\chi=D\varphi$ for $\varphi\in L^p(\Omega)^d$. For general $\chi\in \mathbf{L}^p_{\mathsf{pot}}(\Omega)^d$, it is obtained by an approximation argument. We denote by $\widetilde{\varphi}: \Omega\times \mathbb{Z}^d\to \mathbb{R}^{d}$ the stationary extension of $\varphi$, i.e.,  $\widetilde{\varphi}(\omega,x)=\varphi(T_x \omega)$. In this proof we consider the operators $\nabla^{\varepsilon}$ and $\nabla^{\varepsilon}_s$ with $\varepsilon=1$ and for notational convenience we drop the index $\varepsilon$ and simply write $\nabla$ and $\nabla_s$.

Let $R>0$ and $K>0$ such that $K>\sup\cb{|b_i|:b_i \in \mathsf{E}_0}$. Let $\eta_{R}$ be a cut-off function given by $\eta_{R}: \mathbb{Z}^d \rightarrow \mathbb{R}$ with $\eta_R=1$ in $B_{R+K}\cap\mathbb{Z}^d$  and with $\eta_{R}=0$ otherwise ($B_R\subset \mathbb{R}^d$ is a ball of radius $R$ with center in 0). Using the properties of $\eta_{R}$ and the discrete Korn's inequality, we obtain 
\begin{eqnarray*}
&& \expect{\dashint_{B_R \cap \mathbb{Z}^d}|{D\varphi}(T_x \omega)|^p dm(x)} \\ &\leq & \expect{\frac{1}{|B_R|}\int_{\mathbb{Z}^d}|\nabla(\widetilde{\varphi}(\omega,x)\eta_R(x))|^p dm(x)} \\ & \leq & \expect{\frac{C}{|B_R|}\int_{\mathbb{Z}^d}|\nabla_{s}(\widetilde{\varphi}(\omega,x)\eta_{R}(x))|^p dm(x)}\\ & =& \expect{C\dashint_{B_R\cap \mathbb{Z}^d}|\nabla_{s}(\widetilde{\varphi}(\omega,x))|^p dm(x)}\\ && +\expect{\frac{C}{|B_R|}\int_{(B_{R+2K}\setminus B_R)\cap \mathbb{Z}^d}|\nabla_{s}(\widetilde{\varphi}(\omega,x)\eta_{R}(x))|^p dm(x)}.
\end{eqnarray*}
By invariance of $P$ the left-hand side of the above inequality equals $\expect{|D\varphi|^p}$ for any $R$.
Moreover, the first term on the right-hand side equals $C\expect{| D_s \varphi|^p}$. Therefore, it is sufficient to show that the second term vanishes in the limit $R\rightarrow \infty$. To obtain that, we estimate (for $P$-a.e. $\omega\in \Omega$)
\begin{eqnarray}\label{equation_1092}
&& \frac{1}{|B_R|}\int_{(B_{R+2K}\setminus B_R)\cap \mathbb{Z}^d}|\nabla_{s}(\widetilde{\varphi}(\omega,x)\eta_R(x))|^p dm(x)  \\ & \leq & \frac{C}{|B_R|}\int_{(B_{R+2K}\setminus B_{R-K})\cap \mathbb{Z}^d} |\widetilde{\varphi}(\omega,x)|^p|\eta_R(x)|^p dm(x) \nonumber \\ & \leq & \frac{C}{|B_R|}\int_{(B_{R+2K}\setminus B_{R-K})\cap \mathbb{Z}^d} |\widetilde{\varphi}(\omega,x)|^p dm(x) \nonumber \\ & = & \frac{C}{|B_R|}\int_{B_{R+2K}\cap \mathbb{Z}^d} |\widetilde{\varphi}(\omega,x)|^p dm(x) -\frac{C}{|B_R|}\int_{B_{R-K}\cap \mathbb{Z}^d} |\widetilde{\varphi}(\omega,x)|^p dm(x)\nonumber \\ & = &
\frac{C|B_{R+2K}|}{|B_R|}\dashint_{B_{R+2K}\cap \mathbb{Z}^d} |\widetilde{\varphi}(\omega,x)|^p dm(x) -\frac{C|B_{R-K}|}{|B_R|}\dashint_{B_{R-K}\cap \mathbb{Z}^d} |\widetilde{\varphi}(\omega,x)|^p dm(x).\nonumber
\end{eqnarray}
In the first inequality above, we used the fact that $\nabla_s: L^p(\varepsilon\mathbb{Z}^{d})^d\to L^p(\varepsilon\mathbb{Z}^{d})^k$ is a bounded operator. An integration of (\ref{equation_1092}) over $\Omega$ yields 
\begin{eqnarray*}
&& \expect{\frac{C}{|B_R|}\int_{(B_{R+2K}\setminus B_R)\cap \mathbb{Z}^d}|\nabla_{s}(\widetilde{\varphi}(\omega,x)\eta_{R}(x))|^p dm(x)}\\ &\leq & \expect{|\varphi|^p} \brac{\frac{C|B_{R+2K}|}{|B_R|}- \frac{C|B_{R-K}|}{|B_R|}} {\to} 0 \quad \text{as }R\to \infty.
\end{eqnarray*}
This concludes the proof.
\end{proof}

The above result and the direct method of calculus of variations imply the following.
\begin{lemma}\label{corector} For any $F\in \mathbb{R}^{d\times d}$, there exists $\chi\in \mathbf{L}^p_{\mathsf{pot}}(\Omega)^d$ which attains the infimum in (\ref{problem}) in Section \ref{section_870}.
\end{lemma}
\smallskip

\begin{proof}[Proof of Theorem \ref{gammatheorem}]
(i) The growth conditions of $V$, the Korn property of the lattice and a discrete Poincar\'e inequality imply 
\begin{equation*}
\limsup_{\varepsilon\rightarrow 0}\left\langle \int_{\varepsilon\mathbb{Z}^{d}} |u_{\varepsilon}(\omega,x)|^p+|\nabla^{\varepsilon}u_{\varepsilon}(\omega,x)|^p dm_{\varepsilon}(x) \right\rangle<\infty.
\end{equation*}
Therefore, the claim follows by Proposition \ref{comp3} (i) and Corollary \ref{cor111}. 
\smallskip

(ii) The claim directly follows from Lemma \ref{symmet} and Proposition \ref{pro67} (i). 
\smallskip

(iii) The claim follows from Corollary \ref{rem5} (iii), Lemma \ref{symmet} and Proposition \ref{pro67} (ii). 
\end{proof}
\smallskip

\begin{proof}[Proof of Theorem \ref{gamma2}]
(i) By Theorem \ref{gammatheorem} there exist $U\in W^{1,p}_0(O)^d$, $\chi\in (\mathbf{L}^p_{\mathsf{pot}}(\Omega)\otimes L^p(O))^d$, and a two-scale convergent subsequence such that  
\begin{equation*}
\liminf_{\varepsilon\rightarrow 0} \mathcal{E}\e(u_{\varepsilon})\geq \mathcal{E}_{0}(U,\chi)\geq \mathcal{E}_{\mathsf{hom}}(U). 
\end{equation*}
The corresponding convergence for $u\e$ follows from Corollary \ref{ergodic} and Lemma \ref{lemma_equivalent_conv}.
\smallskip

(ii) It is sufficient to show that for $U\in W^{1,p}_0(O)$, there exists $\chi\in (\mathbf{L}^p_{\mathsf{pot}}(\Omega)\otimes L^p(O))^d$ such that
\begin{equation*}
\mathcal{E}_0(U,\chi)=\mathcal{E}_{\mathsf{hom}}(U).
\end{equation*}
Indeed, if this holds, Theorem \ref{gammatheorem} (iii) implies that there is a $u\e \in L^p(\Omega)\otimes L^p_0(O\cap \varepsilon\mathbb{Z}^{d})^d$ such that
\begin{align*}
& u_{\varepsilon} \overset{2}{\rightarrow} U \text{ in } L^p(\Omega\times \mathbb{R}^d)^d, \quad 
\nabla^{\varepsilon}u_{\varepsilon}\overset{2}{\rightarrow} \nabla U+\chi \text{ in } L^{p}(\Omega\times \mathbb{R}^d)^{d\times d},\\
& \lim_{\varepsilon\rightarrow 0}\mathcal{E}\e(u_{\varepsilon})=\mathcal{E}_{0}(U,\chi)=\mathcal{E}_{\mathsf{hom}}(U)
\end{align*}
and the corresponding convergence for $\expect{u\e}$ and $\expect{\nabla^{\varepsilon}u\e}$ follows from Corollary \ref{ergodic} and Lemma \ref{lemma_equivalent_conv}.

To show the above claim, we note that $\nabla U \in L^p(O)^{d\times d}$ and consider a sequence of piecewise constant functions $F_n(x)=\sum_{i=1}^{k(n)} \mathbf{1}_{O_i^n}(x) F_i^{n}$ where $F_i^n \in \mathbb{R}^{d\times d}$ and such that
\begin{equation*}
F_n \to \nabla U \text{ strongly in } L^p(O)^{d\times d}.
\end{equation*} 
For any $F_i^n$, there is $\chi_i^n \in \mathbf{L}^p_{\mathsf{pot}}(\Omega)^d$ such that $V_{\mathsf{hom}}(F_i^n)=\expect{V(\omega,(F_i^n)_s+(\chi_i^n)_s(\omega))}$. We define $\chi_n(x,\omega)=\sum_{i=1}^{k(n)}\mathbf{1}_{O_i^n}(x)\chi_i^n(\omega)$. Noting that $\chi_n\in (\mathbf{L}^p_{\mathsf{pot}}(\Omega)\otimes L^p(O))^d$ and with the help of the growth assumptions (A3) and the stochastic Korn's inequality (Lemma \ref{kornstoch}), we obtain
\begin{equation*}
\limsup_{n\to \infty} \expect{\int_O |\chi_n(x,\omega)|^p dx}< \infty.
\end{equation*}
As a result of this, there exist a subsequence $n'$ and $\chi \in (\mathbf{L}^p_{\mathsf{pot}}(\Omega)\otimes L^p(O))^d$ such that 
\begin{equation*}
\chi_{n'} \rightharpoonup \chi \text{ weakly in } (\mathbf{L}^p_{\mathsf{pot}}(\Omega)\otimes L^p(O))^d.
\end{equation*}
Furthermore, we have
\begin{align*}
\expect{ \int_O V(\omega, \nabla_s U(x)+ \chi_s(\omega,x)) dx} & \leq \liminf_{n'\to \infty } \int_O \expect{ V(\omega, (F_{n'})_s(x)+(\chi_{n'})_s(\omega,x))}dx \\
& = \liminf_{n'\to \infty} \int_O V_{\mathsf{hom}}(F_{n'}(x))dx= \int_O V_{\mathsf{hom}}(\nabla U(x))dx.
\end{align*}
Above, in the first inequality we use weak lower-semicontinuity of the functional $G\in L^p(\Omega \times O)^{d\times d}\mapsto \expect{\int_O V(\omega, G(\omega,x))dx}$ and the facts that $F_{n'}\to \nabla U$, $\chi_{n'}\rightharpoonup \chi$ and  
that $(\cdot)_s$ is a linear and bounded operator. In order to justify the last equality, we remark that $V_{\mathsf{hom}}$ is convex (and therefore continuous) and satisfies the growth assumptions $-C-C|F|^p \leq V_{\mathsf{hom}}(F)\leq C|F|^p+C$ which implies that the functional $G\in L^p(O)^{d\times d}\mapsto \int_O V_{\mathsf{hom}}(G(x))dx$ is strongly continuous. 

On the other hand, it is easy to see that 
\begin{equation*}
\int_O V_{\mathsf{hom}}(\nabla U(x))dx\leq \expect{\int_O V(\omega,  \nabla_s U(x)+ \chi_s(\omega,x))dx}. 
\end{equation*}
This concludes the proof.
\end{proof}
\smallskip

\begin{proof}[Proof of Proposition \ref{prop41}]
The uniqueness of the minimizer in (\ref{minimum}) follows by uniform convexity assumption on the integrand $V$. Theorem \ref{gamma2} implies that (up to a subsequence) $u\e\overset{2}{\rightharpoonup} U$ in $L^p(\Omega\times \mathbb{R}^d)^d$, where $U$ is a minimizer of $\mathcal{I}_{\mathsf{hom}}$. As in the proof of Theorem \ref{gamma2}, we find $\chi \in (\mathbf{L}^p_{\mathsf{pot}}(\Omega)\otimes L^p(O))^d$ with $\mathcal{I}_{\mathsf{hom}}(U)=\mathcal{E}_{0}(U,\chi)-\int_O l\cdot U dx$. By Theorem \ref{gammatheorem}, there exists a strong two-scale recovery sequence $v\e \in \brac{L^p(\Omega)\otimes L^p_0(O\cap \varepsilon\mathbb{Z}^{d})}^d$ for $(U,\chi)$. We have
\begin{eqnarray*}
&& \expect{\int_{\mathbb{R}^d}|\mathcal{T}_{\varepsilon} u\e(\omega,x) -U(x)|^p dx}\\ & \leq & C \brac{\expect{\int_{\mathbb{R}^d}|\mathcal{T}_{\varepsilon} u\e(\omega,x)-\mathcal{T}_{\varepsilon} v\e(\omega,x)|^p dx}  +\expect{\int_{\mathbb{R}^d}|\mathcal{T}_{\varepsilon} v\e(\omega,x)-U(x)|^p dx}}.
\end{eqnarray*} 
The second term on the right-hand side vanishes in the limit $\varepsilon\to 0$ by the properties of $v\e$. The first term also vanishes as $\varepsilon\to 0$ and this follows by a standard argument using strong convexity: By the isometry property of $\mathcal T\e$,  a discrete Poincar{\'e}-Korn inequality following from (\ref{Korn_assumpt}), the strong convexity $(A4)$, and since $\nabla_s^{\varepsilon}u_{\varepsilon}$ and $\nabla_{s}^{\varepsilon}v_{\varepsilon}$ are supported in $O^{+\varepsilon}\cap \varepsilon \mathbb{Z}^d$,
\begin{equation}\label{equation:1111}
\begin{aligned}
  &\expect{\int_{\mathbb{R}^d}|\mathcal{T}_{\varepsilon} u\e(\omega,x)-\mathcal{T}_{\varepsilon} v\e(\omega,x)|^p dx}   \leq  C \expect{\int_{\varepsilon\mathbb{Z}^{d}}|\nabla^{\varepsilon}_s u\e(\omega,x)- \nabla^{\varepsilon}_s v\e(\omega,x)|^p dm_{\varepsilon}(x)}\\
  & \leq  C \brac{ \frac{1}{2}\big(\mathcal{E}_{\varepsilon}(u_{\varepsilon}) + \mathcal{E}_{\varepsilon}(v\e)\big) - \mathcal{E}_{\varepsilon}\brac{\frac{1}{2}u\e+\frac{1}{2}v_{\varepsilon}} }.
 \end{aligned}
\end{equation}
Since $u_{\varepsilon}$ solves (\ref{minimum}), we have
\begin{align*}
-\mathcal{E}_{\varepsilon}\brac{\frac{1}{2}u_{\varepsilon}+\frac{1}{2}v\e} \leq -\mathcal{E}_{\varepsilon}(u_{\varepsilon})+ \expect{\int_{O\cap \varepsilon\mathbb{Z}^d}l_{\varepsilon}\cdot u_{\varepsilon}}- \expect{\int_{O\cap \varepsilon \mathbb{Z}^d}l_{\varepsilon}\cdot \brac{\frac{1}{2}u\e +\frac{1}{2}v\e}},
\end{align*}
and thus with (\ref{equation:1111}),
\begin{eqnarray*}
\expect{\int_{\mathbb{R}^d}|\mathcal{T}_{\varepsilon} u\e(\omega,x)-\mathcal{T}_{\varepsilon} v\e(\omega,x)|^p dx} \leq  C \brac{ \frac{1}{2}\big(\mathcal{E}_{\varepsilon}(v\e) -\mathcal{E}_{\varepsilon}(u\e)\big) +\frac{1}{2}\expect{\int_{O\cap \varepsilon \mathbb{Z}^d}l\e \cdot \brac{u\e - v\e}} }.
\end{eqnarray*}
The last term on the right-hand side vanishes as $\varepsilon\to 0$ (using strong two-scale convergence of $l\e$). Similarly as in the proof of Theorem \ref{gamma2}, it follows that $\limsup_{\varepsilon\to 0}(\mathcal{E}\e(v\e)-\mathcal{E}\e(u\e))\leq 0$ and therefore  $\limsup_{\varepsilon\to 0}\expect{\int_{\mathbb{R}^d}|\mathcal{T}_{\varepsilon} u\e(\omega,x)-\mathcal{T}_{\varepsilon} v\e(\omega,x)|^p dx} \leq 0$.
This yields the claim for a subsequence. Convergence for the whole sequence follows by a contradiction argument and using the uniqueness of the minimizer $U$.
\end{proof}
\smallskip

\begin{lemma}\label{lemm12} Let $\mathbb{A}$ and $Y$ be defined as in Section \ref{sectionERIS} (with the same assumptions as in Remark \ref{remark_1030}). There exists $C>0$ such that $\expect{\mathbb{A}y,y}_{Y^*,Y}\geq C \|y\|^2_{Y}$ for all $y\in Y$.
\end{lemma}
\begin{proof}
First, we have $\expect{\mathbb{A}y,y}_{Y^*,Y}\geq C \expect{\int_O |\nabla_s U(\omega,x)+\chi_s(\omega,x)|^2+|Z(\omega,x)|^2}$. Note that, $\nabla U $ does not depend on $\omega$ and therefore $\int_O\expect{\nabla_s U(x)\cdot \chi_s(\omega,x)}dx=0$. This implies
\begin{equation*}
\expect{\int_O |\nabla_s U(x)+\chi_s(\omega,x)|^2 dx}=\int_O|\nabla_s U(x)|^2dx+\int_O\expect{|\chi_s(\omega,x)|^2 }dx.
\end{equation*}
Using the continuum Korn's inequality (\ref{continuum_korn}) (Section \ref{lattice_graphs}) and the stochastic Korn's inequality (Lemma \ref{kornstoch}), we conclude the proof.
\end{proof}
\smallskip

Before proving Theorem \ref{evgamma}, we prove as an auxiliary result, the existence of joint recovery sequences, which implies the stability of two-scale limits of solutions.
\begin{lemma}\label{stab111}
Let $t\in [0,T]$ and $y_{\varepsilon}\in S_{\varepsilon}(t)$ such that $y_{\varepsilon}\overset{c2}{\rightharpoonup}y\in Y$. For any $\widetilde{y}\in Y$ there exists $\widetilde{y}_{\varepsilon}\in Y_{\varepsilon}$ such that $\widetilde{y_{\varepsilon}}\overset{c2}{\rightharpoonup}{\widetilde{y}}$ and 
\begin{equation*}
\lim_{\varepsilon\rightarrow 0}\cb{\mathcal{E}\e(t,\widetilde{y}_{\varepsilon})+\Psi\e(\widetilde{y}_{\varepsilon}-y_{\varepsilon})-\mathcal{E}\e(t,y_{\varepsilon})}= \mathcal{E}_{0}(t,\widetilde{y})+\Psi_0(\widetilde{y}-y)-\mathcal{E}_{0}(t,y).
\end{equation*}
This implies $y\in S(t)$.
\end{lemma}
\begin{proof}
Corollary \ref{rem5} (i) implies that there exists a sequence $v_{\varepsilon}\in L^2(\Omega)\otimes L^2_0(O\cap \varepsilon\mathbb{Z}^{d})^d$ with
\begin{equation*}
v_{\varepsilon}\overset{2}{\rightarrow}\widetilde U-U \text{ in } L^{2}(\Omega\times \mathbb{R}^d)^d,\quad
 \nabla^{\varepsilon}v_{\varepsilon}\overset{2}{\rightarrow}\nabla\widetilde{U}-\nabla U+ \widetilde{\chi}-\chi
 \text{ in }L^{2}(\Omega\times \mathbb{R}^d)^{d\times d}.
\end{equation*}
The sequence $g\e\in \brac{L^2(\Omega)\otimes L^2_0(O^{+\varepsilon}\cap \varepsilon\mathbb{Z}^{d})}^k$, given by $g\e=\mathbf{1}_{O^{+\varepsilon}}\mathcal{F}\e (\widetilde{Z}-Z)$, satisfies
\begin{equation*}
g\e\overset{2}{\to} \widetilde{Z}-Z \quad \text{in }L^{2}(\Omega\times \mathbb{R}^d)^k.
\end{equation*}
We define $\widetilde{y}\e$ componentwise: $\widetilde{u}\e=u\e+v\e$ and $ \widetilde{z}\e=z\e + g\e.$
By weak two-scale convergence of $y\e$, we have that $\widetilde{y}\e\overset{c2}{\rightharpoonup} \widetilde{y}$, and furthermore $\widetilde{y}\e-y\e\overset{c2}{\rightarrow}\widetilde{y}-y$.

The energy functional is quadratic and thus it satisfies
\begin{eqnarray}\label{quadr}
&& \mathcal{E}\e(t, \widetilde{y}\e)-\mathcal{E}\e(t,y\e) \\ & = & \frac{1}{2}\expect{\int_{O^{+\varepsilon}\cap \varepsilon\mathbb{Z}^{d}} A(T_{\frac{x}{\varepsilon}}\omega)\binom{\nabla_{s}^{\varepsilon}(\widetilde{u}\e-u\e)(\omega,x)}{(\widetilde{z}\e-z\e)(\omega,x)}\cdot \binom{\nabla_{s}^{\varepsilon}(\widetilde{u}\e+u\e)(\omega,x)}{(\widetilde{z}\e+z\e)(\omega,x)}dm_{\varepsilon}(x)} \nonumber\\ & & -\expect{\int_{O\cap \varepsilon\mathbb{Z}^{d}}\pi\e l(t)(x) \cdot (\widetilde{u}\e-u\e)(\omega,x) dm_{\varepsilon}(x)}.\nonumber
\end{eqnarray}
We rewrite the first term on the right-hand side as
\begin{equation*}
\frac{1}{2}\expect{\int_{\mathbb{R}^d}A(\omega)\binom{\mathcal{T}\e\nabla_{s}^{\varepsilon}(\widetilde{u}\e-u\e)(\omega,x)}{\mathcal{T}\e(\widetilde{z}\e-z\e)(\omega,x)}\cdot \binom{\mathcal{T}\e\nabla_{s}^{\varepsilon}(\widetilde{u}\e+u\e)(\omega,x)}{\mathcal{T}\e(\widetilde{z}\e+z\e)(\omega,x)}dx}.
\end{equation*}
This expression is a scalar product of strongly and weakly convergent sequences (see Lemma \ref{symmet}), and therefore it converges to (as $\varepsilon\rightarrow 0$) 
\begin{eqnarray*}
&& \frac{1}{2}\expect{\int_{\mathbb{R}^d} A(\omega) \binom{\nabla_s \widetilde{U}_s-\nabla_s U+\widetilde{\chi}_s-\chi_s}{\widetilde{Z}-Z}\cdot \binom{\nabla_s \widetilde{U}+\nabla_s U+\widetilde{\chi}_s+\chi_s}{\widetilde{Z}+Z}dx}\\ & = & \frac{1}{2}\brac{\expect{\mathbb{A}\widetilde{y},\widetilde{y}}_{Y^*,Y}-\expect{\mathbb{A}y,y}_{Y^*,Y}}.
\end{eqnarray*}
The second term on the right-hand side of (\ref{quadr}) converges to $-\expect{\int_{O}l(t)\cdot (\widetilde{U}-U)}$. 
Furthermore, by Jensen's inequality we obtain
\begin{align*}
\Psi\e(\widetilde{y}\e-y\e) & \leq \expect{\int_{\varepsilon\mathbb{Z}^{d} }\rho(\omega,\pi\e(\widetilde{Z}-Z)(\omega,x))dm_{\varepsilon}(x)} \\ & \leq \expect{\int_{\varepsilon\mathbb{Z}^{d} }\dashint_{x+\varepsilon \Box}\rho(\omega,\widetilde{Z}(\omega,y)-Z(\omega,y))dy dm\e(x)}\\ & =\Psi_0(\widetilde{y}-y).
\end{align*}
On the other hand, using Fatou's lemma and the fact that $\rho(\omega,\cdot)$ is continuous, we obtain
\begin{equation*}
\liminf_{\varepsilon\rightarrow 0}\Psi\e(\widetilde{y}\e-y\e)\geq \Psi_0(\widetilde{y}-y).
\end{equation*}
This concludes the proof.
\end{proof}
\smallskip

\begin{proof}[Proof of Theorem  \ref{evgamma}] Step 1. Compactness and stability.
We consider the sequence $v_{\varepsilon}:=(\mathcal{T}\e u\e,\mathcal{T}\e \nabla^{\varepsilon} u\e,\mathcal{T}\e z\e ):[0,T]\rightarrow L^{2}(\Omega\times \mathbb{R}^d)^d\times L^{2}(\Omega\times \mathbb{R}^d)^{d\times d}\times L^{2}(\Omega\times \mathbb{R}^d)^k =:H$. By Remark \ref{remark10}, $v_{\varepsilon}$ is uniformly bounded in $C^{Lip}([0,T],H)$. Therefore, the Arzel\`a-Ascoli theorem implies that there exist $v\in C^{Lip}([0,T],H)$ and a subsequence (not relabeled), such that for every $t\in[0,T]$ 
\begin{equation*}
v_{\varepsilon}(t)\rightharpoonup v(t) \text{ weakly in }H.
\end{equation*}
Moreover, by boundedness of $y\e(t)$ and the above, we conclude that for every $t\in [0,T]$, $v(t)=(U(t), \nabla U(t)+\chi(t),Z(t))$, for some $y(t)=(U(t),Z(t),\chi(t))\in Y$. Here we use the fact that if $z\e\in L^2(\Omega)\otimes L^2_0(O\e^+\cap \varepsilon\mathbb{Z}^{d})^k$ converges in the weak two-scale sense, then, similarly as in Corollary \ref{cor111}, the limit may be identified with an $L^2(\Omega\times O)^k$ function. In other words, we have $y_{\varepsilon}(t)\overset{c2}{\rightharpoonup}y(t)$.
Lemma \ref{stab111} implies that $y(t)\in S(t)$ for every $t\in [0,T]$.

Step 2. Energy balance. We pass to the limit $\varepsilon\to 0$ in (\ref{enbal1}) and show that $y(t)$ satisfies
\begin{equation}\label{ebcont}
\mathcal{E}_{0}(t,y(t))+\int_0^t \Psi_0(\dot{y}(s))ds\leq \mathcal{E}_{0}(0,y(0))-\int_0^t \int_O \dot{l}(s)\cdot U(s) dx ds.
\end{equation}
The (EB) equality of the discrete system (\ref{enbal1}) reads
\begin{align}\label{disc}
& \frac{1}{2}\expect{\mathbb{A}\e y\e(t),y\e(t)}_{Y\e^*,Y\e}-\expect{\int_{O\cap \varepsilon\mathbb{Z}^{d}}\pi\e l(t)(x)\cdot u\e(t)(\omega,x)dm_{\varepsilon}(x)}+\int_{0}^t \Psi\e(\dot{y}\e(s))ds  \\  = &  \frac{1}{2}\expect{\mathbb{A}\e y\e(0),y\e(0)}_{Y\e^*,Y\e}-\expect{\int_{O\cap \varepsilon\mathbb{Z}^{d} }\pi\e l(0)(x)\cdot u\e(0)(\omega,x)dm_{\varepsilon}(x)} \nonumber \\ & -\int_0^t \expect{\int_{O\cap \varepsilon\mathbb{Z}^{d}}\pi\e \dot{l}(s)(x)\cdot u\e(s)(\omega,x)dm_{\varepsilon}(x)}ds.\nonumber
\end{align}
The strong convergence of the initial data implies that the first two terms on the right-hand side converge to $\mathcal{E}(0,y(0))$. The remaining term on the right-hand side converges to $-\int_0^t \int_O \dot{l}(s)\cdot U(s) dx ds$ by the Lebesgue dominated convergence theorem. Moreover, using Proposition \ref{pro67} and the strong convergence of $\pi\e l(t)$ we obtain
\begin{eqnarray*}
&&
\liminf_{\varepsilon\to 0}\brac{\frac{1}{2}\expect{\mathbb{A}\e y\e(t),y\e(t)}_{Y\e^*,Y\e}-\expect{\int_{O\cap \varepsilon\mathbb{Z}^{d}}\pi\e l(t)(x)\cdot u\e(t)(\omega,x)dm_{\varepsilon}(x)}} \\&\geq & \mathcal{E}_{0}(t,y(t)).
\end{eqnarray*}
To treat the last term on the left-hand side of (\ref{disc}), we consider a partition  $\cb{t_i}$ of $[0,t]$. We have
\begin{equation*}
\sum_i \Psi_0(y(t_{i})-y(t_{{i-1}}))\leq \liminf_{\varepsilon\rightarrow 0} \sum_i \Psi\e(y\e(t_{i})-y\e(t_{{i-1}})).
\end{equation*} 
Taking the supremum over all partitions $\cb{t_i}$ of $[0,t]$ and exploiting the homogeneity of $\Psi_0$, we obtain
\begin{equation}\label{dissipation10}
\int_{0}^t \Psi_0(\dot{y}(s))dt\leq \liminf_{\varepsilon \rightarrow 0}\int_{0}^t\Psi\e(\dot{y}\e(s))ds.
\end{equation} 
This proves (\ref{ebcont}). The other inequality in the (EB) equality of the limit system can be shown using the stability of $y$ (see \cite[Section 2.3.1]{mielke2005evolution}) and therefore we conclude that $y$ satisfies (\ref{enbal2}). Moreover, using this equality and the fact that the right-hand side of (\ref{disc}) converges to $\mathcal{E}_0(0,y(0))-\int_{0}^t\int_{O}\dot{l}(s)\cdot U(s) dx ds$, we conclude that
\begin{eqnarray}\label{useful:equation}
&& \lim_{\varepsilon\to 0} \brac{\mathcal{E}\e(t,y\e(t))+\int_{0}^{T}\Psi\e(\dot{y}\e(s))ds}\\ &=& \lim_{\varepsilon\to 0}\brac{\mathcal{E}_{\varepsilon}(0,y\e(0))-\int_{0}^t\expect{\int_{O\cap \varepsilon\mathbb{Z}^d} \pi\e \dot{l}(s)\cdot u\e(s)dm\e}}ds \nonumber\\ & = & \mathcal{E}_0(0,y(0))-\int_{0}^t\int_{O}\dot{l}(s)\cdot U(s) dx ds = \mathcal{E}_{0}(t,y(t))+\int_0^t \Psi_0(\dot{y}(s))ds. \nonumber
\end{eqnarray}
\smallskip

Step 3. Strong convergence.
To obtain strong two-scale convergence, we construct a strong recovery sequence $\widetilde{y}\e(t)\in Y\e$ for $y(t)\in Y$ (for every $t\in [0,T]$) in the sense that
\begin{equation*}
\widetilde{y}\e(t)\overset{c2}{\rightarrow}y(t),
\end{equation*}
(cf. the proof of Lemma \ref{stab111}). For notational convenience, we drop the ``t" from the sequences and we denote $v\e{}:=(\mathcal{T}\e u\e{},\mathcal{T}\e\nabla^{\varepsilon}u\e{},\mathcal{T}\e z\e{})$, $\widetilde{v}\e{}:=(\mathcal{T}\e \widetilde{u}\e{},\mathcal{T}\e\nabla^{\varepsilon}\widetilde{u}\e{},\mathcal{T}\e \widetilde{z}\e{})$ and $V{}:=(U{},\nabla U{}+\chi{},Z{})$. By the triangle inequality,
\begin{equation}\label{inequality123}
\|v\e{} - V{}\|_H \leq \|v\e{} - \widetilde{v}\e{} \|_H+\|\widetilde{v}\e{} - V{}\|_H.  
\end{equation}
The second term on the right-hand side vanishes in the limit $\varepsilon\rightarrow 0$. Also, since the energy is quadratic,
\begin{multline*}
\|v\e{} - \widetilde{v}\e{}\|^2_H  \leq  C \bigg( \mathcal{E}\e(t,y\e{})-\mathcal{E}(t,\widetilde{y}\e{})+\expect{\mathbb{A}\e\widetilde{y}\e,\widetilde{y}\e-y\e}_{Y\e^*,Y\e}\\+\expect{\int_{O\cap \varepsilon\mathbb{Z}^{d}} \pi\e l(t)(x) \cdot (u\e-\widetilde{u}\e)(\omega,x)dm_{\varepsilon}(x)}\bigg).
\end{multline*}
The last two terms on the right-hand side vanish as $\varepsilon\to 0$ (cf. the proof of Lemma \ref{stab111}).
The first two terms are treated as follows. 
As a result of (\ref{useful:equation}), we obtain that $\limsup_{\varepsilon\rightarrow 0}\mathcal{E}\e(t,y\e)+\liminf_{\varepsilon\rightarrow 0}\int_{0}^{T}\Psi\e(\dot{y}\e(s))ds=\mathcal{E}_{0}(t,y)+\int_{0}^t \Psi_0(\dot{y}(s))ds$ and using (\ref{dissipation10}) it follows that
\begin{align*}
\limsup_{\varepsilon\rightarrow 0}(\mathcal{E}\e(t,y\e)-\mathcal{E}\e(t,\widetilde{y}\e))\leq \limsup_{\varepsilon\rightarrow 0}\mathcal{E}\e(t,y\e)-\mathcal{E}_{0}(t,y)\leq 0.
\end{align*}
This shows that the first two terms on the right-hand side of (\ref{inequality123}) vanish in the limit $\varepsilon\to 0$ and therefore the claim about strong convergence follows. 

To show that the convergence holds for the whole sequence, for a fixed $t\in [0,T]$, we consider $e\e(t):=\|v\e(t)-V(t)\|_H$.
For any subsequence $\varepsilon'$ of $\varepsilon$, we can find a further subsequence $\varepsilon''$ such that $e_{\varepsilon''}(t)\rightarrow 0$ by the uniqueness of the solution $y$. From this follows that the whole sequence converges in the sense given in the statement of the theorem.
\end{proof}
\smallskip

The proof of Theorem \ref{evgamma4} follows the same strategy and it is very similar to the proof of Theorem \ref{evgamma}. Therefore, we only sketch the proof, emphasizing the differences from the prior setting.
\begin{proof}[Sketch of proof of Theorem \ref{evgamma4}]
Step 1. Compactness. We consider the following sequence $v\e:=\brac{\mathcal{T}_{\varepsilon} u\e, \mathcal{T}_{\varepsilon} \nabla^{\varepsilon}u\e, \mathcal{T}_{\varepsilon} z\e, \mathcal{T}_{\varepsilon} \varepsilon^{\gamma} \nabla^{\varepsilon}z\e}:[0,T]\to L^2(\Omega \times \mathbb{R}^d)^d \times L^2(\Omega\times \mathbb{R}^d)^{d\times d}\times L^2(\Omega\times \mathbb{R}^d)^k\times L^2(\Omega\times \mathbb{R}^d)^{k\times d}=:H$. With help of Remark \ref{remark:1292} and Corollary \ref{cor111}, analogously as in the proof of Theorem \ref{evgamma}, we obtain that (up to a subsequence) for every $t\in [0,T]$
\begin{equation*}
y\e(t)\overset{c2}{\rightharpoonup} (q(t),\chi_1(t)), \quad \varepsilon^{\gamma}\nabla^{\varepsilon}z\e \overset{2}{\rightharpoonup} \chi_2(t) \text{ in }L^p(\Omega\times \mathbb{R}^{k\times d}),
\end{equation*}
where $q(t)\in Q$, $\chi_1(t) \in \brac{\mathbf{L}^2_{\mathsf{pot}}(\Omega)\otimes L^2(O)}^d$ and $\chi_2(t) \in \brac{\mathbf{L}^2_{\mathsf{pot}}(\Omega)\otimes L^2(O)}^k$.
\smallskip

Step 2. Stability.
We fix $t\in [0,T]$.
For an arbitrary $\tilde{q}\in Q$, similarly as in the proof of Theorem \ref{gamma2} (ii), we can find $\tilde{\chi} \in\brac{ \mathbf{L}^2_{\mathsf{pot}}(\Omega)\times L^2(O)}^d$ such that
\begin{equation*}
\mathcal{E}_{\mathsf{hom}}(t,\tilde{q})=\mathcal{E}_0(t,(\tilde{q},\tilde{\chi})).
\end{equation*}
Corollary \ref{rem5} (iii) implies that for $\brac{\tilde{U}-U(t),\tilde{\chi}-\chi_1(t)}$ there exists a sequence $v\e \in L^2(\Omega)\times L^2_0(O\cap \varepsilon\mathbb{Z}^{d})^d $ such that
\begin{equation*}
v\e \overset{2}{\to} \tilde{U}-U(t), \quad \nabla^{\varepsilon}v\e\overset{2}{\to} \nabla \tilde{U}-\nabla U(t)+\tilde{\chi}-\chi_1(t).
\end{equation*}
Furthermore, Corollary \ref{rem5} (iv) implies that for $(\tilde{Z}-Z(t), -\chi_2(t))$, there exists a sequence $g\e \in L^2(\Omega)\times L^2_0(O\cap \varepsilon\mathbb{Z}^{d})^k $ such that
\begin{equation*}
g\e \overset{2}{\to} \tilde{Z}-Z(t), \quad \varepsilon^{\gamma}\nabla^{\varepsilon}g\e\overset{2}{\to} -\chi_2(t).
\end{equation*}
We define $\tilde{y}\e$ componentwise: $\tilde{u}\e=u\e+v\e$ and $\tilde{z}\e=z\e+g\e$. Following the steps in the proof of Lemma \ref{stab111} (with the new energy $\mathcal{E}\e^{\gamma}$), we obtain
\begin{eqnarray*}
& & \lim_{\varepsilon\rightarrow 0}\cb{\mathcal{E}\e^{\gamma}(t,\widetilde{y}_{\varepsilon})+\Psi\e(\widetilde{y}_{\varepsilon}-y_{\varepsilon})-\mathcal{E}\e^{\gamma}(t,y_{\varepsilon})} \\ & = & \mathcal{E}_{0}(t,(\tilde{q},\tilde{\chi}))+\Psi_{\mathsf{hom}}(\widetilde{q}-q(t))-\mathcal{E}_{0}(t,(q(t),\chi_1(t)))\\ && -\int_{O}\expect{G(\omega)\chi_2(t)(\omega,x)\cdot \chi_2(t)(\omega,x)}dx\\
& \leq & \mathcal{E}_{\mathsf{hom}}(t,\tilde{q})+\Psi_{\mathsf{hom}}(\tilde{q}-q(t))-\mathcal{E}_{\mathsf{hom}}(t,q(t)).
\end{eqnarray*}
As a result of this, we obtain $q(t)\in S_{\mathsf{hom}}(t)$. Another important fact following from this inequality is obtained by setting $\tilde{q}=q(t)$ and using positive 1-homogeneity of $\Psi_{\mathsf{hom}}$,
\begin{equation*}
\mathcal{E}_0(t,(q(t),\chi_1(t)))+\int_{O}\expect{G(\omega)\chi_2(t)(\omega,x)\cdot \chi_2(t)(\omega,x)}dx \leq \mathcal{E}_{\mathsf{hom}}(t,q(t)).
\end{equation*}
As a result of this, we conclude that $\chi_1(t)$ is the corrector corresponding to $q(t)$, i.e., 
\begin{equation}\label{equation:1653}
\mathcal{E}_{\mathsf{hom}}(t,q(t))=\mathcal{E}_0(t,(q(t),\chi_1(t)))
\end{equation}
and, moreover, we obtain that $\chi_2=0$.
\smallskip

Step 3. Energy balance. The energy balance equality is obtained in the same manner as in the proof of Theorem \ref{evgamma} by using the assumptions on the initial data and using that
\begin{align*}
\liminf_{\varepsilon\to 0}\mathcal{E}_{\varepsilon}^{\gamma}(t,y\e(t)) & \geq \mathcal{E}_0(t,(q(t),\chi_1(t)))+\int_{O}\expect{G(\omega)\chi_2(\omega,x):\chi_2(\omega,x)}dx\\ & = \mathcal{E}_{\mathsf{hom}}(t,q(t)),
\end{align*}
which is obtained with the help of Proposition \ref{pro67} (i).
\smallskip

Step 4. Strong convergence. This part of the proof is obtained in the same way as strong convergence in the proof of Theorem \ref{evgamma}. 
We remark that the recovery sequence to be used relies on the construction from Proposition \ref{rem5} for $(U(t),\chi_1(t))$ (for the "$u\e$"-variable) and for $(Z(t),0)$ (for the "$z\e$"-variable). Also, the observation (\ref{equation:1653}) is useful in this part. 
Convergence for the whole sequence is obtained as before by a contradiction argument.
\end{proof}
\appendix
\section{Abstract evolutionary rate-independent systems}\label{appendix2}
We consider evolutionary rate-independent systems in the global energetic setting. For a detailed study, we refer the reader to \cite{mielke2005evolution,mielke2015rate}. We consider the Hilbert space case and equations involving quadratic energy functionals. The main ingredients of the theory are: 
\begin{itemize}
\item state space: $Y$ Hilbert space (dual $Y^*$);
\item external force: $l\in C^1([0,T],Y^*)$;
\item energy functional: $\mathcal{E}(t,y)=\frac{1}{2}\langle Ay,y \rangle_{Y^*,Y}-\langle l(t),y \rangle_{Y^*,Y}$, $A\in Lin(Y,Y^*)$ is self-adjoint and coercive, i.e.,  there exists $\alpha>0$ such that $\expect{Ay,y}\geq \alpha \|y\|^2$ for all $y\in Y$;
\item dissipation potential: $\Psi: Y\rightarrow [0,+\infty]$, which is convex, proper, lower-semicontinuous and positively homogeneous of order 1, i.e.,  $\Psi(\alpha y)=\alpha \Psi(y)$ for all $\alpha>0$ and $y\in Y$ and $\Psi(0)=0$.
\end{itemize} 
After a prescribed initial state $y_0 \in Y$, the system's current configuration is described by $y:(0,T]\rightarrow Y$. We say that $y\in AC([0,T],Y)$ is an energetic solution associated with $\brac{\mathcal{E},\Psi}$ if for all $t\in [0,T]$
\begin{align*}
& \mathcal{E}(t,y(t))\leq \mathcal{E}(t,\widetilde{y})+\Psi(\widetilde{y}-y(t)) \text{ for all }\widetilde{y}\in Y\quad \textit{(stability)},           \tag{S}\label{stab}\\
& \mathcal{E}(t,y(t))+\int_0^t \Psi(\dot{y}(s))ds=\mathcal{E}(0,y(0))-\int_0^t \expect{\dot{l}(s),y(s)}ds \quad \textit{(energy balance)}. \tag{EB}\label{enbal}
\end{align*} 
The stability condition is usually stated equivalently using the set of stable states:
\begin{align*}
\mathcal{S}(t):=\cb{y:\mathcal{E}(t,y)\leq \mathcal{E}(t,\widetilde{y})+\Psi(\widetilde{y}-y) \text{ for all }\widetilde{y}\in Y  }.
\end{align*}
(S) is equivalent to $y(t)\in\mathcal{S}(t)$.

For the proof of the following existence result, see \cite{mielke2005evolution}.
\begin{theorem}\label{abstract}
Let $l\in C^1([0,T],Y^*)$ and $y_0\in S(0)$. There exists a unique energetic solution $y\in C^{Lip}([0,T],Y)$ associated with $\brac{\mathcal{E},\Psi}$ with $y(0)=y_0$. Moreover,
\begin{equation*}
\|y(t)-y(s)\|_{Y}\leq \frac{Lip(l)}{C}|t-s| \quad \text{for any }t,s\in [0,T].
\end{equation*}
\end{theorem}

\section*{Acknowledgments}
We thank Alexander Mielke for useful discussions and valuable comments. SN and MV were supported by the DFG
in the context of TU Dresden's Institutional Strategy ``The Synergetic University''.

\bibliographystyle{abbrv}
\bibliography{REFAB}
\end{document}